\documentclass[10pt]{amsart}
\usepackage{a4wide}
\usepackage[utf8]{inputenc}
\usepackage{amsmath}
\usepackage{amsfonts}
 \usepackage{xcolor} 
\usepackage{amssymb}
\usepackage{amsthm}
\usepackage{subfigure}
\usepackage[mathscr]{eucal}
\usepackage{dsfont}
\usepackage{pstricks}
\usepackage{graphicx}

\usepackage[colorlinks=true, linkcolor=purple, citecolor=purple]{hyperref}
\theoremstyle{plain}
\newtheorem{thm}{Theorem}[section]
\newtheorem{theorem}[thm]{Theorem}

\newtheorem{lemma}[thm]{Lemma}
\newtheorem{cor}[thm]{Corollary}
\newtheorem{prop}[thm]{Proposition}
\newtheorem{statement}[thm]{Statement}
\newtheorem{assumption}[thm]{Assumption}

\newtheorem{conj}[thm]{Conjecture}

\theoremstyle{remark}

\newtheorem{remark}[thm]{Remark}
\newtheorem*{remark*}{Remark}

 \usepackage{graphicx}

\numberwithin{equation}{section}

\newcommand{\vect}{\operatorname{span}}
\newcommand{\supp}{\operatorname{supp}}
\newcommand{\Sp}{\operatorname{Sp}}

\newcommand{\app}{\tiny{\operatorname{app}}}

\newcommand{\RR}{\mathbb R}
\newcommand{\R}{\mathbb R}
\newcommand{\N}{\mathbb N}
\newcommand{\C}{\mathbb C}

\let \d \relax
\newcommand{\d}{\partial}
\newcommand{\eps}{\varepsilon}
\let \Re \relax
\DeclareMathOperator{\Re}{Re}

\newcommand{\HM}{\mathfrak H_\mathcal{M}(\alpha)}
\newcommand{\ar}{\rightarrow}

\DeclareMathOperator{\Op}{Op}

\allowdisplaybreaks

\title[Critical points of eigenvalues of Montgomery operators]{On critical points of eigenvalues of the Montgomery family of quartic oscillators}

\author{Bernard Helffer and Matthieu L\'eautaud }
\date{\today}

\begin{document}
\maketitle

\begin{abstract}
We  discuss spectral properties of the family of quartic oscillators $\mathfrak h _{\mathcal M}(\alpha) =-\frac{d^2}{dt^2} +
\Big(\frac{1}{2} t^{2} -\alpha\Big)^2$ on the real line, where $\alpha\in \R$ is a parameter. This operator appears in a variety of applications coming from quantum mechanics to harmonic analysis on Lie groups, Riemannian geometry and superconductivity.
We study the variations of the eigenvalues $\lambda_j(\alpha)$ of $\mathfrak h _{\mathcal M}(\alpha)$ as functions of the parameter $\alpha$.
We prove that for $j$ sufficiently large, $\alpha \mapsto \lambda_j(\alpha)$ has a unique critical point, which is a nondegenerate  minimum.
We also prove that the first eigenvalue $\lambda_1(\alpha)$ enjoys the same property and give a numerically assisted proof that the same holds for the second eigenvalue $\lambda_2(\alpha)$.
The proof for excited states relies on a semiclassical reformulation of the problem. In particular, we develop a method permitting to differentiate with respect to the semiclassical parameter, which may be of independent interest.
\end{abstract}

\bigskip
{\small
\noindent {\sl Keywords:}  Spectral theory, 1D Schr\"odinger operator, quartic potential, eigenvalues, eigenfunctions, semiclassical analysis.

\noindent {\sl AMS Subject Classification (2020):} 
  34L40, 
35P15, 
35P20, 
 81Q05, 
 81Q10, 
 81Q20. 
}

\tableofcontents

\section{Introduction and main  results}
We  consider  the spectral properties of the family  (indexed by $\alpha\in \RR$) of  self-adjoint realizations of the second order differential operator 
$$
\mathfrak h _{\mathcal M}(\alpha):=-\frac{d^2}{dt^2} +
\left(\frac{1}{2} t^{2} -\alpha \right)^2\,,
$$
on the real line.
 This operator appears in many different settings, namely:
 \begin{itemize}
\item in quantum mechanics, see Simon \cite{Sim} (see also~\cite{ReSi});
\item  in the study of irreducible representations of certain nilpotent Lie groups of rank 3: see~\cite{HNdim3,Pham-Robert,He79,He80} (and the references therein) for considerations on analytic hypoellipticity\footnote{The existence of $\alpha \in \mathbb C$ such that $\mathfrak h _{\mathcal M}(\alpha)$ is not injective \cite{Pham-Robert} leads to the proof of the non hypoanalyticity of the Sublaplacian on the Engel group \cite{He79,He80}.} of hypoelliptic operators, see also~\cite{Chatzakou} and~\cite{BBGL} for a recent harmonic analysis on the so-called Engel group, and~\cite{CdVLe} for the analysis of a related sublaplacian;
\item in Riemannian Geometry and the study of Schr\"odinger operators with magnetic fields on compact manifolds and in superconductivity, see~\cite{Mont} (hence the name of Montgomery attributed to this family), and~\cite{AlmHel,FH,HK2008a, HK2008b, HM,HMdim3,Pan,Pan-Kwek,Qmath10}.
 \end{itemize}
 Our results concern the eigenvalues of the operator $\mathfrak h _{\mathcal M}(\alpha)$, acting on 
 $$
 D(\mathfrak h _{\mathcal M}(\alpha)) := \left\{ u \in L^2(\R), -u'' +
\left(\frac{1}{2} t^{2} -\alpha \right)^2u \in L^2(\R) \right\}  \subset L^2(\R) .
 $$
We first recall a list of elementary spectral properties of the operator $\mathfrak h _{\mathcal M}(\alpha)$ (see e.g.~\cite{BBGL} for a detailed proof).  
 \begin{prop}
\label{p:def-lambda}
 For any $\alpha \in\R$, the operator $(\mathfrak h _{\mathcal M}(\alpha), D(\mathfrak h _{\mathcal M}(\alpha)))$ is selfadjoint on $L^2(\R)$, with compact resolvent. Its spectrum consists in countably many real eigenvalues with finite multiplicities, accumulating only at $+\infty$. Moreover, we have the following properties:
\begin{enumerate}
\item The domain $D(\mathfrak h _{\mathcal M}(\alpha))=D(\mathfrak h _{\mathcal M}(0))$ does not depend on $\alpha$\,.
\item \label{eigenvalue-ordering} All eigenvalues are simple and positive, and we may thus write $\Sp(\mathfrak h _{\mathcal M}(\alpha)) = \{\lambda_j(\alpha),j\in \N^*\}$ with 
\begin{align*}
& 0<   \lambda_1(\alpha)< \cdots < \lambda_j(\alpha)< \lambda_{j+1}(\alpha) \to + \infty \,.\\
& \dim \ker (\mathfrak h _{\mathcal M}(\alpha)-\lambda_j(\alpha)) = 1\,.
\end{align*}
\item \label{regularity} All eigenfunctions are real-analytic and decay exponentially fast at infinity (as well as all their derivatives).
\item \label{parity} For all $j\in \N^*$, functions in $\ker (\mathfrak h _{\mathcal M}(\alpha)-\lambda_j(\alpha))$ have the parity of $j+1$.
\end{enumerate}
\end{prop}
For all $j\in \N^*$, we shall denote by $u_j(\alpha)$ any $L^2(\R)$ normalized real-valued function in $\ker (\mathfrak h _{\mathcal M}(\alpha)-\lambda_j(\alpha))$.
According to Item~(\ref{eigenvalue-ordering}), any such family $\big(u_j(\alpha)\big)_{j\in \N}$ forms a Hilbert basis of $L^2(\R)$.

Some properties of the first eigenvalue of $\mathfrak h _{\mathcal M}(\alpha)$ have been studied in~\cite{He2008}.

The present paper investigates the dependence in $\alpha$ of the
eigenvalues $\lambda_j(\alpha)$, and in particular the nature of its
critical points.
This study is motivated by a question of the second author with H.
Bahouri, I. Gallagher and D. Barilari (see~\cite{BBGL}) related to the
Schr\"odinger evolution equation on the Engel group.
In that context, the study of dispersive properties involves that of
oscillatory integrals whose phases contain $\lambda_j(\alpha)$.
It is thus  likely that dispersive estimates depend on the critical points
of the maps $\alpha \mapsto \lambda_j(\alpha)$ for $j\in \N^*$. 
 
 \medskip
 As far a the variations of $\lambda_j$ are concerned, classical arguments show that $\lambda_j (\alpha)\to +\infty$ as $\alpha \to \pm \infty$ and $\lambda'(\alpha)<0$ for $\alpha\leq 0$ (see Section~\ref{s2}). Henceforth, $\lambda_j$ has at least one critical point, namely a global minimum.
 The graph of the first six eigenvalues $\lambda_j(\alpha)$ is plotted on Figure~\ref{f:graphe-6-vp}; see also Section~\ref{s:numerics} for comments on this figure and additional numerical results.
\begin{figure}[h]
\begin{center}
\includegraphics[width=15pc, height=15pc]{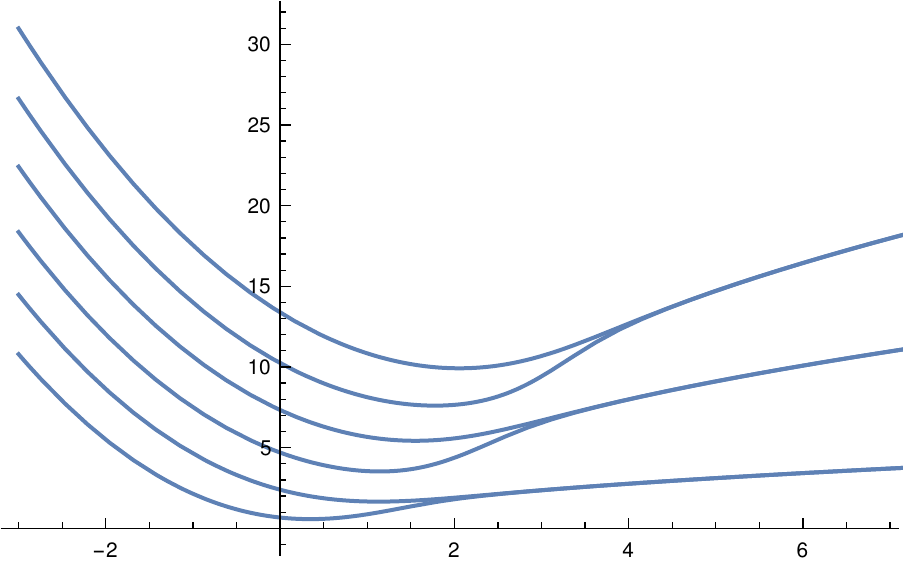}
\end{center}
\caption{Graph of $\alpha \mapsto \lambda_j(\alpha)$ for $j=1,\cdots, 6$.} \label{f:graphe-6-vp}
\end{figure}

Our first main result concerns the first eigenvalue $\lambda_1$.
 \begin{theorem} 
\label{t:lambda1}
The first eigenvalue $\alpha \mapsto \lambda_1(\alpha)$ of $\mathfrak h _{\mathcal M}(\alpha)$ has  a unique critical point $\alpha_{c}$. Moreover $\alpha_c \in (0,1)$, $\lambda_1$ reaches its minimum at $\alpha_c$, i.e. $\lambda_1(\alpha_c)= \min_{\alpha \in \R} \lambda_1(\alpha)$, and this minimum is non-degenerate, i.e. $\lambda_1''(\alpha_c)>0$. 
\end{theorem}
Theorem~\ref{t:lambda1} was conjectured by Montgomery when replacing ``critical point'' by ``minimum''~\cite{Mont}, partially analyzed in~\cite{HM}, stated in Pan-Kwek~\cite{Pan-Kwek}, 
 and verified with a computer assisted\footnote{By numerical
   computations of V. Bonnaillie-No\"el} but  not completely mathematically 
rigorous proof  in~\cite{HK2008a, HK2008b}. One can also find 
 in \cite{FH} a presentation of some of the  results. In the case of the minimum, a complete proof was finally given in \cite{He2008}.  Improvements in the proof and a generalization to a larger class appeared in \cite{HPS}.

Our second main result is a numerically assisted proof of the same property for $\lambda_2$.
 \begin{statement}[Numerically assisted]\label{St02} 
$\lambda_2$ has  a unique critical point $\alpha_{2,c}$. Moreover this critical point is positive, corresponds to  a minimum,
 and is non degenerate. 
\end{statement}

Next, we investigate the case of ``large eigenvalues'', i.e. of $\lambda_j(\alpha)$ for sufficiently large values of $j$. We first provide with a rough asymptotic localization of critical points of $\lambda_j$ for large $j$.
We set 
 $$
 A_{j,c} =\{ \alpha \in \mathbb R\,,\, \lambda'_j(\alpha)=0\} , \quad \text{ for } j \in \N^* .
 $$
 \begin{theorem}\label{th1.4} 
 The function  \begin{equation}
 \label{e:def-F-00}
 (1,+\infty) \ni E  \mapsto F(E)= \int_{0}^{\sqrt{2+2 \sqrt{E}}}\frac{2- x^2}{\sqrt{E-(x^2/2-1)^2}} dx 
\end{equation}
is of class $C^1$ and admits a unique zero $E_c$ in the interval $(1,+\infty)$.
For any choice of $\alpha_{j,c}$ in $A_{j,c}$, we have $ \lim_{j\ar +\infty}\alpha_{j,c} = +\infty$ and
 \begin{equation} \label{lim1}
 \lim_{j\ar +\infty} \lambda_j(\alpha_{j,c})/\alpha_{j,c}^2 = E_c\,.
 \end{equation}
 \end{theorem}
A numerical investigation shows that $E_c\approx 2.35$, with a relatively fast convergence with respect to $j$, see Section~\ref{s:numerics} for more on numerics.

Finally, our last main result proves the analogue of Theorem~\ref{t:lambda1} for large eigenvalues.
 \begin{theorem}\label{thcj}
 There exists $j_0$ such that for $j\geq j_0$ , $\lambda_j$ admits a unique critical point $\alpha_{j,c}$ corresponding to a non degenerate minimum. Moreover, we have 
 \begin{align}
 \label{e:deriv-infty-1}
 \lambda_j''(\alpha_{j,c}) \to - 3 F'(E_c) G(E_c)>0, \quad \text{ as } j \to + \infty ,
 \end{align}
 where $F'(E_c)<0$, and, setting $V(x) = \left(\frac{x^2}{2}-1\right)^2$ and $x_+(E) =\sqrt{2+2 \sqrt{E}}$,
 \begin{align}
  \label{e:deriv-infty-2}
  G(E) = \Big( \int_{0}^{x_+(E)} \left(E-V(x)\right)^{\frac12}dx \Big)\Big(\int_{0}^{x_+(E)} \left(E-V(x)\right)^{-\frac12}dx\Big)^{-2} .
 \end{align}
 \end{theorem}

Theorems~\ref{t:lambda1} and~\ref{thcj} together with Statement~\ref{St02} and Figure~\ref{f:graphe-6-vp} naturally lead us to the following conjecture.
\begin{conj}\label{conj1}
For any $j\geq 1$, $\lambda_j$ has  a unique critical point $\alpha_{j,c}$. Moreover this critical point is positive, corresponds to  a minimum,
 and is non degenerate. 
\end{conj}
 Note that $A_{j,c}\neq \emptyset$ because we will prove in \eqref{uniqmin} that for any $j$, $\lambda_j(\alpha)$ has at least a minimum. Conjecture \ref{conj1} says in particular that $A_{j,c}$ has precisely one element for all $j \in \N^*$.
  Remark also that, on account to Theorem~\ref{thcj}, there is only a finite number of eigenvalues which remain uncovered.

\bigskip
As in~\cite[Section~2]{BBGL}, there are different regions of the space of $(j,\alpha) \in \N \times \R$ to be considered separately. Here, it is rather straightforward to see that $\lambda_j$ has no critical point in $\{\alpha \leq 0\}$ (see Corollary~\ref{c:decreasing-neg}) so there are only two relevant regimes to consider.

The first regime is $\lambda_j(\alpha)\gg \alpha^2$, in which the operator $\mathfrak h _{\mathcal M}(\alpha)$ may be rescaled to (a small perturbation of) the operator $P^\eps= -\eps^2 \frac{d^2}{ds^2} + \frac{s^4}{4}$. In this regime, we prove  in Section~\ref{s:regime-1} that there is no critical point asymptotically. The proof relies on a description of semiclassical measures associated to eigenfunctions of the operator $P^\eps$ following~\cite{LL:22}.

The second regime is $\lambda_j(\alpha)\lesssim \alpha^2$, in which the operator $\mathfrak h _{\mathcal M}(\alpha)$ may be rescaled to the double-well semiclassical Schr\"odinger operator $P_h=-h^2 \frac{d^2}{dx^2} + \left(\frac{x^2}{2}-1\right)^2$. It appears that all critical points belong asymptotically to this regime and a thorough semiclassical analysis is needed to first localize critical points, and then prove their non-degeneracy. This is the heart of the proofs of Theorems~\ref{th1.4} and~\ref{thcj} achieved respectively in Sections~\ref{s:semiclassique} and~\ref{s5}.
The proof of Theorem~\ref{th1.4} relies on the description of semiclassical measures associated to eigenfunctions of the operator $P_h$  at all possible energy levels $E\in \R_+$ (and the semiclassical reformulation of the first variation formula for eigenvalues). This allows to provide in particular with an asymptotics of $\lambda_j'(\alpha)$ when $\frac{\lambda_j(\alpha)}{\alpha^2}\to E$ (see Corollary~\ref{cor3.4} and Proposition~\ref{p:bottom}). We then prove that the latter asymptotics vanishes if and only if $E=E_c$.

The description of semiclassical measures is unfortunately not sufficient for the analysis of  the second derivative $\lambda_j''(\alpha)$, and thus for  the proof of  Theorem~\ref{thcj}. Luckily enough, the energy $E_c$ is a regular energy for the potential $\left(\frac{x^2}{2}-1\right)^2$ since $E_c>1$, and we can push the semiclassical analysis beyond semiclassical measures at the energy $E_c$. To this aim, we use an approximation of eigenfunctions by semiclassical Lagrangian distributions and prove that derivatives of the approximate and the true eigenfunctions are close one to the other in the semiclassical limit. 
We refer to the beginning of Section~\ref{s5} for a detailed description of the proof of Theorem~\ref{thcj}.

Finally, the proof of Theorem~\ref{t:lambda1} is  completed in Section~\ref{s:first-eig}. It relies on several refinements of the proof in~\cite{He2008}. In particular, we first obtain a refined estimate for the localization of possible critical points of $\lambda_1$. Then, we prove an upper bound for $\lambda_1(\alpha)$ and lower bound for $\lambda_3(\alpha)$ which yield a positive lower bound for $\lambda_1''$ in this interval. The estimates on $\lambda_1$ and $\lambda_3$ follow by comparing the operator $\mathfrak h _{\mathcal M}(\alpha)$ with various harmonic oscillators.
Numerical experiments performed by M.~Persson Sundqvist are presented in Subsection~\ref{s:numerics} and allow to both check the accuracy of the results we obtain and prove Statement~\ref{St02}.

\bigskip 
 {\bf  Acknowledgements.}
 Many thanks to F. Nicoleau and M. Persson Sundqvist for their precious help for numerical computations with Matematica. We also thank D. Robert for useful discussions on the possible use of coherent states in semiclassical analysis and Y. Colin de Verdi\`ere for mentioning to us the results of \cite{CdV}. The second author is partially supported by the Agence Nationale de la Recherche under grants SALVE (ANR-19-CE40-0004) and ADYCT (ANR-20-CE40-0017).

\section{Preliminaries}\label{s2}
\subsection{Classical results}\label{ss2.1}
We first collect various known basic results, see~\cite{Mont, HM, Pan-Kwek, HK2008a,BBGL}. Firstly, 
\begin{equation}
\label{e:infty-a-infty}
\text{for all } j \in \N^*, \mbox{ the map }\alpha \mapsto \lambda_j(\alpha) \mbox{ is continuous
 and satisfies}
\lim_{|\alpha|\ar +\infty} \lambda_j(\alpha) =+\infty\,.
\end{equation}
When $\alpha \rightarrow -\infty$, this is a consequence of the direct bound $\lambda_1(\alpha)\geq \alpha^2$ if $\alpha<0$. 
When $\alpha \to + \infty$, this results from standard semiclassical analysis (see Section~\ref{s:semiclassique} below) relative to the operator $P_h$, yielding, for all $k \in \N$,
   \begin{align}\label{eq:asba}
&  \lambda_{2k+1} (\alpha) \sim  \sqrt{2}  ( 2k+1)  \sqrt{ \alpha}\,,\quad \mbox{ as } \alpha \ar +\infty\, ,  \\
\label{eq:asbb}
&   \lambda_{2k+2}(\alpha) -  \lambda_{2k+1}(\alpha) = \mathcal O (\alpha^{-\infty}), \quad  \mbox{ as } \alpha \ar +\infty\,.
   \end{align}
The $\mathcal O (\alpha^{-\infty})$ decay (actually, this is an exponential decay, but this will not be used here) of $\lambda_{2k+2}(\alpha) -  \lambda_{2k+1}(\alpha)$ is explained by a double well analysis of the problem (see \cite{HeSj}, \cite{He88} or \cite{DiSj}).
As a consequence, for any $j \in \N^*$, 
\begin{equation}\label{uniqmin}
\alpha \mapsto \lambda_j(\alpha) \mbox{ admits at
  least one minimum }\,.
\end{equation}
We recall the first variation formula for the eigenvalues $\lambda_j$.
\begin{prop}[Feynmann-Hellmann formula]
The map $\alpha \mapsto \lambda_j(\alpha)$ is $C^\infty$, the map $\alpha \mapsto u_j(\cdot, \alpha), \R \to D(\mathfrak h _{\mathcal M}(0))$ is of class $C^1$ and 
\begin{align}
&\left( \d_\alpha u_j(\cdot , \alpha) , u_j(\cdot , \alpha) \right)_{L^2(\R)} = 0,  \nonumber\\
\label{FeHei}
&\lambda'_j(\alpha) = - 2 \int_\R \left(\frac 12 t^2 -\alpha\right) u_j(t,\alpha)^2\, dt =  -4 \int_0^{+\infty} \left(\frac 12 t^2 -\alpha\right)\;
u_j(t,\alpha)^2\, dt\;.
\end{align}
\end{prop}
 Note that the second equality in~\eqref{FeHei} simply follows from the fact, stated in Proposition~\ref{p:def-lambda}, that eigenfunctions are either odd or even, which implies  $t\mapsto \left(\frac 12 t^2 -\alpha\right)
u_j(t,\alpha)^2$ is even. 

In particular, this implies~:
\begin{cor}
\label{c:decreasing-neg} For all $j \in \N^*$, the function $\alpha\mapsto \lambda_j(\alpha)$ is strictly decreasing
 for $\alpha \leq 0$ and all  critical points  of $\lambda_j$ are in $]0,+\infty[$. 
In addition, $\lambda_j$ admits a strictly positive 
 minimum which can only be attained for positive $\alpha$'s.
\end{cor}
\begin{proof}
Recalling that $u_j(\cdot,\alpha)$ is normalized in $L^2(\R)$, Equation~\eqref{FeHei} rewrites as 
$$
\lambda'_j(\alpha) = -  \|tu_j(\cdot,\alpha) \|_{L^2(\R)}^2 +2 \alpha ,
$$
As a consequence, $\lambda'_j(\alpha) <0$ for all $\alpha < 0$ and $\lambda'_j(0) = -  \|tu_j(\cdot,0) \|_{L^2(\R)}^2 <0$ since $u_j(\cdot,\alpha)$ is normalized in $L^2(\R)$.
\end{proof}

Observing that the first eigenfunction is even, we deduce classically the following lemma (see for example \cite{Pan-Kwek} for a proof).
\begin{lemma}
For all $k \in \N^*, \alpha \in \R$, we have 
$$
\lambda_{2k-1}(\alpha)= \lambda_k^{N}(\alpha) , \quad \lambda_{2k}(\alpha)= \lambda_k^{D}(\alpha) ,
$$
where $\lambda_k^{N}(\alpha)$ (resp. $\lambda_k^{D}(\alpha)$) denotes the $k-$th eigenvalue of the Neumann (resp. Dirichlet) realization  of the differential operator 
 $-\frac{d^2}{dt^2} 
 +(\frac 12 t^2 -\alpha)^2$
 in $\mathbb R^+$.
\end{lemma}
In particular, we shall use 
\begin{equation}
\lambda_1(\alpha)= \lambda_1^{N}(\alpha), \quad  \lambda_3(\alpha)= \lambda_2^{N}(\alpha)\;.
\end{equation}

\subsection{Dilation operators}
In the present article, we shall make an intensive use of dilation operators.
 We define for $\eta >0$ the following unitary (dilation) operator 
 \begin{equation}
 \label{e:Tnu}
\begin{array}{rcl}
T_\eta : L^2(\R)& \to & L^2(\R) , \\
u(x)& \mapsto &\eta^{\frac{1}{2}}u(\eta x) ,
\end{array}
\end{equation}
having adjoint/inverse $T_\eta^*=T_\eta^{-1}=T_{\eta^{-1}}$.

\subsection{Additional properties at a critical point}\label{ss2.2}
Of course, multiplying the eigenvalue equation 
\begin{align}
\label{e:eingenvalue-eqn}
\mathfrak h _{\mathcal M}(\alpha)u (\cdot,\alpha) = \lambda(\alpha) u(\cdot,\alpha) , \quad \| u(\cdot,\alpha)\|_{L^2(\R)}=1 ,
\end{align}
by $u(\cdot,\alpha)$ and integrating by parts, we obtain 
\begin{align}
\label{e:rough-energy}
\lambda(\alpha) = \|\d_tu(\cdot,\alpha)\|_{L^2(\R)}^2 + \left\| \left(\frac{t^2}2-\alpha \right) u(\cdot,\alpha)\right\|_{L^2(\R)}^2 . 
\end{align}
In particular,  we always have  
\begin{align}
\label{e:rough-rough-eig}
\left\| \left(\frac 12 t^{2} -\alpha \right) u(\cdot,\alpha)\right\|_{L^2(\R)}^2 \leq \lambda(\alpha)\,.
\end{align}
\begin{lemma}
Assume that $(\lambda(\alpha), u(\cdot,\alpha))$ satisfy the eigenvalue equation~\eqref{e:eingenvalue-eqn}. Then we have
\begin{align}
\label{e:variation-eta}
\alpha \lambda'(\alpha)  + \lambda(\alpha) = 3 \left\| \left(\frac 12 t^{2} -\alpha \right) u(\cdot,\alpha)\right\|_{L^2(\R)}^2 .
\end{align}
\end{lemma}
As a consequence of this lemma, one can improve a little the estimate~\eqref{e:rough-rough-eig} when $\alpha=\alpha_c$ is a critical point of $\lambda$ (this is already stated in~\cite[Proposition 3.5 and (3.14)]{Pan-Kwek}), namely 
\begin{equation} \label{eq:newid2}
3 \| (\frac 12 t^{2} -\alpha_{c})
 u(\cdot,\alpha_{c})\|^2 = \lambda(\alpha_c)\,.
\end{equation}
The proof uses invariance by dilation $T_\eta$, defined in~\eqref{e:Tnu}, of the spectrum.
\begin{proof}
Since $T_\eta$ defined in~\eqref{e:Tnu} is unitary, we deduce from~\eqref{e:eingenvalue-eqn} that, for all $\eta>0$ and $\alpha \in \R$
$$
T_\eta h _{\mathcal M}(\alpha) T_\eta^{-1}  \ T_\eta u(\cdot,\alpha) =  \lambda (\alpha) T_\eta u(\cdot,\alpha) , \quad \|T_\eta u(\cdot,\alpha)\|_{L^2(\R)}=1  ,
$$
where the conjugated operator is given by 
$$
T_\eta h _{\mathcal M}(\alpha) T_\eta^{-1}  =  - \frac{1}{\eta^2} \frac{d^2}{dt^2} + \left(\frac 12 \eta^{2} t^{2} -\alpha \right)^2 .
$$
That is to say, the spectrum is invariant and the eigenfunctions are only dilated. 
Setting $u(t,\eta,\alpha):=T_\eta u(t,\alpha)$, the above equation rewrites 
$$
 - \frac{1}{\eta^2} \d_t^2 u(t,\eta,\alpha) + \left(\frac 12 \eta^{2} t^{2} -\alpha \right)^2u(t,\eta,\alpha) = \lambda (\alpha)u(t,\eta,\alpha)  , \quad  \|u(\cdot,\eta, \alpha)\|_{L^2(\R)}=1 .
$$
Differentiating this expression with respect to $\eta$, we obtain
$$
2 \eta^{-3} \d_t^2 u  + 2\eta t^2\left(\frac 12 \eta^{2} t^{2} -\alpha \right) u +  \mathfrak h _{\mathcal M}(\alpha) \d_\eta u  = \lambda \d_\eta u ,
$$
together with 
\begin{align}
\label{e:ortho-ortho}
\left( u(\cdot,\eta, \alpha) , \d_\eta u(\cdot,\eta, \alpha) \right)_{L^2(\R)} = 0.
\end{align}
Taking now the scalar product of the previous identity with $u(\cdot,\eta,\alpha)$ and using~\eqref{e:ortho-ortho}, we deduce
$$
- 2 \eta^{-3} \| \d_t u \|^2  + 2\eta \int_\R t^2\left(\frac 12 \eta^{2} t^{2} -\alpha \right) u^2(t,\eta,\alpha) dt  + (\mathfrak h _{\mathcal M}(\alpha) \d_\eta u  , u)_{L^2(\R)} = 0 .
$$
Selfadjointness of $\mathfrak h _{\mathcal M}(\alpha)$ together with~\eqref{e:ortho-ortho} imply  
$$
 (\mathfrak h _{\mathcal M}(\alpha) \d_\eta u  , u)_{L^2(\R)} =  ( \d_\eta u  , \mathfrak h _{\mathcal M}(\alpha)u)_{L^2(\R)}  =\lambda(\alpha)  ( \d_\eta u  , u)_{L^2(\R)} = 0,
 $$
and we finally deduce
$$
- \eta^{-3} \| \d_t u \|^2  + \eta \int_\R t^2\left(\frac 12 \eta^{2} t^{2} -\alpha \right) u^2(t,\eta,\alpha) dt  = 0 .
$$
Fixing now $\eta=1$, this rewrites as
$$
-  \| \d_t u \|^2  +  2\int_\R \left(\frac 12 t^{2} -\alpha  +\alpha\right)\left(\frac 12 t^{2} -\alpha \right) u^2(t,\eta,\alpha) dt  = 0 ,
$$
and, recalling the first variation formula~\eqref{FeHei}, this implies
\begin{align*}
0 & = -  \| \d_t u \|^2  +  2\left\| \left(\frac 12 t^{2} -\alpha \right) u(\cdot,\alpha)\right\|_{L^2(\R)}^2 -  \alpha\lambda'(\alpha) . 
\end{align*}
Combined with the energy identity~\eqref{e:rough-energy}, this implies~\eqref{e:variation-eta} and concludes the proof of the lemma.
\end{proof}
\subsection{An integration by part formula}\label{ss2.3}
The following lemma will be useful in the proof of Theorem~\ref{t:lambda1} and (as a technical device) Theorem~\ref{thcj}.
Although initially due to Pan-Kwek in \cite{Pan-Kwek} for $\lambda_1$ (see also \cite{DauHel} and \cite{He2008}) the most elegant proof is given in \cite{HPS}. We recall it quickly.
\begin{lemma}
Assume that $(\lambda(\alpha), u(\cdot,\alpha))$ satisfy the eigenvalue equation~\eqref{e:eingenvalue-eqn}. Then we have
\begin{equation}
\label{eq:ipf2} 
 2 \int_0^{+\infty} (t- \sqrt{2\alpha}) (t^2/ 2 - \alpha)  u^2 (t,\alpha)\, dt - \frac{\sqrt{\alpha}}{\sqrt{2}}  \lambda'(\alpha)
 =  (\lambda(\alpha) - \alpha^2) u (0,\alpha)^2 + u'(0,\alpha) ^2\,.
\end{equation}
\end{lemma}
\begin{proof}
Using integration by parts together with $u(\cdot,\alpha) \in \mathcal{S}(\R)$,  according to Proposition~\ref{p:def-lambda}, we have 
\begin{align*}
\int_0^{+\infty} \frac{d}{dt} \left((\frac 12 t^2-\alpha)^2\right) u^2 (t,\alpha)\, dt = - \alpha^2 u^2 (0,\alpha) -  2 \int_0^{+\infty}  (\frac 12 t^2-\alpha)^2 u  (t,\alpha) u'(t,\alpha)^2\, dt .
\end{align*}
The  eigenvalue equation~\eqref{e:eingenvalue-eqn} then implies 
\begin{align}\label{eq:ipf}
& \int_0^{+\infty} \frac{d}{dt} \left((\frac 12 t^2-\alpha)^2\right) u^2 (t,\alpha)\, dt  \nonumber\\
& \quad  = - \alpha^2 u^2 (0,\alpha)  -2 \int_0^{+\infty}  \lambda(\alpha) u  (t,\alpha) u'(t,\alpha)\, dt - 2 \int_0^{+\infty}   u''  (t,\alpha) u'(t,\alpha)\, dt  \nonumber \\
& \quad = (\lambda(\alpha) - \alpha^2) u (0,\alpha)^2 + u'(0,\alpha) ^2\,.
\end{align}
We now compute the left hand-side as
\begin{align*} 
 & \int_0^{+\infty} \frac{d}{dt} \left((\frac 12 t^2-\alpha)^2\right) u^2 (t,\alpha)\, dt 
  =  \int_0^{+\infty}  2t \left(\frac 12 t^2-\alpha \right) u^2 (t,\alpha)\, dt \\
 & \quad  =  \int_0^{+\infty}  2(t- \sqrt{2\alpha}  )\left(\frac 12 t^2-\alpha \right) u^2 (t,\alpha)\, dt  + 2 \sqrt{2\alpha} \int_0^{+\infty} \left(\frac 12 t^2-\alpha \right) u^2 (t,\alpha)\, dt  
 \end{align*}
Using the Feynman-Hellmann formula~\eqref{FeHei}, this rewrites
\begin{align*} 
  \int_0^{+\infty} \frac{d}{dt} \left((\frac 12 t^2-\alpha)^2\right) u^2 (t,\alpha)\, dt 
   = 2 \int_0^{+\infty}  (t- \sqrt{2\alpha}  )\left(\frac 12 t^2-\alpha \right) u^2 (t,\alpha)\, dt  - \frac{\sqrt{\alpha}}{\sqrt{2}} \lambda'(\alpha).
 \end{align*}
Combined with~\eqref{eq:ipf}, this proves~\eqref{eq:ipf2} and concludes the proof of the lemma. 
\end{proof}

As an immediate first  application we have the following proposition.
\begin{prop}\label{lem:abound}
Assume that $j$ is odd and that $\alpha_{j,c}$ is a critical point of 
$\lambda_j(\alpha)$. Then
\begin{equation}\label{eq:abound}
\alpha_{j,c}^2  < \lambda_j (\alpha_{j,c})\,.
\end{equation}
\end{prop}
\begin{proof}
When $j$ is odd, we recall from Proposition~\ref{p:def-lambda} that $u_j(\cdot, \alpha)$ is even, so that $u_j'(\cdot, \alpha)=0$. We deduce from the universal formula~\eqref{eq:ipf2} that if in addition $\lambda_j'(\alpha_{j,c})=0$, then
\begin{align*}
 \big(\lambda_j(\alpha_{j,c}) - \alpha_{j,c}^2\big) u_j (0,\alpha_{j,c})^2& = 2\int_0^{+\infty} ( t - \sqrt{2 \alpha_{j,c}})  \left(\frac 12 t^2-\alpha_{j,c}\right) u_j^2 (t,\alpha_{j,c})\, dt \quad , j \mbox{ odd }\\
&  =    \int_0^{+\infty} ( t + \sqrt{2 \alpha_{j,c}})( t - \sqrt{2 \alpha_{j,c}})^2  u_j^2 (t,\alpha_{j,c})\, dt .
\end{align*}
Observing that the right hand-side is positive ($u_j$ does not vanish identically on an open set, otherwise would vanish identically on $\R$) concludes the proof of the proposition.
\end{proof}
 
 \subsection{The second derivative}
For the analysis of the second derivative, we also need the following formula established in \cite{HK2008a} and probably much older. As in the proof of the
Feynman-Hellmann
 formula, we start from~:
\begin{equation}\label{e:Qlambda}
\left(\mathfrak h_{\mathcal M}(\alpha) -\lambda_j(\alpha)\right)\frac{\partial
u_j(\cdot, \alpha)}{\partial\alpha}= \left[2 \left(\frac{t^{2}}{2}-
\alpha \right) +\lambda_j'(\alpha)\right] u_j(\cdot,\alpha)\,.
\end{equation}
Differentiating once more and taking the scalar product with $u_j$ we obtain
\begin{equation}\label{eq:2.15}
\lambda_j''(\alpha) = 2  - 4  \int (\frac 12 t^{2} -\alpha)
u_j(t,\alpha)\,\partial_\alpha u_j(t,\alpha)   \,dt\,,
\end{equation}
where we have used that $u_j$ is normalized. This implies
\begin{equation}\label{eq:2.15aa}
\lambda_j''(\alpha) = 2  - 2   \int  t^{2} 
u_j(t,\alpha)\,\partial_\alpha u_j(t,\alpha)   \,dt =  2  -  \frac{d}{d\alpha} \big(\int t^2 u_j(t,\alpha) ^2  \,dt\big) \,.
\end{equation}
Note finally, that we can rewrite the last inequality in a form where the normalization condition is not necessarily assumed
\begin{equation}\label{eq:2.15b}
\lambda_j''(\alpha) = 2  - 2   \int  t^{2} 
u_j(t,\alpha)\,\partial_\alpha u_j(t,\alpha)   \,dt =  2  - \frac{d}{d\alpha} \Big( \big(\int t^2 u_j(t,\alpha) ^2  \,dt\big)  \big(\int  u_j(t,\alpha) ^2 dt \big)^{-1}\Big)  \,.
\end{equation}

\subsection{No critical point in the regime $\lambda_j(\alpha_j)\gg \alpha_j^2$}
\label{s:regime-1}
We prove in this preliminary section that there is no critical point in the regime  $\frac{\alpha_j^2}{\lambda_j(\alpha_j)} \to 0$. The latter are thus necessarily located in the regime $\lambda_j(\alpha_j)\lesssim \alpha_j^2$, which we investigate in the next sections. 
\begin{prop}
\label{l:regime-Einfty}
Let $(\alpha_j)_{j \in \N}$ a real sequence, and assume $\lambda_j(\alpha_j) \to + \infty$ and $\frac{\alpha_j^2}{\lambda_j(\alpha_j)} \to 0$, as $j \to +\infty$. Then we have 
\begin{equation}
\label{e:regime-interm}
 \frac{ \lambda_j' (\alpha_j)}{\alpha_j \lambda_j (\alpha_j)^{1/2}} \to  - K_1 \int_{-\sqrt{2}}^{\sqrt{2}}  \frac{s^2}{\sqrt{1-s^4/4}} ds  , \quad K_1  = \left( \int_{-\sqrt{2}}^{\sqrt{2}} \frac{ds}{\sqrt{1-s^4/4}} \right)^{-1} ,
\end{equation}
as $j \to +\infty$.
\end{prop}
Note that in this regime, $\alpha$ does not necessarily converge to $+ \infty$. Note also that $\min_{\alpha \in \R}\lambda_j(\alpha) \to +\infty$ as $j \to +\infty$, so that the assumption $\lambda_j(\alpha_j) \to + \infty$ is implied by $j \to+\infty$.

For the next result, recall that $A_{j,c}$ is the set of critical points of $\alpha \mapsto \lambda_j(\alpha)$.

\begin{cor}\label{lemma3.5}
 For any choice of $\alpha_{j,c}$ in $A_{j,c}$ we have
 \begin{equation} \label{lim1}
 \limsup_{j\ar +\infty} \lambda_j(\alpha_{j,c})/\alpha_{j,c}^2 < +\infty\,.
 \end{equation}
\end{cor}
\begin{proof}[Proof of Corollary~\ref{lemma3.5}]
Suppose that there exists a subsequence $j_k$ such that $ \lambda_{j_k}(\alpha_{j_k,c})/\alpha_{j_k,c}^2$ tends to $+\infty$ and $\alpha_{j_k ,c} \in A_{j_k,c}$. Then, applying Proposition~\ref{l:regime-Einfty},  the left hand side of~\eqref{e:regime-interm} equals $0$ since $\alpha_{j_k,c}$ is critical, whereas the right hand side is negative. This yields a contradiction.
\end{proof}

\begin{proof}[Proof of Proposition~\ref{l:regime-Einfty}]
We drop the index $j$ for readability. 
We start from the eigenvalue equation~\eqref{e:eigen} with $\lambda=\lambda(\alpha)$, in which we set $w = T_{\lambda^{1/4}}u$, where the dilation operator $T_\eta$ is defined in~\eqref{e:Tnu}. We obtain
\begin{align*}
\left( -\frac{1}{\lambda^{1/2}} \frac{d^2}{ds^2} + \left(\lambda^{1/2} \frac{s^2}{2}-\alpha\right)^2 \right) w = \lambda w\, . 
\end{align*}
 Dividing by $\lambda$, we obtain a semiclassical problem with semiclassical parameter $\eps:=\lambda^{-3/4}\to 0$ (by assumption), namely 
 \begin{align*}
\left( -\eps^2 \frac{d^2}{ds^2} + \left(\frac{s^2}{2}-\delta \right)^2 \right) w =  w , \quad \text{with } \delta =  \frac{\alpha}{\lambda^{1/2}} \to 0  \text{ (by assumtion)}.
\end{align*}
  Let $\mu$ be a semiclassical measure at scale $\eps_k$, associated to a sequence $(w_k,\eps_k, \delta_k)$ of solutions to this equation (see e.g.~\cite{LL:22} for a definition). The energy estimate 
 $$
  \|\eps_k w_k' \|_{L^2(\R)}^2  + \| ( s^2/2 -\delta_k )  w_k   \|_{L^2(\R)}^2 = \| w_k \|_{L^2(\R)}^2 =1
 $$
 shows on the one hand that $\|\eps_k w_k' \|_{L^2(\R)}$ is bounded, whence $\limsup_{k \to +\infty}  \left\| \mathcal{F}(w_k) (\xi)\right\|_{L^2(\eps_k|\xi|\geq R)} \to 0$ as $R\to + \infty$ (that is to say, the sequence $w_k$ is $\eps_k$-oscillating) and on the other hand that for $k$ sufficiently large so that $\delta_k\leq 1$,
 $$
 \frac12\|  s^2  w_k   \|_{L^2(\R)}  \leq  \| ( s^2/2 -\delta_k ) w_k   \|_{L^2(\R)} +\| w_k   \|_{L^2(\R)}  \leq 2 , 
 $$
 whence $\limsup_{k \to +\infty}  \left\|  w_k \right\|_{L^2(|x|\geq R)} \to 0$ as $R\to + \infty$ (that is to say, $w_k$ is compact at infinity).
 As a consequence (see e.g.~\cite{LL:22}) the semiclassical measure $\mu$ is a probability measure on $\R^2$ and its projection on the $x$-space,  $\pi_* \mu$ is a probability measure on $\R$.
 Then, the semiclassical pseudodifferential calculus shows (see e.g.~\cite{Zwo,LL:22}) that $\supp(\mu) \subset \{(s,\xi) \in \R^2,p(s,\xi) =1 \}$ with $p(s,\xi) =  \xi^2 + \frac{s^4}{4}$, and that $\mu$ is $H_p$-invariant, namely $H_p\mu= 0$. As a consequence, $\mu$ is the unique invariant probability measure supported by $\{(s,\xi) \in \R^2,p(x,\xi) =1 \}$, given explicitly as follows (see~\cite[Theorem~1.6]{LL:22}): for all $a \in C^\infty_c(\R^2)$, 
 $$
 \langle \mu , a \rangle = \frac{K_1}{2} \sum_{\pm} \int_{-\sqrt{2}}^{\sqrt{2}} a (s, \pm \sqrt{1-s^4/4}) \frac{ds}{\sqrt{1-s^4/4}} , \quad K_1  = \left( \int_{-\sqrt{2}}^{\sqrt{2}} \frac{ds}{\sqrt{1-s^4/4}} \right)^{-1} .
 $$ 
 Combined with Lemma~\ref{local} below (stating in a similar context that eigenfunctions decay at infinity faster than any polynomial), this implies that for all $a \in C^\infty(\R)$ having at most polynomial growth at infinity, 
 $$
 \int_\R a(s) |w_k(s)|^2 ds \to K_1 \int_{-\sqrt{2}}^{\sqrt{2}}  \frac{a (s) }{\sqrt{1-s^4/4}} ds\, .
 $$
Since this convergence holds for any a sequence $(w_k,\eps_k, \delta_k)$ of solutions, it always holds for the full sequence $(w,\eps, \delta)$.

 Recalling that $u =T_{\lambda^{-1/4}} w$, we now deduce from the first variation formula~\eqref{FeHei} (and using that eigenfunctions are assumed normalized) that  
 $$
 \lambda'(\alpha) = 2\alpha  - \int_\R t^2 |u_k|^2(t) dt  =  \alpha \left(2 - \lambda^{1/2}\int_\R s^2 |w_k|^2(s) ds  \right)  .
 $$
Recalling that $\lambda = \lambda(\alpha) \to + \infty$  by assumption, we finally deduce that 
 $$
 \lim\frac{\lambda'(\alpha)}{\alpha \lambda(\alpha)^{1/2}} = \langle \pi_* \mu , s^2 \rangle 
 =  - K_1 \int_{-\sqrt{2}}^{\sqrt{2}}  \frac{s^2}{\sqrt{1-s^4/4}} ds  .
 $$
 Since this holds for any subsequence, $(u_k, \lambda_k, \alpha_k)$, this holds for any triple $(u, \lambda,\alpha)$.
\end{proof}

\section{Semiclassical measures in the regime $\lambda_j(\alpha_j)\lesssim \alpha_j^2$ and proof of Theorem~\ref{th1.4}} 
\label{s:semiclassique}
\subsection{Semiclassical rescaling}

Recalling the definition of the dilation operator $T_\eta$ in~\eqref{e:Tnu}, and choosing $\eta=\sqrt{\alpha}$ for the scaling parameter, we obtain
 \begin{equation} \label{eq:rescal}
T_{\sqrt{\alpha}} \, \mathfrak h _{\mathcal M}(\alpha) \, T_{\sqrt{\alpha}}^{-1} = -  \frac 1 \alpha \frac{d^2}{ds^2} + \alpha^2 (s^2/2 -1)^2= \alpha^2 \left(-h^2 \frac{d^2}{ds^2} + (s^2/2 -1)^2\right) = \alpha^2 P_h , 
\end{equation}
 with 
 \begin{equation}\label{eq:scl}
 h = \alpha^{-3/2}, \quad P_h :=-h^2 \frac{d^2}{ds^2} + (s^2/2 -1)^2 .
 \end{equation}
Hence the analysis as $\alpha \rightarrow +\infty$ becomes an analysis as $h \ar 0$.
Recall that we consider the real, $L^2$ normalized eigenfunction $u$ (resp. $u_k$) associated to the eigenvalue $\lambda$ (resp. $\lambda_k$) that is to say solution to 
\begin{align}
\label{e:eigen}
\HM u = \lambda u , \quad u \in D(\HM) , \quad \|u\|_{L^2(\R)}=1 . 
\end{align}
After the semiclassical rescaling~\eqref{eq:rescal}, we set
\begin{equation}
\label{e:vh-ualpha}
v(\cdot,h):=T_{\sqrt{\alpha}}\,u(\cdot , \alpha),\quad \text{resp.}\quad v_k(\cdot,h):=T_{\sqrt{\alpha}}\, u_k(\cdot,\alpha) ,
\end{equation}
and (recalling that $\alpha= h^{-2/3}$), 
\begin{equation}\label{eq:corr1}
E(h) :=\frac{\lambda(\alpha)}{\alpha^2} = h^{4/3} \lambda(\alpha), \quad \text{resp.}\quad E_k(h) :=\frac{\lambda_k(\alpha)}{\alpha^2} = h^{4/3} \lambda_k(\alpha) , 
\end{equation}
which, according to~\eqref{e:eigen} now solve the semiclassical eigenvalue equation
\begin{align}
\label{e:eigen-semi}
P_h v = E(h) v\, , \quad v \in D(P_h)\, , \quad \|v\|_{L^2(\R)}=1 \,,
\end{align}
where we recall that the operator $P_h$ is defined in~\eqref{eq:scl}.

The first variation formula~\eqref{FeHei} for the eigenvalue equation becomes 
\begin{align}
\label{e:lambda'h}
\lambda'(\alpha) = 2\alpha - \|x u(\cdot , \alpha )\|^2_{L^2}   =h^{-2/3} \left( 2 - \|x v(\cdot , h )\|^2_{L^2}  \right) .
\end{align}
Localizing critical points of $\lambda$ is thus equivalent  to solving $2 - \|x v(\cdot , h )\|^2_{L^2} =0$, where $v(\cdot, h)$ solves~\eqref{e:eigen-semi}.
We finally reformulate the second variation formula~\eqref{eq:2.15aa} in the semiclassical setting.
\begin{lemma}
\label{l:derivee-seconde-h}
With $\alpha= h^{-2/3}$ and $v(\cdot,h)$ as in~\eqref{e:vh-ualpha}, we have 
\begin{equation}\label{eq:2.15a}
\lambda''(\alpha) =    (1 -  \frac32 h\, \d_h) \,  \left( \int_\R  (2-s^2)   |v(s,h)|^2 ds\right)\,.
\end{equation}
\end{lemma}
\begin{proof}
We recall from~\eqref{e:vh-ualpha} with $h=\alpha^{-3/2}$ that 
\begin{align}
\label{e:u-v}
u(t,\alpha) = (T_{\alpha^{-1/2}}v)(t,h) = \alpha^{-1/4} v(\alpha^{-1/2}t,\alpha^{-3/2}).
\end{align}
From~\eqref{eq:2.15aa}, we deduce,  
\begin{align*}
\lambda''(\alpha) & = 2 - \d_\alpha \left( \int_\R x^2 \alpha^{-1/2}|v(\alpha^{-1/2}x,\alpha^{-3/2})|^2 dx\right)
 =  2 - \d_\alpha \left( \alpha \int_\R  s^2 |v(s,\alpha^{-3/2})|^2 ds\right) \\
  &  =  2 - \left( \int_\R s^2 |v(s,\alpha^{-3/2})|^2 ds\right) - \alpha \d_\alpha \left( \int_\R  s^2 |v(s,\alpha^{-3/2})|^2 ds\right) \\
    &  =  2 -  \int_\R s^2 |v(s,\alpha^{-3/2})|^2 ds  + \frac32 \alpha^{-3/2} \d_h \left( \int_\R  s^2 |v(s,\alpha^{-3/2})|^2 ds\right) \\
        &  =  2 -  \int_\R s^2 |v(s,h)|^2 ds  + \frac32 h \d_h \left( \int_\R  s^2 |v(s,h)|^2 ds\right) ,
\end{align*}
where we have used that $\d_\alpha (f(\alpha^{-3/2})) = -\frac32 \alpha^{-5/2}f'(\alpha^{-3/2})$ in the penultimate equality with $f(h)= \int_\R  s^2 |v(s,h)|^2ds$.
Finally,~\eqref{eq:2.15a} follows from using the normalization of $v(\cdot,h)$.
\end{proof}
\begin{remark}
\label{r:eigenvalues-semiclass-1}
Since we have $\lambda(\alpha) = \alpha^2 E(\alpha^{-3/2})$, we also have the following direct relations between the variations of the eigenvalues of $\mathfrak h _{\mathcal M}(\alpha)$ and those of the semiclassical operator $P_h$ (in the present regime):
\begin{align*}
\frac{\lambda(\alpha)}{\alpha^2} & =  E(h) , \quad h=\alpha^{-3/2} , \nonumber\\
\frac{\lambda'(\alpha)}{\alpha} &= 2 E(h) -\frac32hE'(h) , \quad h=\alpha^{-3/2} , \\  
\lambda''(\alpha) &= 2 E(h) - \frac{9}{4}hE'(h)+ \frac14h^2E''(h) , \quad h=\alpha^{-3/2} .  
\end{align*}
\end{remark}
 
\subsection{Double well analysis}
The operator $P_h$ defined in~\eqref{eq:scl} has the form $P_h:=-h^2  \frac{d^2}{ds^2}  + V(s)$ with $$V(s) = (s^2/2 -1)^2$$ and the standard multidimensional semiclassical analysis (see for example Helffer-Robert \cite{HeRo}, Helffer-Sj\"ostrand \cite{HeSj}
 and references therein) can be used. The potential is a symmetric double well potential and the two minima occur at $\pm \sqrt{2}$ with energy $0$. The quadratic approximation at the bottom leads to the following result~\cite{HeSj}. 
\begin{lemma}[Bottom of the spectrum]
\label{l:bottom}
For all $\beta\in\R_+\setminus\sqrt{2}(2\mathbb{N}+1)$, there are~$N_\beta \in \N , h_0>0$ such that 
$$
  \Sp(P_h)  \cap (-\infty , \beta h)    
  = \big\{E_n(h), n \in \{1,\dots , N_\beta\} \big\} , \quad \text{ uniformly for }  h \in (0,h_0) ,
$$
with $0< E_n(h)<E_{n+1}(h) < \beta h$ for all $n \in \{1,\dots , N_\beta-1\}$ and $h \in (0,h_0)$. Moreover, as $h \to 0^+$, we have
  \begin{align}\label{eq:asbah}
  &E_{2j+1} (h) \sim  \sqrt{2} \, ( 2j+1)  \, h \,,  \\
\label{eq:asbbh}
 &  E_{2j+2}(h) -  E_{2j+1}(h) = \mathcal O (h^{+\infty})\, ,
   \end{align}
   uniformly  for all $j \in \N$ such that $2j+2\leq N_\beta$.
\end{lemma}
Note that~\eqref{eq:asbbh} is a weak form of the tunneling analysis. The asymptotics \eqref{eq:asba} and \eqref{eq:asbb} are straightforward consequences of~\eqref{eq:asbah}-\eqref{eq:asbbh} together with~\eqref{eq:corr1}.

 Outside the minima, the potential $V$ admits a critical point at $s=0$ corresponding to a local maximum with $V(0)=1$. Except at this energy $1$, one can describe the whole spectrum by using Bohr-Sommerfeld formulas which are specific to the dimension one.
 
We now recall the localization properties of the eigenfunctions in the so-called classical region determined by their  corresponding eigenvalue. 
 For a given energy $E$, we define the classically allowed region by 
 $$
 K_E := V^{-1} ((-\infty,E]) .
 $$
 
\begin{prop}\label{local}
Let $E \in \R$. 
For any sequence $(E_{j(h)} (h), v_{j(h)}(h))$ of eigenpairs of $P_h$ such that $$\limsup_{h\ar 0} E_{j(h)}(h) \leq E$$ and $\|v_{j(h)}(h)\|_{L^2(\R)}=1$, and any closed interval $I$ such that
 $I \cap K_E  =\emptyset$, we have for any $k$ $$\| v_{j(h)}(h)\|_{B^k(I)}=\mathcal O (h^\infty),$$ where $$B^k (I)=\{ u\in L^2(I), s^\ell u^{(m)} (s) \in L^2(I)\,,\forall (\ell,m)  \mbox{ s.t. } \ell + m \leq k\}\,.$$
 \end{prop}
 This is a consequence of classical Agmon estimates, see~\cite{HeSj} or~\cite{He88}. The latter actually yield an exponential decay of eigenfunctions away from the classically allowed region $K_E$, but this  is not needed in our analysis.
  
\subsection{Semiclassical measures}\label{ss3.3}
We now only consider the semiclassical regime $h \to 0^+$ and consider the limit measures for the densities $v(x,h)^2dx$.
Taking advantage of the parity of the potential $V =\left(\frac{x^2}{2}-1\right)^2$, we define for any given energy $E \geq 0$,
\begin{align}
\label{e:def-xpm}
x_+(E) = \sqrt{2+2 \sqrt{E}} \text{ if } E \in [0,+\infty), \quad \text{ and }\quad  x_-(E) = \left\{ \begin{array}{ll}  \sqrt{2-2 \sqrt{E}} & \text{if } E\in [0,1], \\
0 & \text{if } E \geq 1 .
\end{array}
\right.
\end{align}
Notice that $x_\pm$ is defined so that for all $E\geq0$,
$$
[x_-(E) ,x_+(E)] = \{x\geq 0, V(x)\leq E\} = K_E \cap \R_+ 
$$
is the classically allowed region intersected with $\R^+$. That is to say,
$$
K_E = [-x_+(E),x_+(E)] \quad \text{if }E\geq 1 , \quad K_E = [-x_+(E) ,-x_-(E)]  \cup [x_-(E) ,x_+(E)] \quad \text{if }E\leq 1 .
$$
The densities $v(x,h)^2dx$ satisfy the following asymptotic repartition. 
\begin{prop}
\label{l:semi-meas}
Assume $v$ solves~\eqref{e:eigen-semi} with $E(h) \to E$ and $h\to 0^+$, then $v(x,h)^2dx \rightharpoonup \mathfrak{m}_E$ in the sense of measures in $\R$ where $\mathfrak{m}_E$ is the nonnegative Radon measure given by
\begin{align}
\mathfrak{m}_E & = \frac12 (\delta_{-\sqrt{2}} + \delta_{\sqrt{2}}) \quad  \text{ if } E=0 , \label{mu-E=0} \\
\mathfrak{m}_E & =  \frac{C(E) }{2} \frac{ \mathds{1}_{[-x_+(E) ,-x_-(E)]}(x) dx}{\sqrt{E-V(x)}} + \frac{C(E) }{2} \frac{ \mathds{1}_{[x_-(E) ,x_+(E)]}(x) dx}{\sqrt{E-V(x)}}, \quad  \text{ if } E \in (0,1) , \label{mu-E=0-1}\\
\mathfrak{m}_E & = \delta_0 \quad  \text{ if } E= 1  ,\label{mu-E=1}\\
\mathfrak{m}_E & =  \frac{C(E) }{2}   \frac{ \mathds{1}_{[-x_+(E) ,x_+(E)]}(x) dx}{\sqrt{E-V(x)}}, \quad  \text{ if } E >1 \,,\label{mu-E=1-infty}
\end{align}
where 
\begin{align}
\label{eq:def-CE}
C(E)  = \left( \int_{x_-(E)}^{x_+(E)}  (E-V(x))^{-\frac12} dx\right)^{-1} , \quad E\in (0,1) \cup (1,\infty).
\end{align}
\end{prop}
Note that the expression for $E\in(0,1)$ is also valid for $E>1$. The continuity of the map $E\mapsto \mathfrak{m}_E$  at $E=1$ for the weak-* topology of measures is investigated in Proposition~\ref{e:continuity-mE} below.
\begin{proof}[Proof of Proposition~\ref{l:semi-meas}]
Note first that all eigenfunctions being either odd or even, the measures $v(x,h)^2dx$ are all even and so are their limits. 
Next, once restricted to the half real line (thus solution to a Dirichlet or a Neumann problem on $\R^+$), the case $E\neq 1$ is proved e.g. in~\cite{HMR} (Theorem 4.1) (before the introduction of the theory of semiclassical measures)  or more recently in \cite[Theorem~1.6]{LL:22} (see also~\cite[pp 139-143]{GS:94} or \cite[Section~6~pp~190--198]{Olver:book} for different approaches on a similar problem). 

Finally, the case $E=1$ follows from~\cite[Section~5]{CdVPa}. For the sake of completeness, we give a short proof inspired from~\cite[Section~2.2, Proof of Theorem~1.6]{LL:22} in the case $E=1$. 
In case  $E(h) \to 1$, any semiclassical measure $\mu$ (see e.g.~\cite{LL:22} for a definition and a proof of these properties) is a probability measure on $\R^2$, supported by $p^{-1}(1):=\{(x,\xi) \in \R^2 , p(x,\xi)= 1 \}$ where $p(x,\xi)=\xi^2 +\big(\frac{x^2}{2}-1\big)^2$, and moreover invariant by the Hamilton flow of $p$ inside $p^{-1}(1)$, i.e. $0=H_p\mu = \left(2\xi \d_x -2x\big(\frac{x^2}{2}-1\big)\d_\xi \right)\mu$. Now the set $p^{-1}(1)$ is partitioned into three $H_p$-invariant disjoint sets:
$$
p^{-1}(1) = \{(0,0)\}\sqcup \mathcal{C}_+ \sqcup \mathcal{C}_- , \quad \text{with} \quad   \mathcal{C}_\pm := \{(x,\xi) \in \R^2 ,\pm x >0 , p(x,\xi)= 1 \} .
$$
As a consequence, $\mu = \mu_- + \alpha \delta_{(0,0)} + \mu_-$ where $\alpha \in [0,1]$ and $\mu_\pm$ is an invariant finite measure supported by $\mathcal{C}_\pm$. But since $\mathcal{C}_\pm$ is not compact, the only invariant finite measure it carries is $\mu_\pm=0$. Since $\mu$ is a probability measure, we deduce $\mu = \delta_{(0,0)}$ and $\mathfrak{m}_1 = \pi_* \mu = \delta_0$ (see~\cite{LL:22} for a proof of the first equality). Uniqueness of the limit measure shows that the full sequence converges.
\end{proof}

\begin{remark}
Note that in the single well problem~\cite{LL:22}, it is enough that $v(x,h)$ solves~\eqref{e:eigen-semi} approximately, that is to say, with a $o(h)$ remainder, in order to characterize the limit measure $\mu$ and $\mathfrak{m}_E$. In the double well regime, we actually use here in an essential way that $v(x,h)$ is a genuine eigenfunction, i.e. solves~\eqref{e:eigen-semi} exactly (it would actually work if $v(x,h)$ would solve~\eqref{e:eigen-semi} approximately with a $O(e^{-S_1/h})$ remainder, with $S_1$ sufficiently large) due to the tunnelling effect between the two wells, see~\cite{HeSj}.
\end{remark}

\begin{cor}\label{cor3.4}
Let $E\geq 0$, assume that $\alpha \to +\infty$ and $\frac{\lambda (\alpha)}{\alpha^2} \to E$, then  
\begin{equation}
\label{e:lim--e}
\alpha^{-1} \lambda'(\alpha) \to \Phi(E) \,,
\end{equation}
with 
\begin{equation}
\label{e:def-Phi} \Phi(E) := \left\{ \begin{array}{ll}
0, & \text{if } E = 0\,  ,\\
C(E)  F(E) , & \text{if } E\in (0,1) \cup (1,\infty) \,, \\
2 , & \text{if } E=1\,  .
\end{array}
\right.
\end{equation}
where, for $E\in (0,1) \cup (1,\infty)$,  $C(E) $ is defined in~\eqref{eq:def-CE} and
\begin{equation}
\label{e:def-F(E)}
F(E): =  \int_{x_-(E)}^{x_+(E)} \frac{2-x^2}{\sqrt{E-V(x)}}dx=\int_{ \sqrt{(2-2 \sqrt{E})_+}}^{\sqrt{2+2 \sqrt{E}}}\frac{2- x^2}{\sqrt{E-(x^2/2-1)^2}} dx 
, \quad E\in (0,1) \cup (1,\infty)  .
\end{equation}
where $(z)_+=\max\{z,0\}$.
\end{cor}
Note that the function $F$ in~\eqref{e:def-F-00} in Theorem~\ref{th1.4} is the restriction to $(1,+\infty)$ of the function defined in~\eqref{e:def-F(E)}.
The properties of the limit object in~\eqref{e:lim--e}, and in particular of the function $F(E)$ are studied in the next section.
Note that this corollary is not sufficient for proving that there is no critical point near the energy $E=0$.
\begin{proof}
Recalling~\eqref{e:lambda'h},  
as $h \to 0$ and $E(h) = h^{4/3}\lambda(\alpha) \to E$, we have, using Proposition~\ref{l:semi-meas}, 
$$
h^{2/3}\lambda'(\alpha) = 2 - \|x v(\cdot , h )\|^2_{L^2} \to 2 - \int_\R x^2 d\mathfrak{m}_E(x)  = \int_\R (2-x^2) d\mathfrak{m}_E(x) . 
$$
That is to say, according to~\eqref{mu-E=0}--\eqref{mu-E=1-infty} respectively, 
\begin{align*}
h^{2/3}\lambda'(\alpha) & \to 0  \quad  \text{ if } E=0 ,  \\
h^{2/3}\lambda'(\alpha) & \to 2 -  C(E)  \int_{x_-(E)}^{x_+(E)} \frac{x^2 dx}{\sqrt{E-V(x)}}=C(E)  \left(  \int_{x_-(E)}^{x_+(E)} \frac{2-x^2}{\sqrt{E-V(x)}}dx\right), \quad  \text{ if } E \in (0,1) , \\
h^{2/3}\lambda'(\alpha) & \to 2 \quad  \text{ if } E= 1  , \\
h^{2/3}\lambda'(\alpha) & \to 2 -  C(E)  \int_{0}^{x_+(E)} \frac{x^2 dx}{\sqrt{E-V(x)}} = C(E)  \left(  \int_{0}^{x_+(E)} \frac{2-x^2}{\sqrt{E-V(x)}}dx\right), \quad  \text{ if } E >1  . 
\end{align*}
\end{proof}

\subsection{Study of $\mathfrak{m}_E$ and $F(E)$ near $E=1$} \label{ss3.5}
Our goal in this section is to prove continuity of the family $\mathfrak{m}_E$ and the function $C(E)F(E)$ through $E=1$.
As a preliminary, we study properties  as $E \to 1^\pm$ of the integral
$$
\mathsf{F}_\varphi(E): =  \int_{x_-(E)}^{x_+(E)} \frac{\varphi(x)}{\sqrt{E-V(x)}}dx=\int_{\sqrt{(2-2 \sqrt{E})_+}}^{\sqrt{2+2 \sqrt{E}}}\frac{\varphi(x)}{\sqrt{E-(x^2/2-1)^2}} dx ,
$$ 
defined for $\varphi \in C^0(\R)$ and $E \in (0,1) \cup (1,+ \infty)$.
\begin{lemma}
\label{l:lemma-Fphi}
There is $C>0$ such that for all $\delta \in (0,1)$, there is $C_\delta>0$ such that for all $\varphi \in C^1(\R)$, 
$$
\left|\mathsf{F}_\varphi(E) -\varphi(0)\mathsf{F}_1^{\delta,+}(E)\right| \leq C \delta\|\varphi'\|_{L^\infty(0,1)} +C_\delta \|\varphi\|_{L^\infty(0,3)},\quad  \text{ for all } E\in (1,9/4) ,
$$ 
where 
 \begin{align}
 \label{e:asympt+}
0 \leq  \mathsf{F}_1^{\delta,-}(E) -\int_0^{\delta/\sqrt{2\eps}}\frac{dy}{\sqrt{y^2+1}}  \leq C \delta^2 \int_0^{\delta/\sqrt{2\eps}}\frac{dy}{\sqrt{y^2+1}}, \quad \sqrt{E}=1+\eps ,
 \end{align}
and 
\begin{align*}
\left|\mathsf{F}_\varphi(E) -\varphi(0)\mathsf{F}_1^{\delta,-}(E)\right|  \leq C \delta\|\varphi'\|_{L^\infty(0,1)} +C_\delta \|\varphi\|_{L^\infty(0,3)},\quad  \text{ for all } E\in (1/4,1) ,
\end{align*}
where 
 \begin{align}
 \label{e:asympt-}
0 \leq  \mathsf{F}_1^{\delta,-}(E) -\frac{\sqrt{2}}{\sqrt{2-\eps}}\int_1^{\delta/\sqrt{2\eps}}\frac{dy}{\sqrt{y^2-1}}  \leq C \delta^2 \int_1^{\delta/\sqrt{2\eps}}\frac{dy}{\sqrt{y^2-1}} , \quad \sqrt{E}=1-\eps .
 \end{align}
\end{lemma}
Note that Lemma~\ref{l:lemma-Fphi} implies that if  $\varphi(0) \neq 0$, then 
\begin{align}
\label{e:to+-infty}
\mathsf{F}_\varphi(E) \sim  \varphi(0) \frac12 \log\left(\frac{1}{|\sqrt{E}-1|} \right) \sim  \varphi(0) \frac12 \log\left(\frac{1}{|E -1|} \right)  \quad \text{as } E\to 1^\pm .  
\end{align}

As a consequence, we get the following proposition. We denote by $\mathcal{M}_c(\R)$ the set of compactly supported Radon measures on $\R$.  
\begin{prop}
\label{e:continuity-mE}
With $\mathfrak{m}_E$ defined by~\eqref{mu-E=0}--\eqref{mu-E=1-infty}, the map $[1,+\infty) \to \mathcal{M}_c(\R)$, $E \mapsto \mathfrak{m}_E$ is weakly-* continuous, that is to say: for all $\varphi \in C^0(\R)$ and all $E_0\in [0,+\infty)$, we have 
$$
\langle \mathfrak{m}_E , \varphi\rangle \to \langle \mathfrak{m}_{E_0} , \varphi\rangle , \quad \text{ as }E \to E_0 \quad(\text{resp. } E\to 0^+\text{ if }E_0=0) .
$$
\end{prop}
Note that this property is not strictly needed for our analysis  but completes the picture nicely.  Its slightly weaker Corollary~\ref{c:cor-cont-E} below is however needed.\\
\begin{proof}[Proof of Proposition~\ref{e:continuity-mE} from Lemma~\ref{l:lemma-Fphi}]~

First, for $E \in (0,1)\cup(1,+\infty)$, we notice that $C(E)= \frac{1}{\mathsf{F}_1(E)}$.
The measure $\mathfrak{m}_E$ in~\eqref{mu-E=0}--\eqref{mu-E=1-infty} is even and absolutely continuous, so that for any $\varphi \in C^0(\R)$
$$
\langle \mathfrak{m}_E , \varphi \rangle = \int_{\R^+}\varphi \ d\mathfrak{m}_E + \int_{\R^+} \check\varphi\  d\mathfrak{m}_E =  \frac{1}{2\mathsf{F}_1(E)} \left( \mathsf{F}_\varphi(E) + \mathsf{F}_{\check\varphi}(E)\right) ,
$$
where $\check\varphi(x)=\varphi(-x)$. This implies the continuity statement for $E \in (0,1)\cup(1,+\infty)$. 

Second, we prove the continuity statement at $E=1$. Given this expression, it suffices to show that $\frac{\mathsf{F}_\varphi(E)}{\mathsf{F}_1(E)} \to \varphi(0)$ as $E\to 1^\pm$.
 To this aim, we assume first that $\varphi \in C^1(\R)$. According to Lemma~\ref{l:lemma-Fphi} (taken for any fixed $\delta\in (0,1)$), we have 
$$
 \frac{\mathsf{F}_\varphi(E)}{\mathsf{F}_1(E)} =  \frac{\varphi(0)\mathsf{F}_1(E) + O(1)\|\varphi\|_{C^1(0,3)}}{\mathsf{F}_1(E)} \to \varphi(0) , 
$$
as $E\to 1^\pm$, where we have used that $\mathsf{F}_1(E)\to +\infty$ as a consequence of~\eqref{e:asympt+}-\eqref{e:asympt-}. \\ This implies that $\langle \mathfrak{m}_E , \varphi \rangle \to \varphi(0)$ for all $\varphi \in C^1(\R)$.
Finally, if $\varphi \in C^0(\R)$, there is $\varphi_n\in C^1(\R)$ such that $\| \varphi_n - \varphi\|_{L^\infty(-3,3)}\to 0$, and we write 
\begin{align*}
\left| \langle \mathfrak{m}_E , \varphi \rangle -\varphi(0) \right|
&  \leq 
\left| \langle \mathfrak{m}_E , \varphi  - \varphi_n \rangle \right| +
\left| \langle \mathfrak{m}_E , \varphi_n \rangle -\varphi_n (0) \right| 
+ |\varphi_n (0) - \varphi(0)| \\
& \leq 2 \| \varphi_n - \varphi\|_{L^\infty(-3,3)} + \left| \langle \mathfrak{m}_E , \varphi_n \rangle -\varphi_n (0) \right|  , 
\end{align*}
where we have used that $\mathfrak{m}_E$ is a probability measure.\\
 Given $\epsilon>0$, we fix $n \in \N$ such that the first term is $\leq \epsilon$ and then take $E$ close enough to $1$ so that the second term is $\leq \epsilon$, and for such $E$ we have $\left| \langle \mathfrak{m}_E , \varphi \rangle -\varphi(0) \right| \leq 2 \epsilon$. We have thus proved the continuity statement at $E=1$ for any $\varphi \in C^0(\R)$.

Third, we prove the continuity statement at $E=0^+$. To this aim, it suffices to remark that, given any sequence $E_n\to 0^+$, we may extract from the sequence $\nu_n := 2\mathds{1}_{\R^+}\mathfrak{m}_{E_n} \in \mathcal{M}_c(\R^+)$ of probability measure a converging subsequence in $\mathcal{M}_c(\R^+)$. Since $\supp \nu_n = [x_-(E_n),x_+(E_n)]$, the limit measure is a probability measure supported in $\bigcap_{N\in \N}[x_-(E_n),x_+(E_n)] = \{\sqrt{2}\}$. Therefore the unique possible limit is $\delta_{\sqrt{2}}$, and the whole sequence $\nu_n$ converges to $\delta_{\sqrt{2}}$. This implies that $2\mathds{1}_{\R^+}\mathfrak{m}_{E} \rightharpoonup \delta_{\sqrt{2}}$ as $E\to 0^+$, and, since $\mathfrak{m}_E$ is even, that $\mathfrak{m}_E \rightharpoonup \frac{1}{2}\left(\delta_{-\sqrt{2}} + \delta_{\sqrt{2}} \right)$.

This concludes the proof of the proposition. 
\end{proof}

\begin{cor}
\label{c:cor-cont-E}
With the functions $C$ defined in~\eqref{eq:def-CE} and $F$ in~\eqref{e:def-F(E)}, the function $\Phi$ defined in~\eqref{e:def-Phi} is continuous on $[0,+\infty)$.
\end{cor}
\begin{proof}
This follows directly from the fact that $C(E)= \frac{1}{\mathsf{F}_1(E)}$ and $F(E)=\mathsf{F}_{2-x^2}(E)$, so that $$\Phi(E)= C(E)F(E)=  \frac{\mathsf{F}_{2-x^2}(E)}{\mathsf{F}_{1}(E)} =\langle \mathfrak{m}_E , 2-x^2 \rangle,$$ 
to which Proposition~\ref{e:continuity-mE} applies.
\end{proof}
We finally prove Lemma~\ref{l:lemma-Fphi}.
\begin{proof}[Proof of Lemma~\ref{l:lemma-Fphi}]
First, for $E>1$, we set $\sqrt{E}=1+\eps >0$ and have
$$
\mathsf{F}_\varphi(E) =\int_{0}^{\sqrt{2+2 \sqrt{E}}}\frac{\varphi(x)}{\sqrt{E-(x^2/2-1)^2}} dx
=\int_{0}^{\sqrt{4+2\eps}}\frac{\varphi(x)}{\sqrt{(1+\eps)^2-(x^2/2-1)^2}} dx .
$$
For $\delta \in (0,1)$ independent of $\eps$, we write $\mathsf{F}_\varphi(E)  = \mathsf{G}_\varphi^{\delta,+}(E) + \mathsf{F}_\varphi^{\delta,+}(E)$ with 
$$
\mathsf{G}_\varphi^{\delta,+}(E) : =  \int_{\delta}^{\sqrt{4+2\eps}} \frac{\varphi(x)}{\sqrt{(1+\eps)^2-(x^2/2-1)^2}} dx = \int_{\delta}^{\sqrt{4+2\eps}} \frac{\varphi(x)}{\sqrt{2\eps + \eps^2 -x^4/4 + x^2 }} dx  .
$$
In particular,
$$\left| \mathsf{G}_\varphi^{\delta,+}(E) \right| \leq  \|\varphi\|_{L^\infty(0,3)} \mathsf{G}_1^{\delta,+}(E)$$
where 
$$
\mathsf{G}_1^{\delta,+}(E)
\to\int_{\delta}^{2} \frac{1}{\sqrt{ x^2-x^4/4}} dx  =  \int_{\delta}^{2} \frac{2}{x\sqrt{(2+x)(2-x)}} dx <+\infty ,
$$
as $\eps \to 0^+$. As a consequence, for any $\delta \in (0,1)$, there is $C_\delta>0$ such that for all $\eps \in (0,1/2)$, all $\varphi \in C^0(\R)$, 
$$
 \left| \mathsf{G}_\varphi^{\delta,+}(E)\right| \leq C_\delta \|\varphi\|_{L^\infty(0,3)} .
 $$
 It thus remains to consider 
 \begin{align*}
\mathsf{F}_\varphi^{\delta,+}(E) & :=  \int_{0}^{\delta} \frac{\varphi(x)}{\sqrt{(1+\eps)^2-(x^2/2-1)^2}} dx =  \int_{0}^{\delta} \frac{\varphi(x)}{\sqrt{(2+\eps-x^2/2)(x^2/2+\eps)}} dx \\
&  = \sqrt{2} \int_0^{\delta/\sqrt{2\eps}}\frac{\varphi(\sqrt{2\eps}y)}{\sqrt{2+\eps-\eps y^2} \sqrt{y^2+1}} dy ,
\end{align*}
where we have set $x=\sqrt{2\eps} y$. We also notice that 
$$
y \in [0,  \delta/\sqrt{2\eps}] \quad  \implies \quad   \frac{1}{\sqrt{2+\eps}} \leq \frac{1}{\sqrt{2+\eps-\eps y^2}} \leq \frac{1}{\sqrt{2+\eps- \delta^2/2}} .
$$
As a consequence, we have 
 \begin{align*}
\left|\mathsf{F}_\varphi^{\delta,+}(E) -\varphi(0)\mathsf{F}_1^{\delta,+}(E)\right|
&  =\sqrt{2} \left|\int_0^{\delta/\sqrt{2\eps}}\frac{\varphi(\sqrt{2\eps}y)-\varphi(0)}{\sqrt{2+\eps-\eps y^2} \sqrt{y^2+1}}dy\right| \\
& \leq \sqrt{2}\sqrt{2\eps}\|\varphi'\|_{L^\infty(0,1)} \int_1^{\delta/\sqrt{2\eps}}\frac{dy}{\sqrt{2+\eps-\eps y^2} \sqrt{y^2+1}}  \leq C \delta\|\varphi'\|_{L^\infty(0,1)} ,
\end{align*}
uniformly for $\eps \in (0,1/2)$.
Finally, we have 
 \begin{align*}
\frac{\sqrt{2}}{\sqrt{2+\eps}}\int_0^{\delta/\sqrt{2\eps}}\frac{dy}{\sqrt{y^2+1}}  \leq  \mathsf{F}_1^{\delta,+}(E)  \leq \frac{\sqrt{2}}{\sqrt{2+\eps- \delta^2/2}} \int_1^{\delta/\sqrt{2\eps}}\frac{dy}{\sqrt{y^2+1}} ,
 \end{align*}
from which~\eqref{e:asympt+} follows.

We now consider the case $E<1$. Setting $\sqrt{E}=1-\eps$ with $\eps>0$, we have 
$$
\mathsf{F}_\varphi(E) =\int_{\sqrt{(2-2 \sqrt{E})}}^{\sqrt{2+2 \sqrt{E}}}\frac{\varphi(x)}{\sqrt{E-(x^2/2-1)^2}} dx = \int_{\sqrt{2\eps}}^{\sqrt{4-2\eps}} \frac{\varphi(x)}{\sqrt{(1-\eps)^2-(x^2/2-1)^2}} dx  .
$$
For $\delta \in (0,1)$ independent of $\eps$, we have $\mathsf{F}_\varphi(E)  = \mathsf{G}_\varphi^{\delta,-}(E) + \mathsf{F}_\varphi^{\delta,-}(E)$ with 
$$
\mathsf{G}_\varphi^{\delta,-}(E) : =  \int_{\delta}^{\sqrt{4-2\eps}} \frac{\varphi(x)}{\sqrt{(1-\eps)^2-(x^2/2-1)^2}} dx = \int_{\delta}^{\sqrt{4-2\eps}} \frac{\varphi(x)}{\sqrt{-2\eps + \eps^2 -x^4/4 + x^2 }} dx  .
$$
In particular,
$$\left| \mathsf{G}_\varphi^{\delta,-}(E) \right| \leq  \|\varphi\|_{L^\infty(0,3)} \mathsf{G}_1^{\delta,-}(E)$$
where 
$$
\mathsf{G}_1^{\delta,-}(E)
\to\int_{\delta}^{2} \frac{1}{\sqrt{ x^2-x^4/4}} dx  =  \int_{\delta}^{2} \frac{2}{x\sqrt{(2+x)(2-x)}} dx <+\infty ,
$$
as $\eps \to 0^+$. As a consequence, for any $\delta \in (0,1)$, there is $C_\delta>0$ such that for all $\eps \in (0,1/2)$, all $\varphi \in C^0(\R)$, 
$$
 \left| \mathsf{G}_\varphi^{\delta,-}(E)\right| \leq C_\delta \|\varphi\|_{L^\infty(0,3)} .
 $$
 It thus remains to consider 
 \begin{align*}
\mathsf{F}_\varphi^{\delta,-}(E) & :=  \int_{\sqrt{2\eps}}^{\delta} \frac{\varphi(x)}{\sqrt{(1-\eps)^2-(x^2/2-1)^2}} dx =  \int_{\sqrt{2\eps}}^{\delta} \frac{\varphi(x)}{\sqrt{(2-\eps-x^2/2)(x^2/2-\eps)}} dx \\
&  = \sqrt{2} \int_1^{\delta/\sqrt{2\eps}}\frac{\varphi(\sqrt{2\eps}y)}{\sqrt{2-\eps-\eps y^2} \sqrt{y^2-1}} dy ,
\end{align*}
where we have set $x=\sqrt{2\eps} y$. We also notice that 
$$
y \in [1,  \delta/\sqrt{2\eps}] \quad  \implies \quad \frac{1}{\sqrt{2-\eps}} \leq \frac{1}{\sqrt{2-\eps-\eps y^2}} \leq \frac{1}{\sqrt{2-\eps- \delta^2/2}}.
$$
As a consequence, we have 
 \begin{align*}
\left|\mathsf{F}_\varphi^{\delta,-}(E) -\varphi(0)\mathsf{F}_1^{\delta,-}(E)\right|
&  =\sqrt{2} \left|\int_1^{\delta/\sqrt{2\eps}}\frac{\varphi(\sqrt{2\eps}y)-\varphi(0)}{\sqrt{2-\eps-\eps y^2} \sqrt{y^2-1}}dy\right| \\
& \leq \sqrt{2}\sqrt{2\eps}\|\varphi'\|_{L^\infty(0,1)} \int_1^{\delta/\sqrt{2\eps}}\frac{dy}{\sqrt{2-\eps-\eps y^2} \sqrt{y^2-1}}  \leq C \delta\|\varphi'\|_{L^\infty(0,1)} ,
\end{align*}
uniformly for $\eps \in (0,1/2)$.
Finally, we have 
 \begin{align*}
\frac{\sqrt{2}}{\sqrt{2-\eps}}\int_1^{\delta/\sqrt{2\eps}}\frac{dy}{\sqrt{y^2-1}}  \leq  \mathsf{F}_1^{\delta,-}(E)  \leq \frac{\sqrt{2}}{\sqrt{2-\eps- \delta^2/2}} \int_1^{\delta/\sqrt{2\eps}}\frac{dy}{\sqrt{y^2-1}} ,
 \end{align*}
from which~\eqref{e:asympt-} follows.
\end{proof}

\subsection{Analysis of $F(E)$}
To prove Theorem~\ref{th1.4}, we now study properties of the function $F$ defined in~\eqref{e:def-F(E)}. We first rewrite this quantity in a more customary form.
\begin{lemma}
\label{l:F-theta}
We have 
  \begin{align*}
 F (E)&= -2^{1/2} E^{1/4} \theta_+(E^{-1/2}) = -2^{1/2} \eta^{-1/2} \theta_+(\eta)   , \quad \text{for } \eta = E^{-1/2} \in (0,1), \\
 F(E)&= -2^{1/2} E^{1/4} \theta_-(E^{-1/2}) = -2^{1/2} \eta^{-1/2} \theta_-(\eta)  , \quad \text{for } \eta = E^{-1/2} \in (1,+\infty) ,
 \end{align*}
 with
 \begin{align}\label{e:def-theta+}
\theta_+(\eta) &: =\int_{-\eta}^{1} \tau  (1- \tau ^2) ^{-1/2} (  \tau  +\eta)^{-1/2}\, d\tau , \quad \text{for } \eta = E^{-1/2} \in (0,1), \\
\label{e:def-theta-}
\theta_-(\eta)&:=   \int_{-1}^{1} \tau  (1- \tau ^2) ^{-1/2} (  \tau  +\eta)^{-1/2}\, d\tau , \quad \text{for } \eta = E^{-1/2} \in (1,+\infty) .
\end{align}
\end{lemma}
This reduces the study of $F$ to that of $\theta_+$ and $\theta_-$.
\begin{proof}
After a change of variable $\sigma = x^2/2 -1$ (that is $x=\sqrt{2(\sigma +1)}$, $dx= ({2(\sigma +1)})^{-1/2}\, d\sigma$) we obtain from~\eqref{e:def-F(E)} that 
$$
F(E)= -2  \int_{(1-E^{1/2})_+-1}^{E^{1/2}} \sigma  (E- \sigma^2) ^{-1/2} ({2(\sigma +1)})^{-1/2}\, d\sigma\,.
$$
The change of variable $\sigma= E^{1/2} \tau$ then leads to 
\begin{align*}
F(E)= -2  E^{1/2} \int_{ (E^{-1/2} -1)_+ -E^{-1/2} }^{1} \tau  (1- \tau ^2) ^{-1/2} ({2(E^{1/2} \tau  +1)})^{-1/2}\, d\tau ,
\end{align*}
that is to say
 \begin{align*} 
 F (E)&= -2^{1/2}  E^{1/2} \int_{-E^{-1/2} }^{1} \tau  (1- \tau ^2) ^{-1/2} (E^{1/2} \tau  +1)^{-1/2}\, d\tau , \quad \text{for } E \in (1,+\infty) , \\
 F(E)&= -2^{1/2} E^{1/2} \int_{-1}^{1} \tau  (1- \tau ^2) ^{-1/2} (E^{1/2} \tau  +1)^{-1/2}\, d\tau , \quad \text{for } E \in (0,1) . 
 \end{align*}
 Setting $\eta = E^{-1/2}$, this rewrites 
  \begin{align*}  
 F (E)&= -2^{1/2} \eta^{-1/2} \int_{-\eta}^{1} \tau  (1- \tau ^2) ^{-1/2} (  \tau  +\eta)^{-1/2}\, d\tau , \quad \text{for } \eta = E^{-1/2} \in (0,1), \\
 F(E)&= -2^{1/2} \eta^{-1/2}  \int_{-1}^{1} \tau  (1- \tau ^2) ^{-1/2} (  \tau  +\eta)^{-1/2}\, d\tau , \quad \text{for } \eta = E^{-1/2} \in (1,+\infty) , 
 \end{align*}
 which is the sought result.
 \end{proof}
 
\begin{figure}[h]
\begin{center}
\includegraphics[width=15pc, height=15pc]{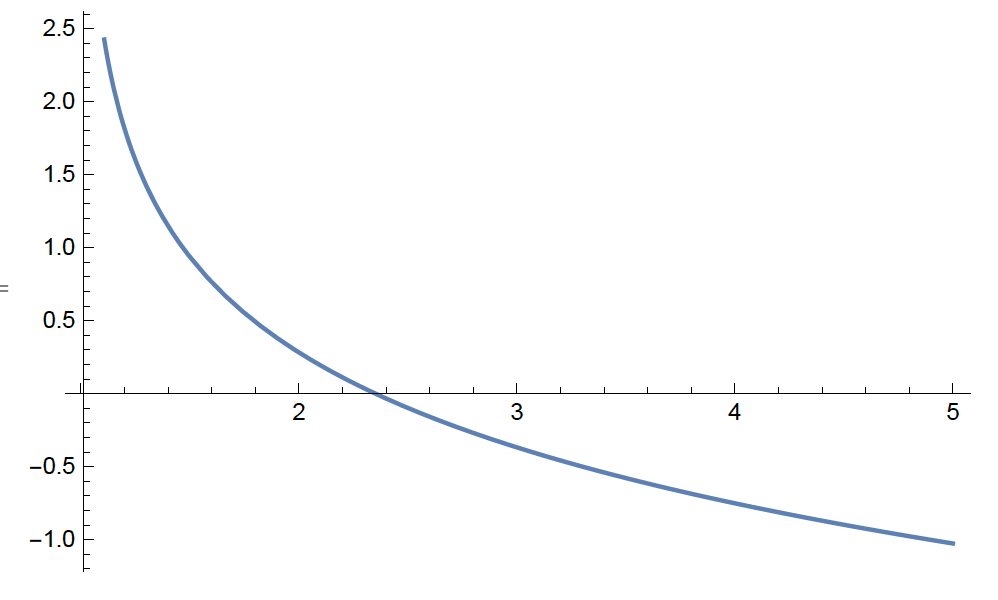}
\end{center}
\caption{Graph of $E\mapsto F(E)$ computed with Matematica for $E\in (1.01,5)$. The change of sign occurs for $E_c\approx 2.35$.} \label{f:0}
\end{figure}
We may now analyze properties of the function $F$ in~\eqref{e:def-F(E)}. The goal of this section is to prove the existence and uniqueness of $E_c$ where $F$ vanishes, which is part of the statement of Theorem~\ref{th1.4}.
\begin{prop}\label{prop4.1}
The following statements hold:
\begin{enumerate}
\item \label{e:F-increasing} The function $F$ is of class $C^1$ and we have $F'<0$ on $(1,+\infty)$.
\item \label{i:limitFinfty} $\lim_{E\ar +\infty} F(E) = -\infty$.
\item  \label{i:limitF1} $\lim_{E\ar 1^+} F(E) =+ \infty$ and $\lim_{E\ar 1^-} F(E) =+ \infty$.
 \item \label{i:F-uniq-0} $F$ admits on $(1,\infty)$ a unique zero $E_c \in (1,\infty)$ and $F'(E_c) < 0$. Numerically $E_c \approx  2.35$.
 \item \label{i:F>0} $F(E)>0$ for all $E \in (0,1)$.
 \end{enumerate}
 \end{prop}
Note that in Item~\ref{i:limitF1}, the divergence is logarithmic according to~\eqref{e:to+-infty} (taken for $F(E)=\mathsf{F}_{2-x^2}(E)$ where $\varphi(x)=2-x^2$).

 \begin{proof}
 We use Lemma~\ref{l:F-theta} to reduce the analysis of the variation of $F$ on $(1,\infty)$ to that of $\theta_+$ defined in~\eqref{e:def-theta+}, on $(0,1)$. 
We decompose $\theta_+$ as $\theta_+ =\theta_1 +\theta_2$ with
 $$
  \theta_1 (\eta):=  \int_{0}^{1} \tau  (1- \tau ^2) ^{-1/2} ({ \tau  +\eta})^{-1/2}\, d\tau , \quad \eta \in (0,1) ,
 $$
 and 
 \begin{equation}\label{theta2}
  \theta_2 (\eta):=  \int_{-\eta }^{0} \tau  (1- \tau ^2) ^{-1/2} ({ \tau  +\eta})^{-1/2}\, d\tau, \quad \eta \in (0,1) .
 \end{equation}
 It is clear that  $\theta_1$ is of class $C^1$ and decreasing. For the analysis of  $\theta_2$, we start by an integration by part which yields
 $$
  \theta_2 (\eta) =  -2 \int_{-\eta }^{0} ( \tau  (1- \tau ^2) ^{-1/2} )'({ \tau  +\eta})^{1/2}\, d\tau = -2   \int_{-\eta }^{0}  (1- \tau ^2) ^{-3/2}  ( \tau  +\eta)^{1/2}\, d\tau \,.
 $$
In this form we see that $\theta_2$ is of class $C^1$ and that its derivative is given by
 $$
   \theta_2' (\eta) = - \int_{-\eta }^{0}  (1- \tau ^2) ^{-3/2}  ( \tau  +\eta)^{-1/2}\, d\tau .
   $$
This yields $\theta_2' (\eta) <0$ for $\eta \in (0,1)$ so that $\theta_2$ is decreasing, and thus so is $\theta_+ =\theta_1 +\theta_2$.
Recalling from Lemma~\ref{l:F-theta} that $F (E) = -2^{1/2} E^{1/4} \theta_+(E^{-1/2})$ for $E \in (1,+\infty)$ we deduce that $F'<0$ on $(1,+\infty)$, which is Item~\ref{e:F-increasing}.

Concerning Item~\ref{i:limitFinfty}, we notice that $\theta_+$ is continuous at $0^+$ with $\theta_+ (0)= \int_{0}^{1} \tau^{1/2}   (1- \tau ^2) ^{-1/2}\, d\tau >0$. As a consequence, we deduce 
that  $F (E) \sim  -2^{1/2} E^{1/4} \theta_+ (0) \to -\infty$ as $E \to + \infty$.

Item~\ref{i:limitF1} is a direct consequence of Lemma~\ref{l:lemma-Fphi} since $F(E)=\mathsf{F}_{2-x^2}(E)$ (see for instance Equation~\eqref{e:to+-infty} with $\varphi(x)=2-x^2$).

Finally, Item~\ref{i:F-uniq-0} is a direct consequence of  Items~\ref{e:F-increasing}, \ref{i:limitFinfty} and~\ref{i:limitF1} (numerically $E_c \sim  2.35$).

 To examinate the function $F$ on the interval $(0,1)$, we rather study according to Lemma~\ref{l:F-theta} the function
  $$
  \theta_- (\eta) =  \int_{-1 }^{1} \tau  (1- \tau ^2) ^{-1/2} ({ \tau  +\eta})^{-1/2}\, d\tau , \quad \text{ for } \eta \in (1,+\infty).
 $$
 Splitting the integral and changing variables, we rewrite it as 
  $$
   \theta_- (\eta)=  \int_{0 }^{1} \tau  (1- \tau ^2) ^{-1/2} \left( ({ \eta+\tau})^{-1/2}-  ({   \eta-\tau })^{-1/2} \right)\, d\tau .
  $$
Noticing that $({ \eta+\tau})^{-1/2} <  ({   \eta-\tau })^{-1/2}$, this implies that $\theta_- (\eta)<0$ for $\eta \in (1,\infty)$, and from  Lemma~\ref{l:F-theta} we deduce that $F>0$ on $(0,1)$, whence Item~\ref{i:F>0}.
 \end{proof}

 Note that so far, we have proved in Corollary~\ref{cor3.4} that 
 $$
\alpha \to + \infty , \quad  \frac{\lambda(\alpha)}{\alpha^2} \to E \in [0,+\infty) \quad \implies \quad \frac{\lambda'(\alpha)}{\alpha} \to \Phi(E)= C(E)F(E)  
=\langle \mathfrak{m}_E , 2-x^2 \rangle.
 $$
 This is a useful piece of information as long as $\Phi(E)\neq 0$. We have already seen that $\Phi(E_c)=0$.
We have seen that $\Phi(0)=0$ so that the description of the measure $\mathfrak{m}_E$ is not enough to prove that $\lambda'$ does not vanish in this regime. In order to prove Theorem~\ref{th1.4}, an additional analysis is hence required near $E=0$, which is performed in the next section.

\subsection{Analysis below the energy $1$}\label{ss4.4}
The main goal of this subsection is to prove that $\lambda'(\alpha)$ does not vanish in the regime $\frac{\lambda(\alpha)}{\alpha^2} \to 0$. We prove a slightly more precise result.
\begin{prop}
\label{p:bottom}
Assume that $\alpha \to +\infty$ and $\frac{\lambda(\alpha)}{\alpha^2} \to E \in [0,1)$.  
Then, for all $\delta>0$, there is $\alpha_0 >0$ such that for all $\alpha \geq \alpha_0$, we have
$$
\lambda'(\alpha) \geq \left( \frac{4}{3\sqrt{2+2 \sqrt{E}} +\sqrt{2}}  - \delta \right)\alpha^{-{1/2}} .
$$
\end{prop}
In particular, there is no critical point in this regime. Note that the case $\frac{\lambda(\alpha)}{\alpha^2} \to 0$ was not covered yet, and Proposition~\ref{p:bottom} yields
$$
\alpha \to +\infty,  \quad \frac{\lambda(\alpha)}{\alpha^2} \to  0 \quad \implies \quad \lambda'(\alpha) \geq \left( \frac{1}{\sqrt{2}}  - \delta \right)\alpha^{-{1/2}}  .
$$
The vanishing rate of the derivative as $(2\alpha)^{-1/2}$ is consistent (and optimal) since it corresponds to that of the formal derivation of the low lying eigenvalue~\eqref{eq:asba}--\eqref{eq:asbah}.
Note that this asymptotic behavior is worse than in the case $\frac{\lambda(\alpha)}{\alpha^2} \to E \in (0,1)$ (in which case $\lambda'(E) \sim \Phi(E) \alpha$ as $\alpha \to +\infty$ according to Corollary~\ref{cor3.4}).

\begin{proof}
We start from Equation~\eqref{eq:ipf2}, which we reformulate in the semiclassical setting. 
Recalling that $v(\cdot,h)$ is defined in~\eqref{e:u-v}, we have $u(0,\alpha) = \alpha^{-1/4}v(0,\alpha^{-3/2})$ and  $\d_t u(0,\alpha) = \alpha^{-3/4} \d_x v(0,\alpha^{-3/2})$. Equation~\eqref{eq:ipf2} thus rewrites as 
\begin{multline*}
 2 \int_0^{+\infty} (t- \sqrt{2\alpha}) (t^2/ 2 - \alpha) \alpha^{-1/2} v(\alpha^{-1/2}t,\alpha^{-2/3})^2\, dt - \frac{\sqrt{\alpha}}{\sqrt{2}}  \lambda'(\alpha) \\
 =  (\lambda(\alpha) - \alpha^2) \alpha^{-1/2}v(0,\alpha^{-3/2})^2 +  \alpha^{-3/2} \d_x v(0,\alpha^{-3/2})^2 .
\end{multline*}
Changing variable $x=\alpha^{-1/2}t$ in the integral and recalling that $h=\alpha^{-3/2}$, we have obtained
\begin{align*}
2h^{-1}\int_0^\infty(x-\sqrt{2})(x^2/2-1) v(x,h)^2 dx  - \frac{h^{-1/3}}{\sqrt{2}} \lambda'(\alpha) =(\lambda(\alpha) - h^{-4/3}) h^{1/3}v(0,h)^2 +h \d_x v(0,h)^2 .
\end{align*}
We now recall that $v(\cdot,h)$ solves the semiclassical eigenvalue equation~\eqref{e:eigen-semi}, associated to a family of eigenvalues $E(h) :=\frac{\lambda(\alpha)}{\alpha^2} = h^{4/3} \lambda(\alpha)\to E <1$. 
Hence the point $x=0$ stands in the classically forbidden region $\R \setminus K_E$ at energy $E$.
According to Proposition~\ref{local} combined with rough Sobolev embeddings, we thus have 
$$
h \d_x v(0,h)^2 = O(h^\infty) , \quad (\lambda(\alpha) - h^{-4/3}) h^{1/3}v(0,h)^2 = O(h^\infty) .
$$
Hence, we have obtained, as $h =\alpha^{-3/2} \to 0^+$, 
\begin{align}
\lambda'(\alpha) & = \sqrt{2}h^{-2/3}\int_0^\infty(x-\sqrt{2})(x^2-2) v(x,h)^2 dx  +   O(h^\infty) \nonumber \\
& = \sqrt{2}h^{-2/3}\int_0^\infty \frac{(x+\sqrt{2})}{(x+\sqrt{2})}(x-\sqrt{2})(x^2-2) v(x,h)^2 dx  +   O(h^\infty) \nonumber \\
& = \sqrt{2}h^{-2/3}\int_0^\infty \frac{1}{(x+\sqrt{2})}(x^2-2)^2 v(x,h)^2 dx  +   O(h^\infty) \nonumber. 
\end{align}
Using again twice the localization in $K_E$ given by Proposition~\ref{local}, we obtain, for any $\eps>0$, 
\begin{align}
\label{e:estimate-1-1}
\lambda'(\alpha) & = 4 \sqrt{2} h^{-2/3}\int_{0}^{x_+(E)+\eps} \frac{1}{(x+\sqrt{2})}(x^2/2-1)^2 v(x,h)^2 dx  + O_\eps(h^\infty) \nonumber \\
& \geq  4 \sqrt{2} h^{-2/3}\frac{1}{(x_+(E)+\eps+\sqrt{2})} \int_{0}^{x_+(E)+\eps} (x^2/2-1)^2 v(x,h)^2 dx  + O_\eps(h^\infty) \nonumber \\
& \geq   \frac{4 \sqrt{2}}{(x_+(E)+\eps+\sqrt{2})} h^{-2/3}\int_0^\infty (x^2/2-1)^2 v(x,h)^2 dx  + O_\eps(h^\infty) .
\end{align}
We now rewrite \eqref{e:variation-eta} in the semiclassical setting using~\eqref{e:u-v} and $x=\alpha^{-1/2}t$ to obtain
\begin{align*}
h^{-2/3} \lambda'(\alpha)  + \lambda(\alpha) 
& = 3 \int_\R \left(\frac 12 t^{2} -\alpha \right)^2 u(\cdot,\alpha)^2 dt  . 
\end{align*}
Using the parity of $v(\cdot , h)^2$ for any eigenfunction, we deduce 
\begin{align*}
\lambda'(\alpha)  + h^{2/3} \lambda(\alpha) = 6 h^{-2/3}  \int_0^\infty (x^2/2-1)^2 v(x,h)^2 dx ,
\end{align*}
Combining this together with~\eqref{e:estimate-1-1} yields
$$
\lambda'(\alpha) \geq  \frac{2 \sqrt{2}}{3(x_+(E)+\eps+\sqrt{2})} \big( \lambda'(\alpha)  + h^{2/3} \lambda(\alpha) \big) + O_\eps(h^\infty) .
$$
We then notice that $\frac{2 \sqrt{2}}{3(x_+(E)+\eps+\sqrt{2})} <1$ since $x_+(E) \geq \sqrt{2} \geq 0$. We deduce that
\begin{align*}
\lambda'(\alpha) \geq \frac{2\sqrt{2}}{3(x_+(E)+\eps)+\sqrt{2}}   h^{2/3} \lambda(\alpha)  + O_\eps(h^\infty) =  \frac{2\sqrt{2}}{3(x_+(E)+\eps)+\sqrt{2}}  h^{-2/3}E(h) + O_\eps(h^\infty),
\end{align*}
after having recalled~\eqref{eq:corr1}. Finally, we always have $E(h)\geq E_1(h)$ where, according to the bottom of the well asymptotics~\eqref{eq:asbah} $E_1(h) \sim \sqrt{2}h$ as $h \to +\infty$. As a consequence, for all $\delta>0$, there exists $h_0>0$ such that for all $h \in (0,h_0)$, $\lambda'(\alpha) \geq \left( \frac{4}{3x_+(E)+\sqrt{2}}  - \delta \right)h^{1/3}$. This concludes the proof of the lemma using that  $h^{1/3}=\alpha^{-1/2}$ and recalling the explicit value of $x_+(E)$ in~\eqref{e:def-xpm}.
\end{proof}

    \subsection{End of the proof of Theorem~\ref{th1.4}}
Existence and uniqueness of $E_c$ were proven in Proposition~\ref{prop4.1}. 
Next,  according to Corollary~\ref{lemma3.5}, there is $E_m>0$ such that $ \alpha_{j,c} \to + \infty$ and $\lambda_j( \alpha_{j,c}) \leq E_m \alpha_{j,c}^2$ for any  $\alpha_{j,c} \in A_{j,c}$.
Up to extracting a subsequence, we may assume that $\frac{\lambda_j( \alpha_{j,c})}{\alpha_{j,c}^2} \to E \in [0,E_m]$ as $j \to + \infty$. 
We have $E \notin [0,1)$ since Proposition~\ref{p:bottom} would then imply $\lambda_j'( \alpha_{j,c})>0$. 
If $E \in (0,E_m]$, then Corollary~\ref{cor3.4} implies that $\alpha_{j,c}^{-1} \lambda_j'( \alpha_{j,c})  \to \Phi(E)$. If $E \neq E_c$, then $\Phi(E)>0$ and thus $\lambda_j'( \alpha_{j,c}) \sim \Phi(E)\alpha_{j,c}$ which contradicts $\lambda_j'( \alpha_{j,c}) =0$. As a consequence, we necessarily have $\frac{\lambda_j( \alpha_{j,c})}{\alpha_{j,c}^2} \to E_c$ (since this holds for any subsequence, this holds for the full sequence), which concludes the proof of Theorem~\ref{th1.4}.

    \section{Analysis beyond semiclassical measures and proof of Theorem~\ref{thcj}}
    \label{s5}   
    \subsection{Introduction}  
 According to the results of Section~\ref{s:semiclassique}, we have now identified the location of critical points of $\lambda_j$. In the semiclassical rescaling of the problem, the latter only arise (in the semiclassical limit) near the energy $E_c>1$. 
 The goal of the present section is to analyze the second derivative of $\lambda_j$ around these critical points, that is to say,  the semiclassical rescaling of the problem,  near the energy $E_c$.
 To this aim, we need to analyze precisely the right hand-side in (the semiclassical version of) the second variation formula~\eqref{eq:2.15a}.
 The latter however involves the derivative with respect to the semiclassical parameter $h$ of quantities like $(a v_h^j ,v_h^j)_{L^2}$ (or of the eigenfunctions themselves), where $a$ is a fixed polynomial and $v_h^j$ is an eigenfunction associated to the eigenvalue $E_j(h) \to E_c$. This piece of information is not encoded in the semiclassical measures studied in Section~\ref{s:semiclassique}.
 
 The main goal of the present section is thus essentially to give an explicit approximation $v_h^{j,\app}$ of the eigenfunction $v_h^j$ (associated to the eigenvalue $E_j(h) \to E_c$), that is differentiable with respect to $h$, and such that $\d_h v_h^{j,\app}$ is close to $\d_h v_h^{j}$ in the semiclassical limit.
   Luckily, the energy $E_c$ is a regular level set of the potential $(x^2/2-1)^2$, and we use this fact all along the proof.\\
   
In Subsection~\ref{s:class-quantities}, we therefore introduce a slightly more general setting of a 1D Schr\"odinger operator at a regular energy level. This allows to focus on the sole relevant properties of the potential $(x^2/2-1)^2$ at the energy $E_c$. We also introduce additional classical quantities which arise in the semiclassical limit.

Our analysis of eigenvalues near $E_c$ will rely on the Bohr-Sommerfeld quantization rules at this energy level, which we review in Subsection~\ref{s:bohr-sommerfeld}, with an emphasis on the labelling of the eigenvalues.

In Subsection~\ref{s:abstract-app}, we develop a rather systematic approach to compare approximate and exact eigenvalues/eigenfunctions. We assume there that approximate eigenvalues/eigenfunctions are constructed, and 
 give sufficient conditions on the approximate eigenfunctions $v_h^{j,\app}$ so that $\d_h v_h^{j,\app}$ is close to $\d_h v_h^{j}$. 

In Subsection~\ref{s:lagrangian-distrib}, we recall a classical way for constructing approximate eigenvalues/eigenfunctions (which eventually leads in particular to the Bohr-Sommerfeld quantization conditions discussed in Section~\ref{s:bohr-sommerfeld}) relying on WKB expansions or semiclassical Lagrangian Distributions.  This approach was initiated  by V.P. Maslov in a rather formal way but a complete mathematical proof can be found in the article by J. Duistermaat~\cite{Du} and in a more formalized way in the article by B. Candelpergher and J.C. Nosmas~\cite{CN}. For the analysis modulo $\mathcal O (h^\infty)$, we also refer to the arguments given in the appendix of Helffer-Robert~\cite{HMR} (itself relying on~\cite{Du} and~\cite{CN}). Here, we need to revisit this construction and control that we can differentiate with respect to the parameters.

 As mentioned to us by Didier Robert, we could also construct approximate eigenfunctions using coherent states instead of Lagrangian distributions.

\subsection{Setting and classical quantities}
\label{s:class-quantities}
Our  approach in this section involves refined semiclassical analysis at a non critical energy (namely $E_c$) for the semiclassical Schr\"odinger operator
 \begin{equation}\label{eq:5.18}
P_h = -h^2 \frac{d^2}{dx^2} +V(x) .
 \end{equation}
Throughout this section, we consider this operator and let $E_0 \in \R$ be an energy level satisfying the following assumption.
\begin{assumption}[Non-degeneracy at energy level $E_0$]
\label{a:nondeg}
The potential $V$ is $C^\infty$, real-valued and even. The classically allowed region at energy $E_0$, $K_{E_0}= V^{-1}((-\infty, E_0])$ is compact, connected and nondegenerate in the sense that $V'(x) \neq 0$ for all $x \in K_{E_0}$.
\end{assumption}
Note that the evenness of the potential is actually not needed for the results presented in this section, but it simplifies notation and proofs slightly.
Our goal is to apply the results of the present section to $V(x)=(x^2/2-1)^2$ and $E_0=E_c \approx 2.35 >1$ which satisfy Assumption~\ref{a:nondeg}. 
Note that if $(V,E_0)$ satisfy Assumption~\ref{a:nondeg}, then, there exists $\epsilon_0>0$ small such that $(V,E)$ also satisfy Assumption~\ref{a:nondeg} for all $E \in I_{E_0}(\epsilon_0)$, where we have denoted by 
\begin{align}
\label{e:I-eps}
I_{E_0}(\epsilon_0) :=(E_0-\epsilon_0, E_0+\epsilon_0)
\end{align}
such an interval of non-degenerate energies. Moreover, we have
$$
K_E = V^{-1}((-\infty, E]) = [-x_+(E), x_+(E)]\quad \text{ for all } E \in I_{E_0}(\epsilon_0) ,
$$
where $x_+(E)$ is the unique $x \in \R^+$ such that $V(x_+(E))=E$.
Note also that under Assumption~\ref{a:nondeg}, $P_h$ is essentially self-adjoint and from now on $P_h$ denotes this realization. 

\bigskip
To conclude these preliminaries, let us now introduce some classical quantities (in the sense that they arise from classical mechanics).
 For $\mu > \inf V$, we first introduce the energy
$$  
\mathcal E (\mu)= \frac{1}{2\pi} \int_{\xi^2 + V \leq \mu} dx d\xi . $$
Under Assumption~\ref{a:nondeg} and with the notation~\eqref{e:I-eps}, for all  $\mu \in I_{E_0}(\epsilon_0)$,   the energy $\mathcal E (\mu)$ is finite and rewrites as 
\begin{align}
\label{e:def-cal-E}
 \mathcal E (\mu)= \frac 1 \pi  \int_{- x_+(\mu)}^{x_+(\mu)} \sqrt{\mu-V(x)} dx\,.
 \end{align}
Still using Assumption~\ref{a:nondeg}, $\mu \mapsto  \mathcal E (\mu)$ is differentiable on $I_{E_0}(\epsilon_0)$ and satisfies 
$$
\mathcal E'(\mu)  =\frac{1}{2\pi} 
\int_{-x_+(\mu)}^{x_+(\mu)} \frac{dx}{\sqrt{\mu-V(x)}} \,.
$$
In particular, $\mathcal E$ is a smooth, strictly increasing function on $I_{E_0}(\epsilon_0)$.
Note that, in the case where  $V(x) = (x^2/2  -1)^2$, we have
\begin{equation}\label{eq:a5.1}
 \pi \mathcal E'(\mu)  = 
  C(\mu)^{-1}\,,
\end{equation}
where $C(\mu) = \left(\int_0^{x_+(\mu)}(\mu-V(x))^{-1/2}dx\right)^{-1}$ was introduced in~\eqref{eq:def-CE}.

\subsection{Bohr Sommerfeld condition and labelling of eigenvalues}
\label{s:bohr-sommerfeld}

\subsubsection{Spectrum and Bohr Sommerfeld condition}

We start with discussing eigenvalues of $P_h$ and their approximations, see e.g.~\cite{HeRo}. Eigenfunctions and their approximations are discussed later on (although these approximations are of course intimately related).

Under Assumption~\ref{a:nondeg} and with the notation~\eqref{e:I-eps}, the spectrum of $P_h$, denoted $\Sp(P_h)$ satisfies 
\begin{align}
\label{e:spectrum-Ej}
\Sp(P_h) \cap (-\infty, E_0+\epsilon_0)  &= \{E_k(h),1 \leq k \leq J(h)\} ,  \\  
\label{e:spectrum-Ej-bis}
\Sp(P_h) \cap I_{E_0}(\epsilon_0)& = \left\{E^j(h) , j \in \mathcal J (h)\right\} , 
\end{align}
where $J(h) \in \N$ and $\mathcal J(h) \subset \N$ is a finite set for all $h \in (0,h_0]$. Moreover, a Sturm-Liouville argument (as in Proposition~\ref{p:def-lambda}) shows that $E^j(h)$ has multiplicity one for all $j \in \mathcal J(h)$, and we thus assume that 
\begin{align*}
&E_k(h) < E_{k+1}(h), \quad \text{ for all }1 \leq k \leq J(h) ;\\
&E^j(h) < E^{j+1}(h), \quad \text{ for all } j \in \mathcal J (h) .
\end{align*}

Under Assumption~\ref{a:nondeg} and with the notation~\eqref{e:I-eps}, for all  $\mu \in I_{E_0}(\epsilon_1)$,  we consider
the pairs $(h,j) \in (0,h_0] \times \N$ such that the Bohr-Sommerfeld quantization condition
\begin{equation}\label{eq:quantiz1}
 \mathcal E (\mu) = (j+\frac 12) h\,.
\end{equation}
As proven for example in~\cite{HeRo} or in the appendix of \cite{HMR} (see also Theorem 6.3 and the following corollaries in \cite{CN}), there exists a family of functions $f_\ell  \in C^\infty(I_{E_0}(\epsilon_1))$ for $\ell \in \N$, such that the following holds.
There is a unique $E(j,h,\mu)\in \Sp(P_h)$ such that the following relation holds, for a sequence of $C^\infty$ functions $f_\ell$
\begin{equation}\label{eq:quantiz2}
E(j,h,\mu) = \mu + \sum_{\ell \geq 1} f_\ell (\mu) h^{\ell}+ \mathcal O (h^\infty) , \text{ uniformly for }(\mu,j,h) \text{ satisying~\eqref{eq:quantiz1}, } j \in \mathcal J(h) .
\end{equation}
Conversely any eigenvalue of $P_h$ in $I_{E_0}(\epsilon_0)$ 
can be obtained for $h$ small enough in this way. 
In other words, for $h\in (0,h_0]$ and for any $j \in \mathcal J(h)$ (that is to say, for any $E^j(h) \in \Sp(P_h) \cap I_{E_0}(\epsilon_0)$, see~\eqref{e:spectrum-Ej}), we have
\begin{equation}\label{formbis}
E^j(h) =  \mu_j(h)  + \sum_{\ell \geq 1} f_\ell (\mu_j(h)) h^{\ell} + \mathcal O(h^\infty) ,
\end{equation}
with  \begin{equation} \label{eq:quantrule1}
 \mu_j(h) = \mathcal E ^{-1} ((j+1/2) h).
 \end{equation} 
Note that, as a direct consequence of~\eqref{formbis}-\eqref{eq:quantrule1} and a Taylor expansion of $\mathcal E^{-1}$, the spectrum of $P_h$ in the energy window $I_{E_0}(\epsilon_0)$ satisfies a gap condition: there is $C,h_0>0$ such that for all $h \in (0,h_0]$ and $j \in \mathcal J(h)$,
 \begin{equation}\label{eq:5.21}
 d\left( E^j(h), \Sp(P_h)\setminus \{E^j(h)\} \right) \geq \frac 1C h.
 \end{equation}

\subsubsection{Labelling of eigenvalues}
\label{s:labelling}
So far, the set $\Sp(P_h) \cap I_{E_0}(\epsilon_0)$ is described in two different ways according to~\eqref{e:spectrum-Ej}--\eqref{e:spectrum-Ej-bis}, using respectively $E_k(h)$ and $E^j(h)$. 
Since eigenvalues are simple, there is for all $(j, h) \in \mathcal{J}(h)\times(0,1)$ a unique $\hat j = \hat j(j,h)$ such that 
 $$E^j(h) =E_{\hat j } (h) , $$ 
where $\hat j$ refers to the ``global labelling'' of the eigenvalue, that is to say $E_{\hat j } (h)$ is the $\hat j$'s eigenvalue of $P_h$. 
Note that in the case when $V$ has a unique non degenerate minimum without any other critical points below $E_0$ (which is not the case we are interested in in our application to $V(x)=(x^2/2-1)^2$), then we have $\hat j =j$, see~\cite{HeRo}. 
Without this assumption, using the asymptotic Weyl formula with remainder (see again~\cite{HeRo}) there are $C_0,h_0>0$ such that 
 \begin{equation}
 \left|\hat j(j,h)  -j \right| \leq C_0, \quad \text{for all }j \in \mathcal{J}(h) , h \in (0,h_0) ,
 \end{equation}
 that is to say, for all those $j$ such that $E^j(h) \in I_{E_0}(\epsilon_0)$.
 More recently, Y. Colin de Verdi\`ere showed in \cite[Theorem~3]{CdV} the following more precise result. 
 \begin{prop}
 \label{p:Ej=Ej}
Under Assumption~\ref{a:nondeg} and with the notation~\eqref{e:I-eps}, 
there exist $h_0>0$ and $\epsilon_1\in (0, \epsilon_0)$  such that  $\hat j  =j$ for any eigenvalue  $E^j(h)
 \in I_{E_0}(\epsilon_1)$ for all $h \in (0,h_0)$.
 \end{prop}
In other words, the proposition says that $E^j(h)=E_j(h)$ when $E^j(h)$ is close to $E_0$.

 \begin{proof}  
 We just sketch the main lines of the proof since no detail is given in \cite{CdV}. The idea is to construct a deformation of the potential $V$ onto a harmonic potential for which the spectrum is explicit. 
 We introduce, for $t\in [0,1]$,
  $$
  W_t(x)  = (1-t) V(x) + t R_0 x^2, \quad\text{with} \quad R_0= E_0 \,  x_+(E_0)^{-2} .
  $$
  With this choice, we have $W_0=V, W_1=R_0x^2$ and $W_t(x_+(E_0))= E_0$ for all $t \in [0,1]$. Moreover, $W_t$ satisfies the same Assumption~\ref{a:nondeg}, in which the nondegeneracy is uniform with respect to $t\in [0,1]$. 
  We finally define 
    $$
  a(t) = \frac{1}{2\pi} \int_{\xi^2 + W_t(x) \leq E_0} dx d\xi  =  \frac{1}{\pi} \int_{-x_+(E_0)}^{x_+(E_0)} \sqrt{E_0- W_t(x)} dx ,
  $$
  where $x_+(E_0)$ is for all $t\in[0,1]$ the unique $x\geq 0$ such that $W_t(x_+(E_0))=E_0$, and
  $$
  V_t(x):= W_t (a(t) x/a(0) )\,.
  $$
  With this choice we have $V_0(x)=V(x)$, $V_1(x)=R_0\frac{a(1)^2}{a(0)^2}x^2$, and 
  $$
  \mathcal E_t (E_0):=  \frac{1}{2\pi} \int_{\xi^2 + V_t(x) \leq E_0} dx d\xi 
  $$
  satisfies
\begin{align}
\label{e:E-t-0}
    \mathcal E_t (E_0) & =  \frac{1}{ \pi} \int_{-\frac{a(0)}{a(t)}x_+(E_0)}^{\frac{a(0)}{a(t)}x_+(E_0)} \sqrt{E_0 - V_t(x)} dx  = \frac{1}{ \pi} \int_{-\frac{a(0)}{a(t)}x_+(E_0)}^{\frac{a(0)}{a(t)}x_+(E_0)} \sqrt{E_0 - W_t \left(\frac{a(t)}{a(0)}x\right)} dx \nonumber  \\
    & =  \frac{a(0)}{a(t)}\frac{1}{ \pi} \int_{-x_+(E_0)}^{x_+(E_0)} \sqrt{E_0 - W_t \left(y\right)} dy= a(0)= \mathcal E (E_0)\,.
\end{align}
We now consider the spectra of the family of operators 
$$
P_{h,t} := -h^2\frac{d^2}{dx^2} + V_t(x) ,  \quad t \in [0,1] ,
$$
where $P_{h,0}=P_h$ and $P_{h,1} = -h^2\frac{d^2}{dx^2} + R_0\frac{a(1)^2}{a(0)^2}x^2$. In view of~\eqref{e:spectrum-Ej}--\eqref{e:spectrum-Ej-bis}, we write 
\begin{align*}
\Sp(P_{h,t}) \cap (-\infty, E_0+\epsilon_0)  &= \{E_k(h,t),1 \leq k \leq J(h,t)\} ,  \\  
\Sp(P_{h,t}) \cap I_{E_0}(\epsilon_0)& = \left\{E^j(h,t) , j \in \mathcal J (h,t)\right\} ,
\end{align*}
and since eigenvalues are simple, we may denote by $\hat j(j,h,t)$ the unique integer such that 
 $$E^j(h,t) =E_{\hat j(j,h,t)} (h,t).$$ 
 The simplicity of the spectrum also implies that the maps $t \mapsto E^j(h,t)$ and $t \mapsto E_k(h,t)$ are all continuous $[0,1] \to \R$ for any fixed $j,k,h$. 
   We can then apply the above Bohr-Sommerfeld theory~\eqref{formbis}--\eqref{eq:quantrule1}, uniformly with respect to $t$ for the eigenvalues close to $E_0$, namely 
  $$
  E^j(h,t) = \mu_j(h,t)   + h g_1 (\mu_j(h,t),t) + \mathcal O (h^2) , \quad \text{ with } \quad \mu_j(h,t) = \mathcal E_t ^{-1} ((j+1/2) h) ,
  $$
  for all $j,h,t$ such that $\mu_j(h,t) \in  I_{E_0}(\epsilon_0)$.
Using~\eqref{e:E-t-0}, there exists therefore $\epsilon_1<\epsilon_0$ and $h_0>0$ such that, defining $\mathcal{J}^{\epsilon_1}(h) \subset \mathcal J(h,1)$ by
$$
\Sp(P_{h,1}) \cap I_{E_0}(\epsilon_1) = \left\{E^j(h,1) , j \in \mathcal{J}^{\epsilon_1}(h) \right\} ,
$$
we have
$$
  E^j(h,t) \in I_{E_0}(\epsilon_0) \quad \text{ for all }  h \in (0,h_0) , t \in [0,1] ,j \in \mathcal{J}^{\epsilon_1}(h) . 
$$
 As a consequence, for $h \in (0,h_0)$ and $j \in \mathcal{J}^{\epsilon_1}(h)$, the map 
$$
t \mapsto \hat j(j,h,t) - j , \quad  [0,1] \to  \N
$$
is continuous. Hence we have $\hat j (j,h,t) -j$ constant for $t\in [0,1]$. But for $t=1$ we are considering the harmonic oscillator $P_{h,1}$, for which we know that $\hat j (j,h,1)=j$.
This implies that $\hat j (j,h,0)=j$ for the operator $P_{h,0}=P_h$.
 \end{proof}

   \subsection{Functional analysis: comparison between approximate eigenfunctions and eigenfunctions}
   \label{s:abstract-app}
In this subsection, we compare approximate eigenvalues and eigenfunctions with exact ones, and investigate their derivatives with respect to the semiclassical parameter $h$ (having in mind the formula~\eqref{eq:2.15a} for the second variation).
  The approach here is ``abstract'' in the sense that we do not construct the approximate eigenvalues/eigenfunctions but rather assume their existence.
  A way of constructing the latter approximate eigenvalues/eigenfunctions is reviewed in Subsection~\ref{s:lagrangian-distrib} below.
  \begin{prop}
  \label{l:appeig}
 We consider the operator $P_h$ in~\eqref{eq:5.18} under Assumption~\ref{a:nondeg} and with the notation~\eqref{e:I-eps}. Assume further that, for $C_0>0,k >1,$ and $\epsilon_1\in (0,\epsilon_0)$ we have $E^{\app}(h) \in I_{E_0}(\epsilon_1)$, $v_h^{\app} \in D(P_h)$ (resp. real-valued) and $r_h \in L^2(\R)$ satisfying
 \begin{equation}
 \label{e:asspt-app}
 P_h v_h^{\app}  =E^{\app}(h) v_h^{\app} + r_h, \quad \|v_h^{\app}\|_{L^2(\R)} = 1 , \quad \|r_h\|_{L^2(\R)} \leq C_0h^k  .
 \end{equation}
Then, there exist $C, h_0>0$ and $E(h)\in \Sp(P_h) \cap I_{E_0}(\epsilon_1)$ such that 
\begin{align}
\label{e:eig-close-app}|E(h) -E^{\app}(h)| \leq C_0h^k,
\end{align} and $v_h \in D(P_h)$ (resp. real-valued) such that for all $h \in (0,h_0)$,
 \begin{equation}
 \label{e:asspt-app}
 P_h v_h   =E(h) v_h , \quad \|v_h\|_{L^2(\R)} = 1 , \quad \|v_h - v_h^{\app}\|_{L^2(\R)} \leq Ch^{k-1}  .
 \end{equation}
 \end{prop}
  Notice that, since $P_h$ has real coefficients, we may always assume that $v_h^{\app}$ is real valued (taking otherwise its real part or imaginary part and renormalizing).
 The proof is rather classical but we recall it for the sake of completeness.
 Note also that, according to the gap condition~\eqref{eq:5.21} and the assumption $k>1$, there is at most one eigenvalue of $P_h$ satisfying~\eqref{e:eig-close-app} for $h$ small enough.
  \begin{proof}
 According to~\eqref{e:asspt-app}, we have $\| (P_h-E^{\app}(h)) v_h^{\app}\|_{L^2(\R)} = \| r_h\|_{L^2(\R)}\leq C_0h^k  \|v_h^{\app}\|_{L^2(\R)}$ and thus 
 $\| (P_h-E^{\app}(h))^{-1}\|_{L^2(\R) \to L^2(\R)} \geq (C_0h^k)^{-1}$ (where $\| (P_h-E^{\app}(h))^{-1}\|_{L^2(\R)\to L^2(\R)} \in (0,+\infty]$). Since the operator $P_h$ is selfadjoint, the spectral theorem yields that if $z \in \C \setminus \Sp(P_h)$, we have  $\| (P_h-E^{\app}(h))^{-1}\|_{L^2(\R) \to L^2(\R)} = \frac{1}{d(z,\Sp(P_h))}$. As a consequence, we have obtained that either $E^{\app}(h) \in \Sp(P_h)$ or $\frac{1}{d(E^{\app}(h),\Sp(P_h))}\geq (Ch^k)^{-1}$. In any case, we deduce that 
 $$
 d(E^{\app}(h),\Sp(P_h)) \leq C_0h^k .
 $$
 Given that $E^{\app}(h) \in I_{E_0}(\epsilon_1)$, that $C_0h^k\to0$ and that $\Sp(P_h)\cap I_{E_0}(\epsilon_0)$ consists only in eigenvalues (see~\eqref{e:spectrum-Ej} as consequence of Assumption~\ref{a:nondeg}), there exists $E(h) \in \Sp(P_h)\cap I_{E_0}(\epsilon_0)$ with~\eqref{e:eig-close-app} and $\tilde{v}_h \in D(P_h)$ satisfying $\|\tilde{v}_h\|_{L^2(\R)}=1$ such that 
  \begin{equation} 
P_h \tilde{v}_h  =E(h) \tilde{v}_h .
 \end{equation}
 We may assume that $\tilde{v}_h$ is real-valued since $P_h$ has real coefficients.
 We denote by $\Pi_{E(h)}$ the orthogonal projection in $L^2(\R)$ onto $\ker(P_h - E(h))$, which, on account to the simplicity of the spectrum may be written $\Pi_{E(h)}w = (w,\tilde{v}_h)_{L^2(\R)}\tilde{v}_h$. 
 We have $\Pi_{E(h)}v_h^{\app}$ real-valued if $\tilde{v}_h$ and $v_h^{\app}$ are, and
 $$
 (P_h -E(h))\Pi_{E(h)}v_h^{\app} =0 .
 $$
Moreover, we denote by $P_h = \int_{\R} \lambda d\mathsf{P}(\lambda)$ the spectral resolution of the operator $P_h$ and remark that the projector-valued measure $d\mathsf{P}$ is a linear combination of Dirac masses on $I_{E_0}(\epsilon_0)$ (the spectrum is discrete on $I_{E_0}(\epsilon_0)$). We thus have 
 \begin{align*}
 \left\| (1-\Pi_{E(h)}) v_h^{\app} \right\|_{L^2(\R)}^2 & = \int_{\R\setminus \{E(h)\}}  d \left(\mathsf{P}(\lambda)v_h^{\app}, v_h^{\app}\right)_{L^2(\R)} \\
 & = \int_{\R\setminus \{E(h)\}}  \frac{(\lambda - E(h))^2}{(\lambda - E(h))^2}d \left(\mathsf{P}(\lambda)v_h^{\app}, v_h^{\app}\right)_{L^2(\R)} .
 \end{align*}
 Recalling the gap property~\eqref{eq:5.21} (consequence of  Assumption~\ref{a:nondeg}), we deduce 
  \begin{align} 
  \label{e:toto-1}
 \left\| (1-\Pi_{E(h)}) v_h^{\app} \right\|_{L^2(\R)}^2 & \leq  \frac{1}{d(E(h), \Sp(P_h)\setminus E(h))^2} \int_{\R\setminus \{E(h)\}}  (\lambda - E(h))^2 d \left(\mathsf{P}(\lambda)v_h^{\app}, v_h^{\app}\right)_{L^2(\R)}\nonumber \\
 & \leq \frac{C^2}{h^2} \int_{\R}  (\lambda - E(h))^2 d \left(\mathsf{P}(\lambda)v_h^{\app}, v_h^{\app}\right)_{L^2(\R)} \nonumber  \\
 &\leq  \frac{C^2}{h^2}  \left\| (P_h-E(h)) v_h^{\app} \right\|_{L^2(\R)}^2 \nonumber  \\
 & \leq  \frac{C^2}{h^2} 2 \left\| (P_h-E^{\app}(h)) v_h^{\app} \right\|_{L^2(\R)}^2 +  \frac{C^2}{h^2} 2|E(h)-E^{\app}(h)|^2\left\|v_h^{\app} \right\|_{L^2(\R)}^2\nonumber \\
 & \leq Ch^{2k-2} , 
 \end{align}
 according to~\eqref{e:asspt-app}--\eqref{e:eig-close-app}. As a consequence, we have 
   \begin{align} 
  \label{e:toto-2}
 c_h^2 :=  \|\Pi_{E(h)} v_h^{\app}\|_{L^2(\R)}^2 = \|v_h^{\app}\|_{L^2(\R)}^2 - \|(1-\Pi_{E(h)}) v_h^{\app}\|_{L^2(\R)}^2 \geq 1 - Ch^{2k-2} .
 \end{align}
 We now set $v_h := c_h^{-1}\Pi_{E(h)} v_h^{\app}$, which is a normalized eigenfunction satisfying, according to~\eqref{e:toto-1}--\eqref{e:toto-2},
 \begin{align*}
  \|v_h - v_h^{\app}\|_{L^2(\R)}^2 & =c_h^{-2} \| \Pi_{E(h)} v_h^{\app} -c_h v_h^{\app}\|_{L^2(\R)}^2 \\
  &  \leq  2c_h^{-2} \|\Pi_{E(h)} v_h^{\app} - v_h^{\app}\|_{L^2(\R)}^2 + 2c_h^{-2}(1-c_h ) \| v_h^{\app}\|_{L^2(\R)}^2  \\
  & \leq 4 \frac{ Ch^{2k-2}}{1- Ch^{2k-2}} .
 \end{align*}
This concludes the proof of the lemma.
 \end{proof}

 We have now to prove that, under some additional assumptions, derivatives with respect to $h$ of the approximate eigenvalues/eigenfunctions are close to derivatives with respect to $h$ of the associated real eigenvalues/eigenfunctions.

 \begin{prop}
 \label{p:appeig-asspts}
  We consider the operator $P_h$ in~\eqref{eq:5.18} under Assumption~\ref{a:nondeg} and with the notation~\eqref{e:I-eps}. 
    Let $\epsilon_1 \in (0,\epsilon_0)$ and
  $$I_j :=\{h\in (0,1), \mathcal E ^{-1}( (j+1/2)h) \in I_{E_0}(\epsilon_1)\}, \quad j \in \N^*.$$
  Then, for all $j \in \N^*$, the map 
$$I_j \ni  h \mapsto E^j(h) \in \Sp(P_h)$$ 
belongs to $C^1(I_j)$ and for any solution to 
   \begin{equation} \label{eq:5.19}
P_h v_h^j  =E^j(h) v_h^j, \quad  h\in I_j
 \end{equation}
the map $h \mapsto v_h^j$ belongs to $C^1(I_j; \mathcal{S}(\R))$.

Assume further that we are given $ E^{j,\app} (h) \in C^1(I_j)$ and $v_h^{j,\app} \in C^1(I_j; \mathcal{S}(\R))$ such that 
  \begin{equation} 
  \label{e:vp-proches}
 E^{j,\app} (h)-E^j(h)= \mathcal O (h^\infty) \quad  \mbox{ for } h\in I_j ,
 \end{equation}
 \begin{equation}\label{eq:5.27}
 \| v_h^j-v^{j,\app}_h \|_{L^2(\R)}=\mathcal O (h^\infty) \quad  \mbox{ for } h\in I_j ,
 \end{equation}
 \begin{equation}\label{eq:5.20}
 P_h v_h^{j,\app}  =E^{j,\app} (h) v_h^{j,\app} + r_h^j  \quad  \mbox{ for } h\in I_j ,
 \end{equation}
\begin{equation}\label{eq:5.24}
\| v_h^j\|_{L^2}=1\,, \quad  \|v_h^{j,\app}\|_{L^2}  =1 ,
\end{equation}
 \begin{equation}\label{eq:5.25}
 \| r_h^j\|_{L^2} = \mathcal O (h^\infty) \quad \mbox{ for } h\in I_j ,
 \end{equation}
 \begin{equation}\label{eq:5.28}
 \| \partial_h r_h^j\|_{L^2} = \mathcal O (h^\infty) \quad  \mbox{ for } h\in I_j ,
 \end{equation}
 \begin{equation}\label{eq:5.33}
\text{there exists } k\in \mathbb N \text{ such that }
\| \partial_h v_h^{j,\app}\|_{L^2(\R)} =\mathcal O (h^{-k}) \quad   \mbox{ for } h\in I_j  ,
\end{equation}
   \begin{equation}
   \label{e:asspt-cpct}
   \mbox{there exists a compact set } K \subset \R \mbox{ such that } \supp v_h^{j,\app}  \subset K \quad  \mbox{ for }  h\in I_j , j \in \N^* ,
   \end{equation}
   Then, we have
     \begin{equation}
  \label{e:eignvalue-close}
    (E^{j,\app} )'(h)  - (E^j)'(h) =\mathcal O (h^\infty) \quad  \mbox{ uniformly for } h\in I_j , j \in \N^* .
  \end{equation}
   \begin{equation}
 \label{e:deriv-close}
\|  \partial_h v_h^j -\partial_h v_h^{j,\app} \|_{L^2(\R)}= \mathcal O (h^\infty)  \quad \mbox{ uniformly for } h\in I_j\,.
 \end{equation}
 \end{prop}
In the statement of the proposition, we keep the subscript $j$ to remember that, for eigenvalues $E^j(h)$ in the energy window $I_{E_0}(\epsilon_1)$, the product $jh$ belongs to a fixed interval and thus $j \to +\infty$ simultaneously as $h \to 0^+$.
We will however drop the dependence with respect to $j$ in the proof since this statement only concerns a fixed couple of eigenvalue/eigenfunction and approximate eigenvalue/eigenfunction.
  
  Note that the existence of $E^j(h),v_h^j$ such that~\eqref{e:vp-proches} and~\eqref{eq:5.27} are satisfied is actually a consequence of Lemma~\ref{l:appeig} once~\eqref{eq:5.20}--\eqref{eq:5.24}--\eqref{eq:5.25} are satisfied.
  
  Note that in our application we can choose $K$ in~\eqref{e:asspt-cpct} to be any compact neighborhood of $K_{E_0+\epsilon_0} =V^{-1}((-\infty, E_0 +\epsilon_0] )$.

Note finally that this proposition is written assuming remainder estimates in $\mathcal O (h^\infty)$. In applications such $\mathcal O (h^\infty)$ are issued from Borel summation providing differentiable remainders, and in which one can always impose differentiability with respect to parameters.
One could also write an analogue of this proposition with $\mathcal O (h^N)$ remainders, keeping track of the size $N$ but we choose not to do it for the sake of the presentation.

 \begin{proof}
First notice that eigenvalues $E(h)$ are simple from Sturm-Liouville theory and are thus $C^\infty$ with respect to $h\in I_j$. That $h \mapsto v_h$ belongs to $C^1(I_j; \mathcal{S}(\R))$ follows from the following computations.
 In the interval $I_j$, we differentiate~\eqref{eq:5.19} and~\eqref{eq:5.20} with respect to $h$ and obtain 
 \begin{equation} \label{eq:5.29}
 (-h^2 \d_x^2 + V - E(h) ) \partial_h v_h  =E'(h) v_h + 2h \d_x^2 v_h \,,
 \end{equation}
 and
  \begin{equation}\label{eq:5.30}
 (-h^2 \d_x^2 + V-E^{\app} (h)) \partial_h v_h^{\app}  = (E^{\app} )'(h)  v_h^{\app} + \partial_h r_h  + 2h \d_x^2 v_h^{\app} \,.
 \end{equation}
We  first compare $E'(h)$ and $ (E^{\app} )'(h)$ and then compare $\partial_h v_h$ and $\partial_h v_h^{\app}$.
Differentiating~\eqref{eq:5.24}, we obtain
\begin{align}
\label{e:h-orth}
 \left(v_h,\d_hv_h\right)_{L^2(\R)}= 0,\qquad \left(v_h^{\app},\d_hv_h^{\app}\right)_{L^2(\R)}=0.
\end{align}
Taking the scalar product with $v_h$ in \eqref{eq:5.29} and using~\eqref{e:h-orth}, selfadjointness of $P_h$ and~\eqref{eq:5.19}, we obtain
 \begin{equation}\label{eq:5.31}
   E'(h) = - 2h \left( v_h,\d_x^2 v_h\right)_{L^2(\R)}\,,
 \end{equation}
 and taking the scalar product with $v_h^{\app}$ in \eqref{eq:5.30}  and using~\eqref{e:h-orth} selfadjointness of $P_h$ and~\eqref{eq:5.20}, we obtain 
\begin{equation}\label{eq:5.32}
 (E^{\app} )'(h) = - 2h \left( v_h^{\app} ,\d_x^2 v_h^{\app} \right)_{L^2(\R)} - \left(\partial_h r_h , v_h^{\app} \right)_{L^2(\R)} + \left( r_h,\partial_h v_h^{\app}\right)_{L^2(\R)} \,. 
\end{equation}
 Substracting  \eqref{eq:5.31} to \eqref{eq:5.32} we deduce
\begin{align}
\label{e:Eprime-close}
 | (E^{\app} )'(h)  - E'(h)| & \leq 2h \left|  \left( v_h,\d_x^2 v_h\right)_{L^2(\R)} - \left( v_h^{\app} ,\d_x^2 v_h^{\app} \right)_{L^2(\R)} \right| \nonumber\\
 & \quad +  \left| \left(\partial_h r_h , v_h^{\app} \right)_{L^2(\R)}\right|  + \left| \left( r_h,\partial_h v_h^{\app}\right)_{L^2(\R)}\right| .
 \end{align}
 The last two terms in the right hand side are $\mathcal O (h^\infty)$ according to Assumptions~\eqref{eq:5.24},~\eqref{eq:5.25},~\eqref{eq:5.28} and~\eqref{eq:5.33}, respectively on $v_h^{\app}$, $r_h$, $\partial_h r_h$ and $\partial_h v_h^{\app}$. For the first term on the right hand-side of~\eqref{e:Eprime-close}, we write
\begin{align}
\label{e:control-cont-app}
   \left( v_h,\d_x^2 v_h\right)_{L^2(\R)} - \left( v_h^{\app} ,\d_x^2 v_h^{\app} \right)_{L^2(\R)} & =  \left( v_h-v_h^{\app} ,\d_x^2 v_h\right)_{L^2(\R)} + \left(v_h^{\app} , \d_x^2 v_h- \d_x^2 v_h^{\app}\right)_{L^2(\R)} \nonumber \\
   & = \left( v_h-v_h^{\app} ,\d_x^2 v_h + \d_x^2 v_h^{\app}\right)_{L^2(\R)} \,.
   \end{align}
   Then we will obtain again that this term is  $\mathcal O (h^\infty)$ by using~\eqref{eq:5.27}, once we have proved the following semiclassically temperate control of $\d_x^2 v_h$ and $\d_x^2 v_h^{\app}$.
   The first control can be deduced from  \eqref{eq:5.19}  that we rewrite in the form
   $$
   \d_x^2 v_h = h^{-2} V(x) v_h - h^{-2} E(h) v_h .
   $$
Letting $\tilde{K}$ be a compact neighborhood of $K_{E_0+\epsilon_0} =V^{-1}(]-\infty, E_0 +\epsilon_0] )$ and recalling that $E(h) \in I_{E_0}(\epsilon_0)$, we thus have  
     \begin{align*}
   \|\d_x^2 v_h\|_{L^2(\R)} & \leq  h^{-2} \| V v_h \|_{L^2(\R)} + h^{-2} |E(h)| \|v_h\|_{L^2(\R)} \\ 
   &  \leq C h^{-2} \| v_h \|_{L^2(\tilde{K})} + h^{-2} \| V v_h \|_{L^2(\R \setminus \tilde{K})} + C h^{-2}  \|v_h\|_{L^2(\R)} .
   \end{align*}
Using semiclassical control of $v_h$ in the classically forbidden region, see Proposition~\ref{local} together with $L^2$ normalization~\eqref{eq:5.24}, we obtain
\begin{align}
\label{eq:5.34-1}
   \|\d_x^2 v_h\|_{L^2(\R)} & \leq  C h^{-2} .
   \end{align}
   The estimate 
\begin{align}
\label{eq:5.34-2}
   \|\d_x^2 v_h^{\app}\|_{L^2(\R)} & \leq  C h^{-2} , 
   \end{align}
    follows the exact same arguments from the assumptions~\eqref{eq:5.20},~\eqref{eq:5.24}, \eqref{eq:5.25} and~\eqref{e:asspt-cpct}.
    The last two inequalities, combined to~\eqref{e:control-cont-app} and~\eqref{e:Eprime-close} implies~\eqref{e:eignvalue-close}.

 We now turn to the proof of~\eqref{e:deriv-close}. We first prove that 
   \begin{equation}\label{eq:5.37}
   V(x) (v_h-v_h^{\app}) =\mathcal{O}_{L^2(\R)} (h^\infty) \quad  \mbox{ and } \quad  \d_x^2 (v_h-v_h^{\app})  =\mathcal{O}_{L^2(\R)} (h^\infty)\,.
   \end{equation}
   Indeed, we first have (with $\tilde{K}$ a compact neighborhood of $K_{E_0+\epsilon_0}$)
   \begin{align}
   \|V (v_h-v_h^{\app}) \|_{L^2(\R)} & \leq    \|V (v_h-v_h^{\app}) \|_{L^2(K \cup \tilde{K})} + \|V (v_h-v_h^{\app}) \|_{L^2(\R \setminus (K \cup \tilde{K}))} \nonumber \\
   & \leq    C \|v_h-v_h^{\app}\|_{L^2(K \cup \tilde{K})} + \|V v_h \|_{L^2(\R \setminus (K \cup \tilde{K}))}  \nonumber \\
   & \leq    C \|v_h-v_h^{\app}\|_{L^2(\R)} + \mathcal{O}(h^\infty)  = \mathcal{O}(h^\infty) , \label{e:poly-growth}
   \end{align}
   where we have used~\eqref{e:asspt-cpct} in the second line, Proposition~\ref{local} together with~\eqref{eq:5.24} in the last inequality,  and~\eqref{eq:5.27} in the last equality.
   Then, substracting~\eqref{eq:5.20} to~\eqref{eq:5.19}, we obtain
   \begin{align*}
   -h^2\d_x^2 (v_h-v_h^{\app}) & = -V(v_h-v_h^{\app}) + E(h)v_h - E^{\app}(h)v_h^{\app} - r_h  \\
   & = -V(v_h-v_h^{\app}) + (E(h)-E^{\app}(h))v_h + E^{\app}(h)(v_h -v_h^{\app} )- r_h \\
   & = \mathcal{O}_{L^2(\R)}(h^\infty) ,
   \end{align*}
   after having used~\eqref{eq:5.27},~\eqref{e:vp-proches}, $E(h) \in I_{E_0}(\epsilon_0)$, together with the last estimate. This concludes the proof of~\eqref{eq:5.37}.
   Note also that we deduce from \eqref{eq:5.32} together with \eqref{eq:5.24}, \eqref{eq:5.25}, \eqref{eq:5.28}, \eqref{eq:5.33}, and \eqref{eq:5.34-2} the existence of $k \in \N$ such that
   \begin{equation}\label{eq:5.38}
    (E^{j,\app} )'(h)  = \mathcal O (h^{-k} )  \mbox{ uniformly for } h\in I_j\,.
   \end{equation}
   We now rewrite the difference of \eqref{eq:5.30} and \eqref{eq:5.29} under the form
    \begin{align} \label{eq:?}
 (P_h - E(h) ) (\partial_h v_h - \partial_h v_h^{\app})&  = (E(h)-E^{\app} (h)) \d_h v_h^{\app}+ 2h \d_x^2 ( v_h -v_h^{\app})    \nonumber \\
  & \quad + (E'(h)-(E^{\app})'(h) ) v_h  -(E^{\app})'(h) ( v_h^{\app} - v_h) 
   -  \partial_h r_h  \,.
 \end{align}
  Using~\eqref{e:vp-proches},~\eqref{eq:5.24},~\eqref{eq:5.27},~\eqref{eq:5.28},~\eqref{eq:5.37}, and~\eqref{eq:5.38}  this implies 
 \begin{equation}
 \label{e:approxequat}
 (P_h -E(h)) (\partial_h v_h -\partial_h v_h^{\app} ) = \mathcal{O}_{L^2(\R)} (h^\infty)\,.
 \end{equation}
Now recalling~\eqref{e:h-orth}, we deduce from \eqref{eq:5.27} and \eqref{eq:5.33} that 
 $$
\left( (\partial_h v_h -\partial_h v_h^{\app} ),  v_h\right)_{L^2(\R)} = \left(\partial_h v_h^{\app},  v_h\right)_{L^2(\R)}  =   \left(\partial_h v_h^{\app},  v_h -v_h^{\app} \right)_{L^2(\R)} = \mathcal O (h^\infty)\,.
 $$
 We now denote by $\Pi_{E(h)}$ the orthogonal projection in $L^2(\R)$ onto $\ker (P_h-E(h))$, which we may write $\Pi_{E(h)}w = (w,v_h)_{L^2(\R)} v_h$ according to the simplicity of the spectrum.
 On the one hand, the last identity implies $$\left\|\Pi_{E(h)} (\partial_h v_h -\partial_h v_h^{\app} ) \right\|_{L^2(\R)} = \mathcal O (h^\infty) .$$
 On the other hand, proceeding exactly as in~\eqref{e:toto-1} in the proof of Lemma~\ref{l:appeig}, using the approximate equation~\eqref{e:approxequat} together with the gap condition~\eqref{eq:5.21}, we obtain 
 $$\left\|(1 - \Pi_{E(h)}) (\partial_h v_h -\partial_h v_h^{\app} ) \right\|_{L^2(\R)} = \mathcal O (h^\infty) .$$
 Combining the last two identities concludes the proof of~\eqref{e:deriv-close}, and thus that of the proposition.
 \end{proof}

Recalling the expression of $\lambda_j''(\alpha)$ in the semiclassical scaling in Lemma~\ref{l:derivee-seconde-h}, given a $C^\infty$ real-valued function $a$ with polynomial growth, we are now interested in the asymptotic expansion of the quantity
 \begin{equation}
 \label{e:def-M}
 \mathfrak M_a ^j (h) := \int a(x) |v_h^j(x)|^2\,dx\,,\, \mbox{ for } h\in I_j\,,
 \end{equation}
 and its derivative with respect to $h$:
 \begin{equation}
( \mathfrak M_a ^j)' (h)= 2 \Re \int a(x)( \partial_h v_h^j)(x)\, \overline{v_h^j (x)} \,dx\,,\, \mbox{ for } h\in I_j\,,
 \end{equation}
We also define its approximate analogue \begin{equation} \label{e:def-Mapp}
 \mathfrak M_a ^{j,\app} (h) := \int a(x) |v_h^{j,\app}(x)|^2\,dx\,,\, \mbox{ for } h\in I_j\,,
 \end{equation}
 and obtain the following corollary of Proposition~\ref{l:appeig} (applied for all $k \in \N$) and Proposition~\ref{p:appeig-asspts}.
 \begin{cor}
 \label{c:mathfrak-M}
 Under the assumptions of Proposition~\ref{p:appeig-asspts}, we have
 \begin{align}\label{eq:5.45}
  \mathfrak M_a ^j (h) -  \mathfrak M_a ^{j,\app} (h) &=\mathcal O (h^{\infty}) \mbox{ and } \\ 
  ( \mathfrak M_a ^j )'(h) -  (\mathfrak M_a ^{j,\app})' (h)& =\mathcal O (h^\infty),  \mbox{ uniformly for } h\in I_j , \label{e:mathkfrak-M-prime}
  \end{align}
  for any $a \in C^\infty(\R)$ with polynomial growth.
\end{cor}
\begin{proof}
The validity of~\eqref{eq:5.45} for $a \in C^\infty(\R)$ with polynomial growth follows from the same decomposition as in~\eqref{e:poly-growth}.
Concerning~\eqref{e:mathkfrak-M-prime}, we decompose 
\begin{align*}
  ( \mathfrak M_a ^j )'(h) -  (\mathfrak M_a^{j,\app})' (h)& = 2 \Re \int a(x)( \partial_h v_h^j)(x)\, \overline{v_h^j (x)}dx - 2\Re \int a(x)( \partial_h v_h^{j,\app})(x)\, \overline{v_h^{j,\app}(x)}dx  \\
  &  = 2\Re \left( a v_h^j , \partial_h v_h^j-\partial_h v_h^{j,\app} \right)_{L^2}+  2\Re \left( \partial_h v_h^{j,\app}  , a( v_h^j- v_h^{j,\app} )\right)_{L^2}
\end{align*}
whence 
\begin{align*}
 \left| ( \mathfrak M_a ^j )'(h) -  (\mathfrak M_a^{j,\app})' (h) \right| 
 \leq  2 \left\| a v_h^j \right\|_{L^2}\left\| \partial_h v_h^j-\partial_h v_h^{j,\app}\right\|_{L^2}+2 \left\| \partial_h v_h^{j,\app} \right\|_{L^2}  \left\| a( v_h^j- v_h^{j,\app} )\right\|_{L^2}
\end{align*}
According to~\eqref{eq:5.33} and the same decomposition as in~\eqref{e:poly-growth}, we have $$\left\| \partial_h v_h^{j,\app} \right\|_{L^2}  \left\| a( v_h^j- v_h^{j,\app} )\right\|_{L^2} =\mathcal{O}(h^\infty).$$
And according to  Proposition~\ref{local} together with~\eqref{e:deriv-close}, we also have 
$$
\left\| a v_h^j \right\|_{L^2}\left\| \partial_h v_h^j-\partial_h v_h^{j,\app}\right\|_{L^2} =\mathcal{O}(h^\infty).$$
The last three statements conclude the proof of~\eqref{e:mathkfrak-M-prime} and of the corollary.
\end{proof}

  \subsection{Semiclassical Lagrangian distributions revisited}
  \label{s:lagrangian-distrib}

In this subsection, we recall a classical way of constructing approximate eigenfunctions for 1D Schr\"odinger operators at a regular energy level $E_0$, through WKB expansions or so-called semiclassical Lagrangian distributions. We briefly review such a construction and check that one can differentiate with respect to parameters (in particular $h$). We then consider product of two such objects, 
    
\subsubsection{Semiclassical Lagrangian distributions}
We only need part of the local theory of Lagrangian distributions. We refer to the first section of \cite{Du} (p. 209-232), to~\cite{CN} or~\cite{Dy} (see also~\cite[Section~25.1]{Hor-4} in the non-semiclassical homogeneous setting).
Given a smooth family of compact Lagrangian submanifolds $\mu \mapsto \Lambda_\mu \subset T^*\R^n=\R^{2n}$, with $\mu$ a parameter in a small interval $\mu \in I_{E_0}(\epsilon_0)=(E_0-\epsilon_0,E_0+\epsilon_0)$, a semiclassical Lagrangian distribution (or Lagrangian function, or Lagrangian state) associated to $\Lambda_\mu$ is a linear combination of integrals in the form (cf (1.2.1) in \cite{Du} or \cite{CN})
\begin{equation}\label{eq:intosc}
I(x,\mu,h) = (2\pi h)^{-k/2} \int_{U} e^{i  \varphi (x,\theta,\mu)/h} a(x,\theta, \mu,h) d\theta \,,
\end{equation}
where \begin{itemize}
\item  $k \in \N$, $U\subset \R^k$ is an open set,
\item  $\varphi \in C^\infty(\R^n\times U\times \R)$ is a real valued phase-function associated to the Lagrangian $\Lambda_\mu$, that is to say, such that, 
on the critical set (depending on $\mu$) $C_{\varphi, \mu} := \{(x,\theta), d_\theta \varphi(x,\theta,\mu)=0\}$, the differentials $d_x(\d_{\theta_1}\varphi) ,\cdots ,d_x(\d_{\theta_k}\varphi)$ are linearly independent, and 
$$\{(x, d_x\varphi(x,\theta, \mu)), (x, \theta)\in C_{\varphi,\mu}\} \subset \Lambda_\mu.$$
\item 
$a$ is in $C^\infty(\R^n\times U \times I_{E_0}(\epsilon_0) \times (0,h_0))$, uniformly compactly supported in $(x,\theta) \in \R^n\times U$ and admitting (as well as its derivatives) an expansion in the form
\begin{equation}
\label{e:expansion}
\d_{(x,\theta,\mu)}^\beta a(x,\theta,\mu,h)\sim \sum_{r=0}^{+\infty} a_{r,\beta}(x,\theta,\mu) h^r\,.
\end{equation}
\end{itemize}
Here, $\mu \in \R$ is an energy parameter and the dependence on $\mu$ is $C^\infty$. Note that the asymptotic expansion for derivatives of $a$ is not always explicit in most references (e.g.~\cite{Du}). Note also that, from the Borel summation process, the expansion~\eqref{e:expansion} is also differentiable with respect to $h$ and $\mu$.
The case $k=0$ in~\eqref{eq:intosc} simply corresponds to $I(x,h,\mu) = e^{i  \varphi (x,\mu)/h} a(x, \mu,h)$, whereas, in general,~\eqref{eq:intosc} has to be understood in the sense of oscillatory integrals, see~\cite{Du} or~\cite[Chapter~1]{GS:94}.
Note finally that, since $\Lambda_\mu$ is assumed to be compact here, any semiclassical Lagrangian distribution $I(x,\mu,h)$ as defined above is actually a smooth function of $(x,\mu,h) \in \R^n \times I_{\epsilon_0}(E_0)\times (0,h_0)$.
Another important feature is that, if $I(x,\mu,h)$ is a semiclassical Lagrangian distribution associated to $\Lambda_\mu$ and $a \in C^\infty(\R^n)$ (or more generally $a \in S^m(T^*\R^n)$), then $a(x)I(x,\mu,h)$ (or more generally $\Op_h(a)I(\cdot ,\mu,h)$ for any semiclassical quantization $\Op_h$, see~\cite{Zwo}) is as well a semiclassical Lagrangian distribution associated to $\Lambda_\mu$.

\subsubsection{Semiclassical Lagrangian distributions approximating eigenfunctions in dimension $1$}
In the present 1D context ($n=1$), under Assumption~\ref{a:nondeg} and with the notation~\eqref{e:I-eps}, the Lagrangian manifold is the energy level 
$$\Lambda_\mu = \{ (x,\xi) \in \R^2 , \xi^2 + V(x) =\mu\} , \quad \mu \in I_{E_0}(\epsilon_0).
$$
We denote $\pi_x : \R_x\times \R_\xi \to \R_x$ and $\pi_\xi : \R_x\times \R_\xi \to \R_\xi$ the canonical projections and notice that $\pi_x(\Lambda_\mu) = [-x_+(\mu),x_+(\mu)]$ (recall that Assumption~\ref{a:nondeg} includes that $V$ is even).
Notice that away from the two points $(\pm x_{+}(\mu), 0) \in \Lambda_\mu$, $\Lambda_\mu$ projects nicely on the $x$ variable and near these two points, $\Lambda_\mu$ projects nicely on the $\xi$ variable. 
Therefore (see e.g.~\cite[Exercise~12.3]{GS:94}), one can show that all semiclassical Lagrangian distributions/functions associated to $\Lambda_\mu$ write 
\begin{align}
& I(x,\mu,h) =  \sum_\pm a_\pm(x,\mu,h)e^{\pm \frac{i}{h}\varphi(x,\mu)} +(2\pi h)^{-1/2} \int_\R b_\pm(x,\xi,\mu,h) e^{\frac{i}{h}(x\xi \pm \psi(\xi))}d\xi \label{WKB-1} \\
& \varphi(x,\mu) = \int_0^x\sqrt{\mu-V(s)}ds , \quad \psi(\xi,\mu) = \int_0^\xi x_+(\mu-\zeta^2) d\zeta ,\label{WKB-2} 
\end{align}
with $\mu \in I_{E_0}(\epsilon_0)$, $\supp_x a_\pm \subset [-x_+(E_0-\epsilon_0), x_+(E_0-\epsilon_0)]$ and $\supp_\xi b_\pm$ being a fixed small set containing $0$. Note that, in these two sets respectively,
$$
\Lambda_\mu \cap (-x_+(E_0-\epsilon_0), x_+(E_0-\epsilon_0)) \cap \{\pm \xi >0\} =\{ (x, \pm \varphi'(x)) \} , \quad \text{ or } \quad \Lambda_\mu =\{ (\pm \psi'(\xi), \xi) \}  
$$
(where the last equality holds only locally).

\medskip
The next important feature we use is that, under the quantization relation \eqref{eq:quantrule1}, i.e. $\mu =\mu(h)=\mu_j(h)$, there exist $ E_j^{\app}(h)= \mu_j(h) + \mathcal{O}(h) \to E_0$ and amplitudes $a_\pm(x,\mu_j(h),h), b_\pm(x,\xi,\mu_j(h),h)$ satisfying~\eqref{e:expansion} such that, setting $$u_j^{\app} (x,h):=  c_h^{-1} I(x, \mu_j(h),h) , \quad c_h := \left\|  I(x, \mu_j(h),h)\right\|_{L^2(\R)} \to c_0 >0,
$$ where $I(x,\mu,h)$ is the Lagrangian function~\eqref{WKB-1} associated to these particular $a_\pm , b_\pm$, we have 
\begin{align}
\label{e:Pu=presq-0} &P_h u_j^{\app}  (x,h) = E_j^{\app}(h) u_j^{\app}  (x,h) + r(x,h) , \quad \text{with } \\
\label{e:Pu=presq-1} & \| u_j^{\app}(\cdot,h)\|_{L^2(\R)}=1, \quad  r(\cdot, h) = \mathcal{O}_{C^\infty}(h^\infty) ,\quad \d_h r(\cdot, h) = \mathcal{O}_{C^\infty}(h^\infty)  .
\end{align}
Note that this function is not real-valued, which is not an issue in the above results.

 \subsubsection{Product of semiclassical Lagrangian distributions}
 In this paragraph, we are interested in the $L^2$ inner product of two general Lagrangian distributions associated with the same Lagrangian submanifold.
 We recall that from Assumption~\ref{a:nondeg}, and Equation~\eqref{eq:a5.1}, we have $\mathcal E '(E_0) = (\pi C(E_0))^{-1}>0$.
\begin{prop} 
\label{e:product-SLD}
Assume that $I_1 (\cdot ,\mu,h) , I_2(\cdot ,\mu,h)$ are two semiclassical Lagrangian distributions associated to the same Lagrangian submanifold $\Lambda_\mu$ and define the quantity
 $$
K(h,\mu) = \left( I_1 (\cdot ,\mu,h) , I_2(\cdot ,\mu,h)\right)_{L^2(\R)} =  \int_\R I_1 (x,\mu,h) \overline{I_2(x,\mu,h)} dx\,.$$
Then, under Assumption~\ref{a:nondeg}, assuming $\mu(h) \to E_0$ and satisfies the quantization condition \eqref{eq:quantrule1}, we have
\begin{equation}
\label{e:dhK}
 h \partial_h ( K(h,\mu(h)) ) \to   k_0'(E_0)\mathcal E ( E_0)  \mathcal E '(E_0)^{-1}  , \quad \text{as } h \to 0^+ ,
\end{equation} 
where $k_0(E) :=\lim_{h \to 0}K(h,\mu(h))$ when $\mu(h)\to E$ under~\eqref{eq:quantrule1} is a smooth function of $E$ near $E=E_0$.
\end{prop}
\begin{proof}
 It is proven in \cite[p.~224]{Du}  that for $\mu=\mu(h)$, $K(h,\mu(h))$ has a complete expansion in the form (for any $L>0$) 
 \begin{equation}
 K(h,\mu(h))  =  \sum_{\ell =0}^{L} k_\ell (\mu(h)) h^\ell + r_L (h)\,,
 \end{equation} 
where $k_\ell \in C^\infty(I_{E_0}(\epsilon_1))$ and 
 \begin{align}
 \label{e:rest-L}
 r_L (h) = \mathcal O (h^{L+1}) \mbox{ and }  \partial_h r_L (h) = \mathcal O (h^{L})\,.
 \end{align}
 This follows of the inspection of the proof which is simply a stationary phase theorem with $\mu$ as a parameter, the oscillatory term disappearing when the quantization relation is satisfied.
 The possibility to differentiate with respect to the parameter $h$ or to the parameter $\mu$ when applying the stationary phase theorem is proved for example in \cite{FeMa} (See Part I, Section 1
  and particularly p. 31, lines 5-11).

 With this property, we get by a term by term differentiation
 \begin{equation*}
\partial_h ( K(h,\mu(h)) ) = \sum_{\ell =1}^{L} \ell k_\ell (\mu(h)) h^{\ell-1}  +  \sum_{\ell =0}^{L} k_\ell' (\mu(h))\mu'(h) h^{\ell} +  \partial_h r_L (h) \,.
\end{equation*}
Then, notice that the quantization relation \eqref{eq:quantrule1} implies 
 \begin{equation}\label{eq:deriv}
 \mathcal E '(\mu(h)) \mu'(h) = \mathcal E (\mu(h)) /h\,.
 \end{equation}
 Recalling that $\mathcal E '(\mu)>0$ for $\mu \in I_{E_0}(\epsilon_0)$, we deduce that 
 \begin{equation*}
\partial_h ( K(h,\mu(h)) ) = \sum_{\ell =1}^{L} \ell k_\ell (\mu(h)) h^{\ell-1}  +  \sum_{\ell =0}^{L} k_\ell' (\mu(h)) \mathcal E ( \mu(h)) \mathcal E '(\mu(h))^{-1} h^{\ell -1} +  \partial_h r_L (h) \,.
\end{equation*}
According to~\eqref{e:rest-L} and the fact that $\mu(h) \in I_{E_0}(\epsilon_0)$, ordering the terms in the two sums, this rewrites 
 \begin{equation*}
\partial_h ( K(h,\mu(h)) ) =h^{-1} k_0'(\mu(h))\mathcal E ( \mu(h)) \mathcal E '(\mu(h))^{-1}  + \mathcal{O}(1) , \quad \text{as } h \to 0^+ .
\end{equation*}
 In our application $\mu(h) \to E_0$ and we deduce~\eqref{e:dhK}.
 \end{proof}
 
 From the above results, we now deduce the following corollary.
 \begin{cor}
 \label{c:combine-tout}
Under Assumption~\ref{a:nondeg}, we assume $\mu(h) \to E_0$ and satisfies the quantization condition \eqref{eq:quantrule1}.
 For all $a \in C^\infty(\R)$ with polynomial growth, 
 \begin{align*}
 \mathfrak M_{a}^j (h) & \to \left< \mathfrak{m}_{E_0} , a \right>  , \quad \text{as } h \to 0^+ , \\
h \d_h \mathfrak M_{a}^j (h) &  \to \frac{d}{d\mu}\left(\left< \mathfrak{m}_\mu , a \right> \right) \Big|_{\mu = E_0}\mathcal E ( E_0 ) \mathcal E '(E_0)^{-1}  , \quad \text{as } h \to 0^+ ,
 \end{align*}
 where $\mathfrak M_{a}^j (h)$ is defined in~\eqref{e:def-M} and $\mathfrak{m}_\mu$ in~\eqref{mu-E=1-infty} .
 \end{cor}
 \begin{proof}
 Firstly, according to~ \eqref{e:Pu=presq-0}--\eqref{e:Pu=presq-1} above, under the quantization condition~\eqref{eq:quantrule1} $\mu(h)=\mu_j(h)$, there is an approximate eigenfunction $u_j^{\app}(\cdot,h)$ associated to the eigenvalue $E_j^{\app}(h)$, such that $u_j^{\app}(\cdot,h)$ is a semiclassical Lagrangian distribution associated to $\Lambda_{\mu(h)}$.
 
Secondly, as a consequence of Proposition~\ref{l:appeig}, the gap property~\eqref{eq:5.21} and the simplicity of the spectrum (following from the Sturm--Liouville theory), there is a unique eigenvalue $E_j(h)$ in an $\mathcal{O}(h)$ neighborhood of $E_j^{\app}(h)$, and we have $E_j(h)-E_j^{\app}(h)= \mathcal{O}(h^\infty)$. Moreover, still according to Proposition~\ref{l:appeig}, there is a (unique) normalized eigenfunction $u_j(\cdot,h)$ of $P_h$ associated to the eigenvalue $E_j(h)$ and such that $\|u_j(\cdot ,h)-u_j^{\app}(\cdot,h)\|_{L^2(\R)}=\mathcal{O}(h^\infty)$.

Thirdly, we notice that all assumptions of  Proposition~\ref{p:appeig-asspts} are satisfied by $u_j(\cdot,h),E_j(h),u_j^{\app}(\cdot,h), E_j^{\app}(h)$ (in particular as a consequence of the fact that semiclassical Lagrangian distributions are differentiable with respect to $h$ and yield a semiclassically temperate function).
Proposition~\ref{p:appeig-asspts} together with Equation~\eqref{e:mathkfrak-M-prime} in Corollary~\ref{c:mathfrak-M} then yield
 $$
   ( \mathfrak M_a ^j )'(h) -  (\mathfrak M_a ^{j,\app})' (h)  =\mathcal O (h^\infty) ,
 $$
 where these two quantities are defined in~\eqref{e:def-M}--\eqref{e:def-Mapp} respectively, with $v_h^j=u_j(\cdot,h)$ and $v_h^{j,\app}=u_j^{\app}(\cdot,h)$.
 
 Fourthly and lastly, it thus suffices to compute $(\mathfrak M_a ^{j,\app})' (h)$. To this aim, recalling that $u_h^{j,\app}$ and $a u_h^{j,\app}$ are two semiclassical Lagrangian distributions, we use Proposition~\ref{e:product-SLD} with $I_1 = a u_h^{j,\app}$ and $I_2 = u_h^{j,\app}$. Noticing that in this case, $k_0(\mu) = \left<\mathfrak{m}_\mu , a \right>$, Equation~\eqref{e:dhK} thus rewrites
 $$
h (\mathfrak M_a ^{j,\app})' (h)  = h \partial_h ( K(h,\mu_j(h)) ) \to \frac{d}{d\mu}(\left<\mathfrak{m}_\mu , a \right>)(\mu = E_0)\mathcal E ( E_0)  \mathcal E '(E_0)^{-1}  , \quad \text{as } h \to 0^+ ,
$$
which is the sought result.
 \end{proof}
 
\subsection{The second derivative: end of proof of Theorem~\ref{thcj}}

We now want to compute the asymptotics of $\lambda_j''(\alpha_{j,c})$. Using~\eqref{eq:2.15a}, $\lambda_j''$ rewrites as
\begin{align*}
\lambda_j''(\alpha) & =    (1 -  \frac32 h\, \d_h) \,  \left( \int_\R  (2-s^2)   |v_j(s,h)|^2 ds\right) , \quad h=\alpha^{-3/2} \\
& =     (1 -  \frac32 h\, \d_h)  \mathfrak M_{(2-s^2)}^j (h) ,
\end{align*}
where $\mathfrak M_{a}^j (h)$ is defined in~\eqref{e:def-M}. We now apply Corollary~\ref{c:combine-tout} with  $a(x) =2-x^2$ and $E_0 =E_c$ (which satisfies Assumption~\ref{a:nondeg}). 
We notice that, with this choice of $a$, we have 
$\left<\mathfrak{m}_\mu , a \right> = \Phi(\mu)$ where $\Phi$ is defined in~\eqref{e:def-Phi}, and satisfies $\Phi(E)=C(E)F(E)$ for $E$ near $E_c>1$.
As a consequence of  the fact that $\Phi(E_c)=\left<\mathfrak{m}_{E_c} , 2-x^2 \right> = 0$ (by definition of $E_c$), combined with Corollary~\ref{c:combine-tout}, we obtain at a critical point $\alpha_{j,c}$
\begin{align}
\label{e:tototo}
\lambda_j''(\alpha_{j,c}) & \to   -  \frac32 \Phi'(E_c) \mathcal E( E_c) \mathcal E'(E_c)^{-1} , \quad \text{ as } j \to +\infty .
\end{align}
Next, we notice that, by definition of $E_c$, we have $ \Phi'(E_c) =  C'(E_c)F(E_c) +  C(E_c)F'(E_c) = C(E_c)F'(E_c)$.
And according to Item~\ref{e:F-increasing} in Proposition~\ref{prop4.1}, we have $F'(E_c)<0$.
Using Assumption~\ref{a:nondeg}, and Equation~\eqref{eq:a5.1}, we have $\mathcal E '(E_0) = (\pi C(E_0))^{-1}>0$, so that~\eqref{e:tototo} rewrites 
\begin{align*}
\lambda_j''(\alpha_{j,c}) & \to   -  \frac{3\pi}2 F'(E_c) C(E_c)^2\mathcal E( E_c) >0 , \quad \text{ as } j \to +\infty .
\end{align*}
The expression in~\eqref{e:deriv-infty-1}--\eqref{e:deriv-infty-2} then follows from the definitions of $C$ and $\mathcal{E}$ in~\eqref{eq:def-CE} and ~\eqref{e:def-cal-E} respectively.

Note finally that this property also implies the uniqueness of the critical point of $\lambda_j$ for $j$ large enough (otherwise $\lambda_j$ would have a local maximum), and concludes the proof of Theorem~\ref{thcj}.

    \section{Analysis of the first two eigenvalues}
    \label{s:first-eig}
In this section, we give a complete mathematical proof that $\lambda_1$ has a unique critical point (improving \cite{He2008} which was always considering the minimum) and a numerically assisted proof that the same property holds for $\lambda_2$. Unfortunately we do not see 
 how to adapt the strategy to more excited eigenvalues.
\subsection{Strategy of the proof}
 In the case $j=1$, the proof is a combination of Lemma \ref{lemodd}  and  Proposition~\ref{lem:abound}. In the case  $j=2$, we cannot use Proposition~\ref{lem:abound} and we need some numerical help.
Of course it is less efficient that in the analysis of the De Gennes
 model\footnote{The De Gennes model corresponds to the family
 of Harmonic oscillators $D_t^2 + (t+\alpha)^2$ on the half line with
 Neumann condition with groundstate energy $\mu(\alpha)$. It has been
 shown in \cite{DauHel}  that $\lambda_1$ has  a
 unique minimum and that this minimum is non degenerate.} where we have equality in Proposition \ref{lem:abound}  but this will be enough.
 In the case $j=1$, we will start by using Proposition \ref{lem:abound} to deduce an upper bound for $\alpha_c$ and then we will use Lemma \ref{lemodd} to conclude.
The proof will be a consequence of a good upper bound for
$\lambda_1(\alpha)$
 and a good lower bound for $\lambda_3(\alpha)$.

\subsection{A preliminary estimate}
The following lemma has been proved in~\cite{HK2008a}.
\begin{lemma}\label{lemodd}
If $\alpha_c$ is a critical point of $\lambda_1$ (resp. $\lambda_2$) and if
\begin{equation}\label{hyp37}
3 \lambda_3(\alpha_c) > 7 \lambda_1 (\alpha_c) \quad \text{(resp. $3 \lambda_4(\alpha_c) > 7 \lambda_2 (\alpha_c)$)} ,
\end{equation}
then 
\begin{equation}\label{ndg}
 \lambda_1''(\alpha_c) >0 \quad \text{(resp. $\lambda_2''(\alpha_c) >0$)}.
\end{equation}
\end{lemma}
Let us recall the proof for the sake of completeness, which relies on the following general Lemma~\ref{l:estim-lambda-j-gap}.
Note that Lemma~\ref{lemodd} also holds for $\lambda_2$ instead of $\lambda_1$, as a direct consequence of Lemma~\ref{l:estim-lambda-j-gap}.
We will verify numerically that this condition is indeed satisfied for $\lambda_2$.

\begin{lemma}
\label{l:estim-lambda-j-gap}
For all $j \in \N$ and $\alpha_c$ a critical point of $\lambda_j$, we have 
\begin{align*}
\lambda_j''(\alpha_{c})
&  \geq \frac 23 \; \frac{3
 \lambda_{j+2}(\alpha_c) - 7\lambda_j(\alpha_c)}{\lambda_{j+2}(\alpha_c)-\lambda_j(\alpha_c)}  , \quad  \text{if }j \in \{1,2\} , \\
 \lambda_j''(\alpha_{c})
 & \geq 2 - \frac{8}{3} \lambda_j(\alpha_c) \max \left\{ \frac{1}{\lambda_{j}(\alpha_c)-\lambda_{j-2}(\alpha_c)} , \frac{1}{\lambda_{j+2}(\alpha_c)-\lambda_{j}(\alpha_c)} \right\}, \quad   \text{if }j \geq 3 .
 \end{align*}
\end{lemma}
\begin{remark}
Unfortunately, the inequalities obtained for $\lambda_j$ with $j \geq 3$ are less readable/useful. Also, they become inaccurate for large eigenvalues since the gap is asymptotically negligible with respect to the eigenvalue.
\end{remark}

\begin{proof}
We start from the expression of $\lambda_j''(\alpha)$ in~\eqref{eq:2.15}, in which the Cauchy-Schwarz inequality yields
\begin{align}
\label{eq:2.15-ineq}
\lambda_j''(\alpha) \geq  2  - 4 \left\| (\frac 12 t^{2} -\alpha)
u_j(\cdot,\alpha) \right\|_{L^2(\R)} \|\partial_\alpha u_j(\cdot,\alpha) \|_{L^2(\R)} ,
\end{align}
and the remaining part of the proof consists in estimating the right hand-side. To estimate $\|\partial_\alpha u_j(\cdot,\alpha) \|_{L^2(\R)}$, we recall the equation satisfied by $\partial_\alpha u_j(\cdot,\alpha)$ in~\eqref{e:Qlambda}.
We set 
$$
H_j (\alpha):= \{v \in L^2(\R),v \text{ even (resp. odd)}, v \perp u_j(\cdot, \alpha)\} \quad \text{if } j \text{ is odd (resp. even)},
$$ 
and notice that $\partial_\alpha u_j(\cdot,\alpha) \in H_j (\alpha)$ for all $\alpha \in \R, j \in \N^*$. Moreover, if $\Pi_{H_j (\alpha)}$ denotes the orthogonal projector onto $H_j(\alpha)$, we have $\mathfrak h_{\mathcal M}(\alpha) \Pi_{H_j (\alpha)} = \Pi_{H_j (\alpha)} \mathfrak h_{\mathcal M}(\alpha)$ and the operator
$$
P_j(\alpha) := \left( \mathfrak h_{\mathcal M}(\alpha) - \lambda_j(\alpha) \right) \Pi_{H_j (\alpha)} 
$$
is an isomorphism of $H_j (\alpha)$, since $\ker  \left( \mathfrak h_{\mathcal M}(\alpha) - \lambda_j(\alpha) \right)  = \vect u_j(\cdot, \alpha)$ (see Proposition~\ref{p:def-lambda}).
Now if $\alpha_c$ is a critical point of $\lambda_j$, then~\eqref{e:Qlambda} rewrites 
\begin{equation}\label{e:Qlambda-bis}
\left(\mathfrak h_{\mathcal M}(\alpha_c) -\lambda_j(\alpha_c)\right) \d_\alpha u_j(\cdot, \alpha_c) = 2 \left(\frac{t^{2}}{2}- \alpha_c \right)  u_j(\cdot,\alpha_c)\,.
\end{equation} 
Note also that, according to the first variation formula~\eqref{FeHei}, $\left(\frac{t^{2}}{2}- \alpha_c \right)  u_j(\cdot,\alpha_c) \in H_j(\alpha_c)$ as well. Hence~\eqref{e:Qlambda-bis} is an equation in $H_j(\alpha_c)$, namely
$$
P_j(\alpha)\d_\alpha u_j(\cdot, \alpha_c) = 2 \left(\frac{t^{2}}{2}- \alpha_c \right)  u_j(\cdot,\alpha_c),
$$ 
or equivalently
$$
\d_\alpha u_j(\cdot, \alpha_c) = 2 P_j(\alpha_c)^{-1} \left( \left(\frac{t^{2}}{2}- \alpha_c \right)  u_j(\cdot,\alpha_c) \right) .
$$ 
As a consequence, 
$$
\| \d_\alpha u_j(\cdot, \alpha_c)  \|_{L^2(\R)} \leq  2  \| P_j(\alpha_c)^{-1}\|_{\mathcal{L}(H_j(\alpha_c))}  \left\|  \left(\frac{t^{2}}{2}- \alpha_c \right)  u_j(\cdot,\alpha_c) \right\|_{L^2(\R)} .
$$ 
But $P_j(\alpha)$ is a selfadjoint operator on $H_j(\alpha_c)$ with spectrum $\{ \cdots , \lambda_{j-2}(\alpha)-\lambda_{j}(\alpha), \lambda_{j+2}(\alpha)-\lambda_{j}(\alpha), \lambda_{j+4}(\alpha)-\lambda_{j}(\alpha), \cdots\}$ if $j \geq 3$, and $\{ \lambda_{j+2}(\alpha)-\lambda_{j}(\alpha), \lambda_{j+4}(\alpha)-\lambda_{j}(\alpha), \cdots\}$ if $j \in \{1,2\}$.
As a consequence 
\begin{equation}
\label{e:estim-P-j}
\| P_j(\alpha_c)^{-1}\|_{\mathcal{L}(H_j(\alpha_c))}   =  
\left\{
\begin{array}{ll} 
\displaystyle \frac{1}{\lambda_{j+2}(\alpha)-\lambda_{j}(\alpha)} , &  \text{if }j \in \{1,2\} , \\
\displaystyle  \max \left\{ \frac{1}{\lambda_{j}(\alpha)-\lambda_{j-2}(\alpha)} , \frac{1}{\lambda_{j+2}(\alpha)-\lambda_{j}(\alpha)} \right\}, &  \text{if }j \geq 3 . 
 \end{array}
 \right.
\end{equation}
Coming back to~\eqref{eq:2.15-ineq}, we have obtained that for any $j \in \N$ and any critical point $\alpha_c$ of $\lambda_j$, we have
\begin{align*}
 \lambda_j''(\alpha_c) \geq  2  - 8\| P_j(\alpha_c)^{-1}\|_{\mathcal{L}(H_j(\alpha_c))} \left\| (\frac 12 t^{2} -\alpha_c)
u_j(\cdot,\alpha_c) \right\|_{L^2(\R)}^2 ,
\end{align*}
which, after having used~\eqref{eq:newid2}, rewrites
\begin{align*}
 \lambda_j''(\alpha_c) \geq  2  - \frac{8}{3}\| P_j(\alpha_c)^{-1}\|_{\mathcal{L}(H_j(\alpha_c))} \lambda_j(\alpha_c)
 \end{align*}
Recalling the expression of $\| P_j(\alpha_c)^{-1}\|_{\mathcal{L}(H_j(\alpha_c))}$ in~\eqref{e:estim-P-j} then concludes the proof of the lemma.
\end{proof}

\subsection{Upper bounds for $\lambda_1(\alpha)$}
In order to apply Lemma~\ref{lemodd}, we need to prove, at a critical point $\alpha_c$ of $\lambda_1$, that $\frac{\lambda_3(\alpha_c)}{\lambda_1(\alpha_c)} > \frac73$. To this aim, we need to have a rough localization of $\alpha_c$, and to provide with an upper bound for $\lambda_1$ lower bound for $\lambda_3$ near the place where $\alpha_c$ has been localized.
We start with the following lemma.
\begin{lemma}
\label{l:upper-lambda1}
Assume $\alpha_c$ is a critical point for $\lambda_1$. Then we have 
\begin{equation}\label{eq:alphac}
 \alpha_c  \in \left(0, \left(\frac{24}{25}\right)^{1/3}\right) \subset (0,1).
\end{equation}
Moreover, we have 
\begin{equation}\label{eq:lambda1a-bis}
\lambda_1(\alpha) \leq \frac{4}{5}\left(\frac{9}{5}\right)^{1/3}\approx 0.97315 , \quad \text{ for all } \alpha \in  \left(0, \left(\frac{24}{25}\right)^{1/3}\right) .
\end{equation}
\end{lemma}
\begin{proof}
We denote, for $v \in D(\mathfrak h_{\mathcal M}(0))$ the energy  
$$
\mathscr{E}(v, \alpha) := \left( \mathfrak h_{\mathcal M}(\alpha)v,v\right)_{L^2(\mathbb R)}  = \|v'\|_{L^2(\R)}^2 + \left\|\left( \frac12 t^2-\alpha\right)v\right\|_{L^2(\R)}^2 ,
$$
and recall that 
\begin{align}
\label{e:minimax}
\lambda_1(\alpha) = \mathscr{E}(u_1(\cdot,\alpha), \alpha) = \min \left\{ \mathscr{E}(v, \alpha) , v \in D(\mathfrak h_{\mathcal M}(0)), \|v\|_{L^2(\R)} = 1\right\} .
\end{align}
We compute\footnote{This idea was rather efficient for giving an upperbound
 of the analog problem for the De Gennes model.} the energy $\mathscr{E}(u_\rho, \alpha)$ of the $L^2$-normalized 
Gaussian
$$
u_\rho = c_\rho \exp \left(- \frac \rho 2 t^2 \right) , \quad c_\rho = \sqrt{\frac{2\pi}{\rho}} , \quad \rho >0 . 
$$
Using that 
$$
\|t u_\rho\|_{L^2(\R)}^2 = \frac{1}{2\rho} , \quad \|t^2 u_\rho\|_{L^2(\R)}^2 = \frac{3}{(2\rho)^2} ,
$$
 we obtain
$$
\mathscr{E}(u_\rho,\alpha) = \frac \rho 2 +  \frac{3}{16} \rho^{-2}
 - \frac{\alpha}{2\rho} + \alpha ^2\,.
$$
From~\eqref{e:minimax}, we deduce that, for any $\alpha \in \R$ and $\rho >0$, 
\begin{equation}\label{meilleur}
\lambda_1(\alpha) \leq \mathscr{E}(u_\rho,\alpha) =  \frac \rho 2 +  \frac{3}{16} \rho^{-2}
 - \frac{\alpha}{2\rho} + \alpha ^2\,.
\end{equation} 
We now write this inequality for particular values of $\rho$.

{\it First upper bound.} \\
 If we take the minimizing $\rho$ for $\alpha=0$, we obtain
$
\rho_0= (\frac 34)^{\frac 13}
$
and the corresponding energy is
$$
\mathscr{E}(u_{\rho_0},\alpha)
 = \frac 12 \big(\frac 34\big)^{\frac 13} +  \frac{3}{16} \big(\frac 34\big)^{-\frac
   23}
 + \alpha^2 - \frac 12 \alpha  \big(\frac 34\big)^{-\frac 13}\;.
$$
Hence we obtain, for this specific $\rho_0$, for all $\alpha \in \R$,
\begin{equation}\label{za}
\lambda_1(\alpha) \leq \mathscr{E}(u_{\rho_0},\alpha) = \alpha^2- 6^{-\frac 13} \alpha  +\left(\frac 34\right)^\frac 43  .
 \end{equation} 
As in \cite{He2008}, we use this upper bound for $\alpha>0$  small enough, namely, for $\alpha \in (0,0.7)$. 
 The minimum of the function $\alpha \mapsto \mathscr{E}(u_{\rho_0},\alpha)$ is attained at $\alpha = \frac{6^{-\frac 13}}{2} \approx 0.28 \in (0,0.7)$.
 Thus, on the interval $\alpha \in (0,\frac{7}{10})$, we have 
\begin{equation}\label{uplambda1a}
\lambda_1(\alpha) \leq \max \left( \mathscr{E}(u_{\rho_0},0) ,  \mathscr{E}\left(u_{\rho_0},\frac{7}{10}\right) \right) = \mathscr{E}\left(u_{\rho_0},\frac{7}{10}\right)  \approx 0.79\quad\mbox{for all } \alpha \in (0,0.7)\,.
\end{equation}
{\it  Second upper bound}~\\ Without being optimal, we can take $\rho= \frac{1}{2\alpha}>0$ in \eqref{meilleur} (corresponding to the cancellation of the last two terms) which leads to
\begin{equation}\label{uplambda1b}
\lambda_1(\alpha) \leq \mathscr{E}(u_{ \frac{1}{2\alpha}} ,\alpha) = \frac{1}{4\alpha}  +  \frac{3}{4} \alpha^2 \,,
\end{equation} 
for all $\alpha >0$.
Implementing this in \eqref{eq:abound} we deduce in particular that if $\alpha_c$ is a critical point of $\lambda_1$, then
$$
\alpha_c^2 <  \frac{1}{4\alpha_c}  +  \frac{3}{4} \alpha_c^2 \,.
$$
According to Corollary~\ref{c:decreasing-neg}, we have $\alpha_c >0$, so that this inequality is equivalent to $\alpha_c <1$.

We then note that $\alpha \mapsto \mathscr{E}(u_{ \frac{1}{2\alpha}} ,\alpha) =  \frac{1}{4\alpha}  +  \frac{3}{4} \alpha^2$ reaches its minimum on $\R_+^*$ at $\alpha = 6^{-1/3}\approx 0.55$. In particular, this function is increasing on $(0.7\,,\, 1)$ and thus
\begin{equation}\label{uplambda1c}
 \mathscr{E}(u_{ \frac{1}{2\alpha}} ,\alpha)  \leq  \mathscr{E}(u_{ \frac{1}{2}} ,1) 
 \leq 1, \quad \mbox{ for all } \alpha\in (0.7\,,\, 1)\,.
\end{equation}
Using \eqref{uplambda1a}, \eqref{uplambda1b} and \eqref{uplambda1c}, we also deduce that $\lambda_1(\alpha) \leq 1$ for all $\alpha \in (0,1)$.\\
{\it Third upper bound}.\\
We refine slightly the second upper bound by taking $\rho= \frac{3} {5 \alpha}$ in \eqref{meilleur}.  We obtain
\begin{equation*}
\lambda_1(\alpha) \leq\mathscr{E}(u_{ \frac{3}{5\alpha}} ,\alpha) =
 \frac{3}{10 \alpha} +  \frac{3\times 25}{16\times 9 } \alpha^2 
 - \frac{5}{6} \alpha^2  + \alpha ^2 =   \frac{3}{10 \alpha} -  \frac{ 5}{16 } \alpha^2  + \alpha ^2   =  \frac{3}{10\alpha} +  \frac{11}{16} \alpha^2   .
\end{equation*} 
Using again~\eqref{eq:abound} we deduce that if $\alpha_c$ is a critical point of $\lambda_1$, then
$$
\alpha_c^2 < \frac{3}{10\alpha_c} +  \frac{11}{16} \alpha_c^2  
$$
which, since $\alpha_c >0$ (according to Corollary~\ref{c:decreasing-neg}), is equivalent to the small improvement
$$
\alpha_c  < (24/25)^{1/3} \approx 0.98645 ,
$$
which proves~\eqref{eq:alphac}.\\
 Moreover, $\alpha \mapsto \mathscr{E}(u_{ \frac{3}{5\alpha}} ,\alpha)$ reaches its minimum on $\R_+^*$ at $\alpha =(12/55)^{1/3}\approx 0.60$. In particular, this function is increasing on $(0.7\,,\, 1)$ and thus
\begin{equation*} 
 \lambda_1(\alpha) \leq \mathscr{E}(u_{ \frac{3}{5\alpha}} ,\alpha)  \leq  \mathscr{E}(u_{ \frac{3}{5(24/25)^{1/3} }} ,(24/25)^{1/3} )  
=\frac{4}{5}\left(\frac{9}{5}\right)^{1/3}\approx 0.97315, 
\end{equation*}
 for all $\alpha\in (0.7, (24/25)^{1/3} )$.
Together with \eqref{uplambda1a} we finally deduce~\eqref{eq:lambda1a-bis}, which concludes the proof of the lemma.
\end{proof}

\subsection{Lower bounds for $\lambda_3 (\alpha)$}

\begin{lemma}
\label{l:lower-lambda3}
For all $\alpha \in \left(0 , \left(\frac{24}{25}\right)^{1/3}\right]$,
\begin{align}
 \label{y3b}
\lambda_3(\alpha) \geq  \sqrt{15} -  \frac{6}{5} \left(\frac{24}{25}\right)^{1/3}- \frac{9}{25} \approx 2.3292 .
\end{align}
\end{lemma}

\begin{proof}
We push the idea of comparing our operator with an harmonic
oscillator.
We first notice that, for any $\gamma \geq 0, \alpha \in \R, t \in \R$, we have
\begin{equation}\label{eq:lwb}
\gamma t^2 + \alpha^2 - (\alpha +\gamma)^2 \leq
\left(\frac{1}{2}t^{2}-\alpha\right)^2\;.
\end{equation}
Indeed, the polynomial $\frac{X^2}{4}-(\alpha+\gamma)X+(\alpha+\gamma)^2$ is nonnegative on $\R$ for all $\alpha, \gamma \in \R$, and thus $\gamma X-(\alpha+\gamma)^2 \leq \frac{X^2}{4}-\alpha X$, which implies~\eqref{eq:lwb} for $X=t^2$.
From this we can compare the operator $\mathfrak h_{\mathcal M}(\alpha)$ to the
 Harmonic Oscillator $$
 \mathfrak H_{\alpha, \gamma} := -\frac{d}{dt} + V_\gamma(t) , \quad 
 V_\gamma(t)= \gamma t^2 + \alpha^2 -
 (\alpha +\gamma)^2.
 $$
 According to the minimax principle (see e.g.~\cite[Lemma~B.6]{BBGL} or~\cite[Chapter~11 and discussion top of p148]{Helffer:book-spectral-theory}), denoting $\Sp( \mathfrak H_{\alpha, \gamma} ) := \{\tilde\lambda_j(\alpha,\gamma), j \in \N^*\}$, \eqref{eq:lwb} implies 
 $$
 \lambda_j(\alpha) \geq \tilde\lambda_j(\alpha,\gamma) , \quad \text{ for all } j \in \N^* . 
 $$
 Recalling that the eigenvalues of the harmonic oscillator $-\frac{d^2}{dt^2}+\gamma t^2$ are given by $\{(2k+1)\sqrt{\gamma}, k \in \N\}$, we deduce that, for any $\gamma>0$, 
\begin{equation} \label{y3}
\lambda_3(\alpha) \geq \tilde\lambda_3(\alpha,\gamma)  =  5 \gamma^\frac 12 - 2 \gamma \alpha -
\gamma^2\,.
\end{equation}
We finally choose a particular value of $\gamma$. Choosing $\gamma=3/5$, we deduce that 
\begin{equation*} 
\lambda_3(\alpha) \geq \tilde\lambda_3(\alpha,3/5)  = \sqrt{15} -  \frac{6}{5} \alpha - \frac{9}{25}, 
\end{equation*}
and in particular that for all $\alpha \in(0,1]$
\begin{equation*} 
\lambda_3(\alpha) \geq  \tilde\lambda_3(\alpha,3/5) \geq \tilde\lambda_3(1,3/5) = \sqrt{15} -  \frac{6}{5} - \frac{9}{25} \approx 2.313 .
\end{equation*}
More precisely,  for all $\alpha \in \left(0 , \left(\frac{24}{25}\right)^{1/3}\right]$,
\begin{equation*}
\lambda_3(\alpha) \geq  \tilde\lambda_3(\alpha,3/5) \geq \tilde\lambda_3\left(\left(\frac{24}{25}\right)^{1/3},3/5\right) = \sqrt{15} -  \frac{6}{5} \left(\frac{24}{25}\right)^{1/3}- \frac{9}{25} \approx 2.3292 ,
\end{equation*}
which concludes the proof of the lemma.
\end{proof}

\subsection{End of the proof of Theorem~\ref{t:lambda1}}
From Lemmas~\ref{lemodd},~\ref{l:upper-lambda1} and~\ref{l:lower-lambda3}, we may now conclude the proof of Theorem~\ref{t:lambda1}.
\begin{proof}[End of the proof of Theorem~\ref{t:lambda1}]
Let $\alpha_c$ be a critical point of $\alpha \mapsto \lambda_1(\alpha)$. 
As a consequence of Lemmas~\ref{l:upper-lambda1} and~\ref{l:lower-lambda3}, namely Equations~\eqref{eq:alphac},~\eqref{eq:lambda1a-bis} and~\eqref{y3b}, we have necessarily 
$$
\lambda_1(\alpha_c) \leq \frac{4}{5}\left(\frac{9}{5}\right)^{1/3}\approx 0.97315 , \quad  \lambda_3(\alpha_c) \geq  \sqrt{15} -  \frac{6}{5} \left(\frac{24}{25}\right)^{1/3}- \frac{9}{25} \approx 2.3292 .
$$
As a consequence, we have 
\begin{equation}
\label{e:quotient-31}
\frac{\lambda_3(\alpha_c)}{\lambda_1(\alpha_c)} \geq  \frac{5}{4}\left(\frac{9}{5}\right)^{-1/3}\left(\sqrt{15} -  \frac{6}{5} \left(\frac{24}{25}\right)^{1/3}- \frac{9}{25} \right) \approx 2.3935
 > \frac73 \approx 2.3333 .
\end{equation}
Using Lemma~\ref{lemodd}, we deduce that $\lambda_1''(\alpha_c) >0$ for any critical point $\alpha_c$. 
According to~\eqref{e:infty-a-infty}, $\lambda_1(\alpha)\to+\infty$ as $\alpha \to +\infty$. These two properties imply uniqueness (and nondegeneracy) of the critical point $\alpha_c$, which is necessarily a minimum, and concludes the proof of the lemma.
 \end{proof}

\subsection{Numerics}
\label{s:numerics}
In the present subsection, we collect numerical data and figures communicated to us by Mikael Persson Sundqvist. The latter are obtained (using the  function NDEigenvalues of mathematica) by computing numerically the first eigenvalues for the Dirichlet problem on the bounded interval $[-20,20]$.

\medskip 
On Figure~\ref{f:lambda-3-1} is represented the graph of the quotient $\alpha \mapsto \lambda_3(\alpha)/\lambda_1(\alpha)$ for $\alpha \in (0,1)$. 
This illustrates the proof of Theorem~\ref{t:lambda1} relying on Lemmas~\ref{lemodd},~\ref{l:upper-lambda1} and~\ref{l:lower-lambda3}, in which we prove that 
$\frac{\lambda_3(\alpha)}{\lambda_1(\alpha)} \geq  2.39$ for $\alpha \in (0, \left(\frac{24}{25}\right)^{1/3})$.
Here Figure~\ref{f:lambda-3-1} suggests that $\frac{\lambda_3(\alpha)}{\lambda_1(\alpha)} \geq  4$ for all $\alpha \in (0,1)$.
Matematica actually gives $\frac{\lambda_3(\alpha)}{\lambda_1(\alpha)} \geq  4.07529$ (with the acccuracy of  the calculator) with a great stability when exchanging from Dirichlet to Neumann boundary conditions at the endpoints of the bounded interval $[-20,20]$\footnote{Due to Agmon estimates, it is rather clear that  the error is very small.}.
Note that in the limit $\alpha \rightarrow + \infty$, we have  $\lambda_3(\alpha)/\lambda_1(\alpha) \to  3$ (as a consequence of~\eqref{eq:asba}). 
In particular the condition $\lambda_3(\alpha)/\lambda_1(\alpha) >7/3$ holds for $\alpha$ large. This gives another way to see that there are no critical point for large $\alpha$, although this does not furnish any explicitly computable bound.\begin{figure}[h]
\begin{center}
\includegraphics[width=20pc, height=10pc]{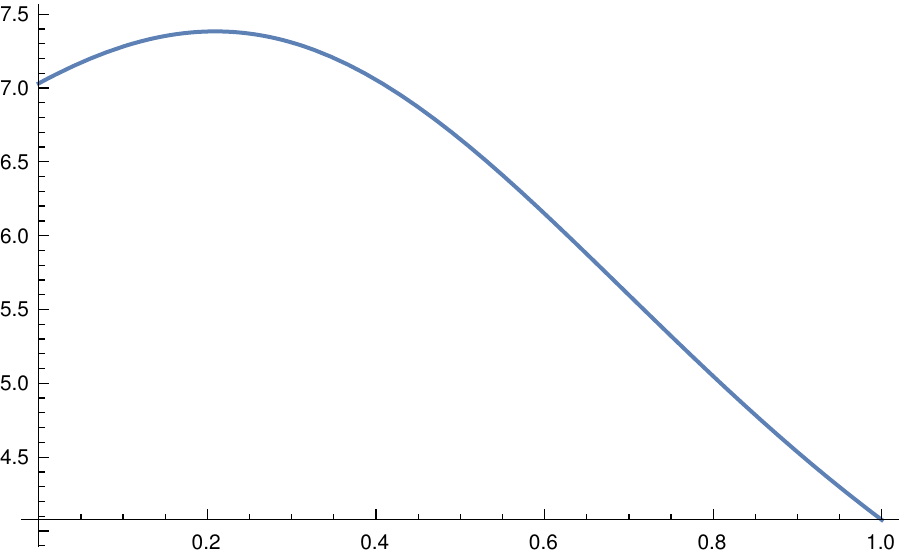}
\end{center}
\caption{ Graph of the quotient $\lambda_3/\lambda_1$ over $(0,1)$} \label{f:lambda-3-1}
\end{figure}

\medskip 
On Figure~\ref{f:graphe-6-vp}, the graph of the first six eigenvalues $\alpha \mapsto \lambda_j(\alpha)$ is represented on the interval $[-2.5,6.5]$, which contains the change of variation of these six functions.
This figure shows a unique critical point, which looks nondegenerate, and therefore contributes to motivate Conjecture~\ref{conj1}.
Note that the asymptotic behavior as $\alpha \to \pm \infty$ is also seen on the figure: we have and $\lambda_{2k+1}(\alpha) \sim \lambda_{2k+2} (\alpha) \sim  \sqrt{2}( 2k+1)\alpha^{1/2}$ as $\alpha \to +\infty$ see~\eqref{eq:asba}--\eqref{eq:asbb}, and a single-well analysis~\cite{HeSj,He88} would yield a precise asymptotic expansion as $\alpha \to -\infty$ as well.
Note that, as  $\alpha \rightarrow +\infty$, Figure~\ref{f:graphe-6-vp} exhibits pairs of asymptotically exponentially close curves, which is predicted by the double well analysis~\eqref{eq:asbb}.
On table~\ref{t:tableau-numer}, we can see the numerical value of the critical point $\alpha_{c,j}$ of $\lambda_j$. In particular, we see that  $\alpha_{c,1} \approx 0.35$, which is consistent with Lemma~\ref{l:upper-lambda1} which proves that $\alpha_{c,1} <\left(\frac{24}{25}\right)^{1/3}$.
Also on table~\ref{t:tableau-numer}, we see the value of the quotient $\lambda_j (\alpha_{c,j})/\alpha_{c,j}^2$ which converges very rapidly towards $E_c \approx 2.35$, as predicted by Theorem~\ref{th1.4} .

\begin{table}
\begin{center}
\begin{tabular}{ | l | c | c | c | c | c | c | }
 \hline			
   Value of $j$& $j=1$ & $j=2$& $j=3$ & $j=4$& $j=5$ & $j=6$ \\
   \hline
   Value of the critical point $\alpha_{c,j}$ of $\lambda_j$ & $0.35$ &  $1.13$ & $1.14$ & $1.55$ & $1.78$ & $2.06$ \\
   \hline
  Value of the quotient  $\lambda_j (\alpha_{c,j})/\alpha_{c,j}^2$ &   $4.78$ & $1.27$ & $2.69$ & $2.25$ & $2.41$ & $2.34$ \\
 \hline  
 \end{tabular}
 \end{center}
\caption{Values of the critical point $\alpha_{c,j}$ and the quotient  $\lambda_j (\alpha_{c,j})/\alpha_{c,j}^2$ for $j=1,\cdots, 6 $.}
 \label{t:tableau-numer}
\end{table}

\subsection{``Proof'' of Statement \ref{St02}}
We here provide with a completely numerical proof. Hence our statement can not be considered as a mathematical theorem. We could have proved rigorously some reduction to an explicitly finite interval of values for $\alpha$, hence excluding 
 the presence of more than one minimum outside this interval. But at the moment, the remaining interval is still very large ($[0, 45]$)  and 
  this does not seem sufficiently decisive. 
Note  (see \eqref{eq:asba} and \eqref{eq:asbb}) that 
$$
\lambda_4(\alpha)/\lambda_2(\alpha) \to  3$$
 as $\alpha \rightarrow +\infty$\,.
 The values of the quotient $\alpha \mapsto \lambda_4(\alpha)/\lambda_2(\alpha)$ are plotted on Figures~\ref{f:lambda-42-1}-\ref{f:lambda-42-2}. 
Figures~\ref{f:lambda-42-1}-\ref{f:lambda-42-2} numerically show that $\min\{\lambda_4(\alpha)/\lambda_2(\alpha), \alpha \in \R\} \approx 2.82$ and thus $\lambda_4/\lambda_2 \geq 7/3$ on $\R$. According to Lemma~\ref{lemodd}, this (numerically) proves that $\lambda_2''(\alpha_c)>0$, and consequently shows uniqueness of a critical point for $\lambda_2$, hence proving Statement~\ref{St02}.

\begin{figure}[h]
\begin{center}
\includegraphics[width=15pc, height=10pc]{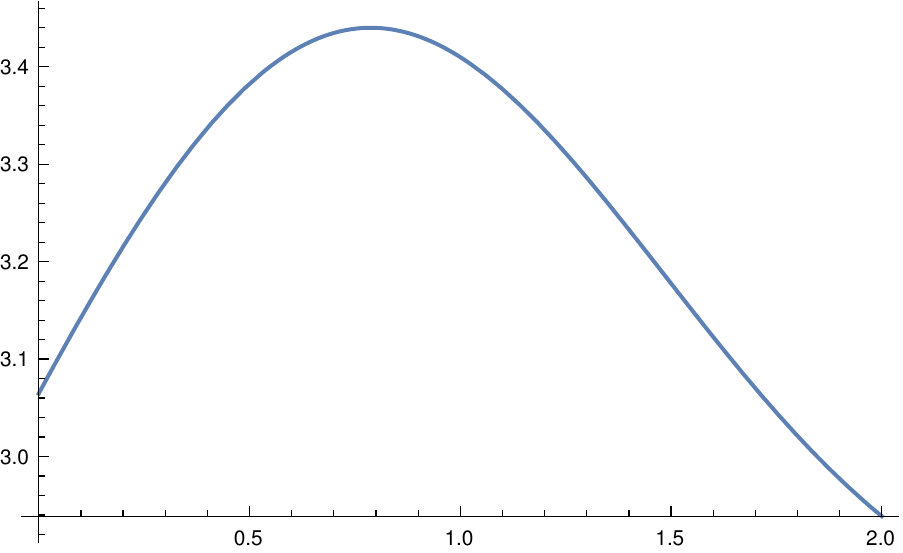}
\end{center}
\caption{Graph of $\lambda_4/\lambda_2$  for $\alpha \in (0,2)$} \label{f:lambda-42-1}
\end{figure}

\begin{figure}[h]
\begin{center}
\includegraphics[width=15pc, height=10pc]{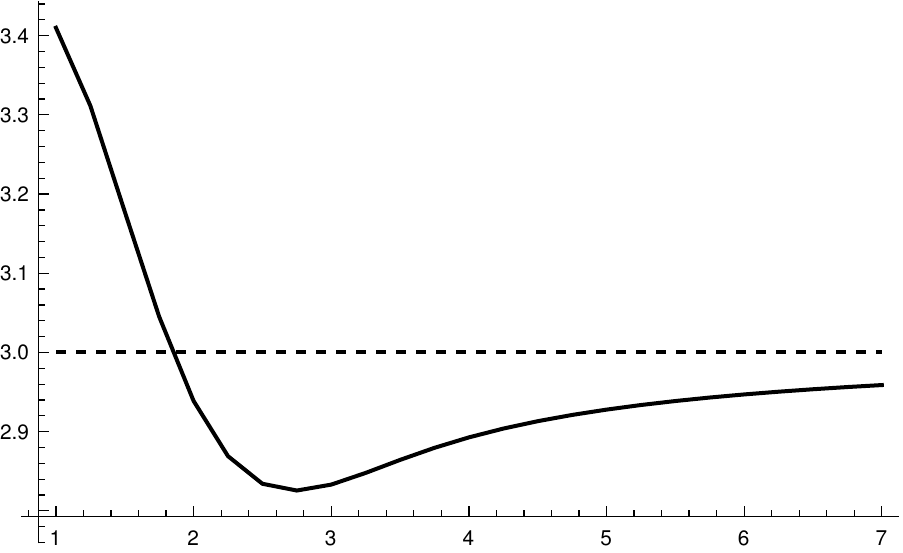}
\end{center}
\caption{Graph of $\lambda_4/\lambda_2$  for $\alpha>1$} \label{f:lambda-42-2}
\end{figure}

\appendix


\begin{thebibliography}{00}

\bibitem{AlmHel} Y.~Almog and  B. Helffer.
  \newblock  On the spectrum of some  Bloch-Torrey vector operators.
\newblock  Pure Appl. Anal. 4 (1), 1--48 (2022). 

 \bibitem{AHP2} Y. Almog, B. Helffer and X. B. Pan,
 \newblock  Superconductivity near the normal state in a
   half-plane under the action of a perpendicular electric current and
   an induced 
 magnetic field.
 \newblock Trans. AMS 365, 1183--1217 (2013).
 
 \bibitem{BBGL} H. Bahouri, D. Barilari, I. Gallagher and M. L\'eautaud.
 \newblock Spectral  summability for  the quartic  oscillator with applications to  the Engel group.
 \newblock  preprint  arxiv.org/abs/2206.10396 (2022).
 
 \bibitem{BPU} R. Brummelhuis, T. Paul and A. Uribe.
 \newblock  Spectral estimates around a critical level.
 \newblock Duke Math. Journal, Vol. 78, No. 3. 477--530   (1995).
 
 \bibitem{CN} B. Candelpergher and J.C. Nosmas.
 \newblock Propri\'et\'es spectrales d'op\'erateurs diff\'erentiels asymptotiques autoadjoints.
 \newblock CPDE 9 (2), 137--167 (1984).
 
 \bibitem{Chatzakou} M. Chatzakou, Quantizations on the Engel and the Cartan groups, {\it arXiv:2003.10744}. 

 \bibitem{CdV} Y. Colin de Verdi\`ere.
 \newblock Bohr-Sommerfeld rules to all orders.
 \newblock Ann. Henri Poincar\'e 925--936 (2005).
 
\bibitem{CdVLe} Y. Colin de Verdi\`ere and C. Letrouit.
\newblock Propagation of well-prepared states along Martinet singular geodesics.
\newblock to appear in Journal of Spectral Theory (2022).

\bibitem{CdVPa} Y. Colin de Verdi\`ere and B. Parisse.
\newblock Equilibre instable en r\'egime semi-classique I. Concentration microlocale.
\newblock Comm. in PDE Vol. 19: 1535-1563 (1994) .


\bibitem{DauHel} M. Dauge and  B.~Helffer.
\newblock Eigenvalues variation I, Neumann problem for Sturm-Liouville
 operators.
\newblock J. Differential Equations 104 (2), 243--262 (1993).

\bibitem{DiSj} M. Dimassi and J. Sj\"ostrand.
\newblock Spectral asymptotics in the semi-classical limit.
\newblock London Math. Soc. Lecture Note Series, Vol. 268. Cambridge University Press (1999).

\bibitem{Du} J. Duistermaat.
\newblock Oscillatory integrals, Lagrange Immersions and unfolding of singularities.
\newblock Comm. Pure Appl. Math. Vol.XXVII, 207--281 (1974).

\bibitem{Dy} S. Dyatlov.
\newblock Semiclassical Lagrangian Distributions, Lecture notes.\\
\newblock  https://math.mit.edu/$\sim$dyatlov/files/2012/hlagrangians.pdf

\bibitem{FH} S.~Fournais and  B. Helffer.
\newblock Spectral methods in surface superconductivity.
\newblock  Progress in Nonlinear Differential Equations and their Applications 77. Birkh\"auser Boston, Inc., Boston, MA (2010). 

\bibitem{GS:94}
A. Grigis and J. Sj{\"o}strand.
\newblock {\em Microlocal {A}nalysis for {D}ifferential {O}perators}.
\newblock Cambridge University Press, Cambridge (1994).

\bibitem{He79} B. Helffer.
\newblock Hypoellipticit\'e analytique sur des groupes nilpotents de rang 2 (d'apr\`es G. M\'etivier).
\newblock S\'eminaire EDP de l'\'Ecole Polytechnique 1979-1980, Exp. 1,  1-13.

\bibitem{He80} B. Helffer.
\newblock Conditions n\'ecessaires d'hypoanalyticit\'e pour des op\'erateurs invariants \`a  gauche sur un groupe nilpotent gradu\'e,
\newblock Journal of differential Equations, Vol.44, No 3, 460--481 (1982). 

\bibitem{He88}  B. Helffer.
\newblock  Introduction to semi-classical methods
for the Schr\"odinger operator.
\newblock Lecture Notes in mathematics,   Vol. 1336 (1988).

\bibitem{He2008} B. Helffer.
\newblock The Montgomery's  model revisited.
\newblock Colloquium Mathematicum, Vol. 118, No 2, 391--400 (2011). 

\bibitem{Qmath10} B. Helffer and Yu. A. Kordyukov.
\newblock  Spectral gaps for
periodic Schr\"odinger operators with hypersurface magnetic wells,
\newblock In ``Mathematical results in quantum mechanics'', Proceedings of the
QMath10 Conference Moieciu, Romania 10 - 15 September 2007, World
Sci. Publ., Singapore (2008).

\bibitem{HK2008a} B. Helffer and Yu. A. Kordyukov.
\newblock
Spectral gaps for
periodic Schr\"odinger operators with hypersurface magnetic wells:
Analysis near the bottom.
\newblock  J. Funct. Anal. 257, no. 10, 3043--3081 (2009).

\bibitem{HK2008b} B. Helffer and  Yu. A. Kordyukov.
\newblock Semiclassical analysis of Schr\"odinger operators with
magnetic wells.
\newblock  Spectral and scattering theory for quantum magnetic systems, 105--121, Contemp. Math. 500, Amer. Math. Soc., Providence, RI (2009).

\bibitem{HMR} B. Helffer, A. Martinez and D. Robert.
\newblock Ergodicit\'e et limite semi-classique.
\newblock Commun. Math. Phys. 109, 313--326 (1987).

\bibitem{HM} B. Helffer and A. Mohamed.
\newblock  Semiclassical analysis
for the ground state energy of a Schr\"odinger operator with
magnetic wells.
\newblock J. Funct. Anal. 138, 40--81 (1996).


\bibitem{HMdim3} B. Helffer and A. Morame.
\newblock Magnetic bottles for the Neumann problem: curvature effect in the case of dimensions $3$ (general case).
\newblock Annales de l'ENS  37, p.~105-170  (2004).

\bibitem{HNdim3} B. Helffer and J. Nourrigat.
\newblock Hypoellipticit\'e pour des groupes nilpotents de rang 3.
\newblock Commun. in P.D.E, 3 (8), 643--743 (1978). 

\bibitem{HPS} B. Helffer and M. Persson Sundqvist.
\newblock Spectral properties of higher order anharmonic oscillators.
\newblock  Journal of Mathematical Sciences, New York Springer Vol. 165, No. 1, 110--126 (2010). 

\bibitem{HeRo} B. Helffer and D. Robert.
\newblock Puits de potentiel g\'en\'eralis\'es et asymptotique semi-classique.
\newblock Annales de l'IHP (section Physique th\'eorique), Vol. 41, No 3, 291--331 (1984). 

\bibitem{Helffer:book-spectral-theory} B. Helffer, 
 \newblock Spectral theory and its applications.
 \newblock Cambridge Studies in Advanced Mathematics,    Cambridge University Press, Cambridge (2013).

\bibitem{HeSj}  B. Helffer and J. Sj\"ostrand. 
\newblock Multiple wells in the semiclassical limit. I. 
\newblock Comm. Partial Differential Equations 9, no. 4, 337--408 (1984). 

\bibitem{Hor-4} L. H{\"o}rmander.
 \newblock  The analysis of linear partial differential operators. {IV},
\newblock   Grundlehren der Mathematischen Wissenschaften, 275, Fourier integral operators,
              Corrected reprint of the 1985 original, Springer-Verlag, Berlin (1994).

\bibitem{LL:22} C. Laurent and M. L\'eautaud.
\newblock  Uniform observation of semiclassical Schr\"odinger eigenfunctions on an interval.
\newblock  to appear in Tunisian Journal of Math. arxiv.org/abs/2203.03271 (2022).

\bibitem{FeMa} V.P. Maslov and M.V. Fedoriuk.
\newblock Semi-classical approximation in quantum mechanics.
\newblock Reidel Publishing Company (1981).

\bibitem{Mont} R. Montgomery.
\newblock  Hearing the zero locus of a magnetic field. 
\newblock Comm. Math. Phys. 168, 651--675  (1995).

\bibitem{Olver:book}
F.~W.~J. Olver.
\newblock {\em Asymptotics and special functions}.
\newblock Computer Science and Applied Mathematics. Academic Press [Harcourt
  Brace Jovanovich, Publishers], New York-London (1974).

\bibitem{Pan} X.B. Pan.
\newblock Surface superconductivity in $3$ dimensions.
\newblock Trans. AMS. 356 (10), 3899--3937 (2004).

\bibitem{Pan-Kwek} X-B. Pan and  K-H. Kwek.
\newblock 
Schr\"odinger operators with non-degenerately vanishing magnetic
fields in bounded domains.
\newblock Trans. Amer. Math. Soc. 354, 4201--4227 (2002).

\bibitem{Pham-Robert} Pham The Lai and D. Robert,
Sur un probl\`eme aux valeurs propres non lin\'{e}aire. 
{\it Journ\'{e}es: \'{E}quations aux {D}\'{e}riv\'{e}es {P}artielles ({S}aint-{C}ast, 1979), \'{E}cole Polytech., Palaiseau}, {
Exp. No. 15, 5} (1979).

\bibitem{ReSi} M. Reed and B. Simon.
\newblock Methods of modern mathematical physics, IV Analysis of operators, Academic Press (1978).

\bibitem{Sim} B. Simon.
\newblock Coupling constant analyticity for the anharmonic oscillator.
\newblock Annals of Physics : 58, 76--136 (1970).

\bibitem{Zwo} M. Zworski. 
\newblock Semiclassical analysis,
\newblock volume 138 of Graduate Studies in Mathematics. American Mathematical Society, Providence, RI (2012).

\end{thebibliography}
\end{document}